\documentclass[11pt,reqno]{amsart}
\usepackage{amssymb}
\usepackage{amsmath}

\usepackage{changes}
\usepackage[mathscr]{euscript}
\usepackage[small]{caption}
\usepackage{mathtools}
\usepackage{cases}
\usepackage{graphicx}
\graphicspath{ {images/} }
\usepackage{xcolor}
\usepackage{comment}
\usepackage[utf8]{inputenc}
\usepackage[T1]{fontenc}
\usepackage[colorlinks=true,linkcolor=blue,citecolor=magenta]{hyperref}
\usepackage{a4wide}
\usepackage{ulem}

\makeatletter
\@addtoreset{equation}{section}
\makeatother

\newtheorem{theorem}{Theorem}[section]
\newtheorem{lemma}[theorem]{Lemma}
\newtheorem{proposition}[theorem]{Proposition}
\newtheorem{corollary}[theorem]{Corollary}

\newtheorem{definition}[theorem]{Definition}
\newtheorem{remark}[theorem]{Remark}

\newcounter{as}[section]

\newcommand{\mc}[1]{{\mathcal #1}}
\newcommand{\mf}[1]{{\mathfrak #1}}
\newcommand{\mb}[1]{{\mathbf #1}}
\newcommand{\bb}[1]{{\mathbb #1}}
\newcommand{\bs}[1]{{\boldsymbol #1}}
\newcommand{\ms}[1]{{\mathscr #1}}

\newcommand{\<}{\langle}
\renewcommand{\>}{\rangle}

\renewcommand{\epsilon}{\varepsilon}

\renewcommand{\epsilon}{\varepsilon}

\newcommand{\pfrac}[2]{\genfrac{}{}{}{1}{#1}{#2}}

\definecolor{bblue}{rgb}{.2,0.2,.8}

\let\oldtocsection=\tocsection
\let\oldtocsubsection=\tocsubsection
\let\oldtocsubsubsection=\tocsubsubsection
\renewcommand{\tocsection}[2]{\hspace{0em}\oldtocsection{#1}{#2}}
\renewcommand{\tocsubsection}[2]{\hspace{1em}\oldtocsubsection{#1}{#2}}
\renewcommand{\tocsubsubsection}[2]{\hspace{2em}\oldtocsubsubsection{#1}{#2}}
\DeclareRobustCommand{\SkipTocEntry}[5]{}

\begin{document}

\title[Large deviations for the exclusion with Robin boundary
conditions]{Dynamical large deviations for the boundary driven
symmetric exclusion process with Robin boundary conditions.}

\begin{abstract}
In this article, we consider a one-dimensional symmetric exclusion
process in weak contact with reservoirs at the boundary. In the
diffusive time-scaling the empirical measure evolves according to the
heat equation with Robin boundary conditions. We prove the associated
dynamical large deviations principle.
\end{abstract}

\allowdisplaybreaks

\author{T. Franco}
\address{UFBA\\
Instituto de Matem\'atica, Campus de Ondina, Av. Adhemar de Barros, S/N. CEP 40170-110\\
Salvador, Brasil} \email{tertu@ufba.br}

\author{P. Gon\c{c}alves} \address{\noindent Center for Mathematical
Analysis, Geometry and Dynamical Systems, Instituto Superior
T\'ecnico, Universidade de Lisboa, Av. Rovisco Pais, 1049-001 Lisboa,
Portugal} \email{p.goncalves@tecnico.ulisboa.pt}

\author{C. Landim} \address{IMPA, Estrada Dona Castorina 110, CEP
22460 Rio de Janeiro, Brasil
and CNRS UMR 6085, Université de Rouen, France. \\
e-mail: \texttt{landim@impa.br} }

\author{A. Neumann} \address{UFRGS, Instituto de Matem\'atica, Campus
do Vale, Av. Bento Gon\c calves, 9500. CEP 91509-900, Porto Alegre,
Brasil} \email{aneumann@mat.ufrgs.br}

\maketitle

\tableofcontents

\section{Introduction}
\label{sec1}

The investigation of the thermodynamic properties of stationary
nonequilibrium states of interacting particle systems has been proven
to be an important step in the understanding of nonequilibrium
phenomena and a rich source of mathematical problems \cite{D2007,
BDGJL2015, j20}.

In the context of lattices gases, the empirical measure is the only
relevant thermodynamical quantity in the macroscopic description of
the system, and the thermodynamical functionals, as the free energy,
can be identified to the large deviations rate functional.

While in equilibrium the stationary state is given by the Gibbs
distribution associated to the Hamiltonian, in nonequilibrium the
construction of the stationary state requires solving a
dynamical-variational problem which defines the so-called
quasi-potential \cite{FW98}.

At the beginning of this century, Derrida, Lebowitz and Speer
\cite{DLS2003} considered the one-dimensional symmetric exclusion
process in strong contact with reservoirs. In this context, the
empirical measure evolves in the diffusive time-scale according to the
heat equation with Dirichlet boundary conditions \cite{KL}. Expressing
the stationary state as a product of matrices \cite{DEHP93}, they
obtained an explicit formula for the large deviations principle rate
functional of the empirical measure under the stationary state, the
so-called nonequilibrium free energy.

Later, \cite{BDGJL} derived the Derrida-Lebowitz-Speer formula (in
short DLS formula) for the nonequilibrium free energy extending to
infinite dimensions the dynamical approach introduced in \cite{FW98}.
More precisely, they first proved a dynamical large deviations
principle for the empirical measure for symmetric exclusion processes
in strong contact with reservoirs \cite{BDGJL2003}.  Denote by
$I_{[0,T]}(u)$ the large deviations rate function of the dynamical
large deviations principle. Hence, $I_{[0,T]}(u)$ represents the cost
of observing a trajectory $u(t)$ in the time-interval $[0,T]$. Let
$\bar\rho$ be the stationary profile of the hydrodynamic equation,
that is, the typical density profile under the stationary
state. Define the quasi-potential as
\begin{equation*}
V(\gamma) \;=\; \inf_{T>0} \inf_u I_{[0,T]} (u)\;,
\end{equation*}
where the second infimum is carried over all trajectories $u(t)$
connecting the stationary density profile $\bar\rho$ to a density
profile $\gamma$ in the time interval $[0,T]$: $u(0) = \bar\rho$,
$u(T) = \gamma$. It is proven in \cite{BDGJL} that the quasi-potential
$V$ coincides with the nonequilibrium free energy, i.e., that it
satisfies the DLS' equations.

In the sequel, \cite{BG04, F2009} derived a large deviations principle
for the empirical measure under the stationary state from the
dynamical one, with rate functional given by the quasi-potential.
This result was later extended to weakly symmetric exclusion processes
in strong contact with reservoirs \cite{ED2004, BLM09, BGL2009,
BDGJL2011} and to reaction-diffusion models \cite{LT2018, FLT2018}.

It has long been understood that these results extend to
boundary-driven one-dimensional symmetric exclusion processes in weak
contact with reservoirs \cite{D}. But only recently, this result
appeared in \cite{DHS2021} through the matrix ansatz product method.

In this article, we accomplish the first step in the project of
deriving the large deviations principle for the empirical measure
under the stationary state, through the dynamical approach, for
boundary-driven one-dimensional symmetric exclusion processes in weak
contact with reservoirs. The law of large numbers has been obtained in
\cite{BMNS2017, FGN_LPP}. We prove here the dynamical large
deviations, while in the companion paper \cite{BEL21}, it is shown
that the quasi-potential satisfies the DLS' equations obtained in
\cite{DHS2021} for this model.

\section{Notation and Results}
\label{sec2}

\subsection*{The model}

We consider one-dimensional, symmetric exclusion processes in a weak
contact with boundary reservoirs. Fix $N\ge 1$, and let
$\color{bblue} \mf e_N = 1/N$, $\color{bblue} \mf r_N = 1 - (1/N)$,
$\color{bblue} \Lambda_N=\{\mf e_N ,\dots, (N-2)\, \mf e_N, \mf
r_N\}$.  The state-space is represented by
$\color{bblue} \Omega_N=\{0,1\}^{\Lambda_N}$ and the configurations by
the Greek letters $\eta$, $\xi$ so that $\color{bblue} \eta_x$,
$x\in \Lambda_N$, represents the number of particles at site $x$ for
the configuration $\eta$. Here and below all notation introduced in
the text and not in displayed equations is indicated in blue.

Fix throughout this article, $\color{bblue} 0<\alpha \le \beta<1$,
$\color{bblue} A>0$, $\color{bblue} B>0$.  The generator of the Markov
process, represented by $\ms L_N = \ms L^{\alpha, A, \beta, B}_N$, is
given by
\begin{equation*}
\ms L_N \;=\;  L^{\rm lb}_{N} \;+\;  L^{\rm bulk}_{N}
\;+\;  L^{\rm rb}_{N}\;.
\end{equation*}
In this formula, for every function $f:\Omega_N\to \mathbb R$,
\begin{equation}
\label{eq:defgenex}
(L^{\rm bulk}_{N}  f) (\eta) \;=\; N^2\,
\sum_{x\in \Lambda^o_N} [\, f(\sigma^{x,x+\mf e_N} \eta) \,-\, f(\eta)\,] \;,
\end{equation}
where $\Lambda^o_N$ represents the interior of $\Lambda_N$,
$\color{bblue} \Lambda^o_N := \Lambda_N \setminus \{\mf r_N\} = \{\mf
e_N ,\dots, (N-2)\, \mf e_N\}$, and
\begin{equation}
\label{eq:defgeneb}
\begin{gathered}
(L^{\rm lb}_N f) (\eta) \;=\; \frac{N}{A}\,
\big[\, (1-\eta_{\mf e_N})\, \alpha \,+\, (1-\alpha)\, \eta_{\mf e_N}\, \big]
\big[\, f(\sigma^{\mf e_N} \eta) \,-\, f(\eta)\, \big] \;, \\
(L^{\rm rb}_N f) (\eta) \;=\; \frac{N}{B}\,
\big[\, (1-\eta_{\mf r_N})\, \beta \,+\, (1-\beta)\, \eta_{\mf r_N}\, \big]
\big[\, f(\sigma^{\mf r_N} \eta) \,-\, f(\eta)\, \big]
\;.
\end{gathered}
\end{equation}
From now on, we omit the subindex $N$ from $\mf e_N$ and $\mf r_N$.  In the
definitions above,
\begin{equation}\label{eq:def_exchanges}
(\sigma^{x,x+\mf e}\eta)_y \;=\;
\begin{cases}
\eta_y &\mbox{ if } y \neq x,x+\mf e\\
\eta_{x+\mf e} &\mbox{ if } y=x \\
\eta_x &\mbox{ if } y=x+\mf e
\end{cases}
\quad \mbox{ and }
\quad
(\sigma^{x}\eta)_y \;=\;
\begin{cases}
\eta_y &\mbox{ if } y\neq x\\
1-\eta_x &\mbox{ if } y=x \;.
\end{cases}
\end{equation}

For a metric space $\mathbb X$, denote by
$\color{bblue} D([0,T], \mathbb X)$, $T>0$, the space of
right-continuous functions $x\colon [0,T] \to \mathbb X$, with
left-limits, endowed with the Skorohod topology and its associated
Borel $\sigma$-algebra. The elements of $D([0,T],\Omega_N)$ are
represent by $\bs \eta (\cdot)$.

For a probability measure $\mu$ on $\Omega_N$, let
$\color{bblue} \mathbb P_\mu$ be the measure on $D([0,T],\Omega_N)$
induced by the continuous-time Markov process associated to the
generator $\ms L_N$ starting from $\mu$. When the measure $\mu$ is the
Dirac measure concentrated on a configuration $\eta\in \Omega_N$, that
is $\mu = \delta_{\eta}$, we represent
$\color{bblue}\mathbb P_{\delta_\eta}$ simply by
$\color{bblue}\mathbb P_\eta$.  Expectation with respect to
$\mathbb P_\mu$, $\mathbb P_\eta$ is denoted by
$\color{bblue}\mathbb E_\mu$, $\color{bblue}\mathbb E_\eta$,
respectively.

\subsection*{Hydrodynamic limit}  Denote by
$\color{bblue}\mathscr{M}$ the set of non-negative measures on
$[0,1]$ with total mass bounded by $1$ endowed with the weak
topology. Recall that this topology is metrisable and that, with this
topology, $\ms M$ is a  compact space.  For a continuous
function $F:[0,1] \to \mathbb R$ and a measure $\pi\in\ms M$, denote by
$\<\pi, F\>$ the integral of $F$ with respect to $\pi$:
\begin{equation*}
\<\pi, F\> \;=\; \int F(x) \, \pi(dx)\;.
\end{equation*}

Given a configuration $\eta\in \Omega_N$, denote by $\pi = \pi(\eta)$
the measure in $\ms M$ obtained by assigning a mass $N^{-1}$ to each
particle:
\begin{equation*}
\pi \;=\; \pi (\eta) \;=\;
\pfrac{1}{N}\sum_{x\in\Lambda_N}\eta_x\, \delta_{x}\;.
\end{equation*}
The measure $\pi$ is called the \emph{empirical measure}.
Denote by
$\color{bblue}\bs \pi: D([0,T],\Omega_N) \to D([0,T],\ms M)$ the map
which associates to a trajectory $\bs \eta (\cdot)$ its empirical
measure:
\begin{equation}
\label{eq:empirical_measure}
\bs \pi (t) \;=\; \pi (\bs \eta(t)) \;=\;
\pfrac{1}{N}\sum_{x\in\Lambda_N}\eta_x(t) \, \delta_{x}\;.
\end{equation}
For a probability measure $\mu$ in $\Omega_N$, let $\mathbb Q^N_{\mu}$
be the measure on $D([0,T], \ms M)$ given by
$\color{bblue} \mathbb Q^N_{\mu} = \mathbb P^N_{\mu} \,\circ\, \bs
\pi^{-1}$.  The first result, due to \cite{BMNS2017}, establishes the
hydrodynamic behavior of the empirical measure.

\begin{theorem}
\label{mt1}
Fix $T>0$, a measurable density profile
$\gamma \colon [0,1] \to [0,1]$, and a sequence $\{\nu^N \}_{N\ge 1}$
of probability measures on $\Omega_N$ associated to $\gamma$ in the
sense that
\begin{equation}
\label{eq:associated}
\lim_{N\to\infty}\nu^N\Big[\,
\Big|\, \<\pi\,,\, H\> \,-\, \int_0^1 \gamma (x)\, H(x)\, dx \,\Big|
\,>\, \delta\,\Big] \;=\; 0
\end{equation}
for all continuous functions $H\colon [0,1] \to \mathbb R$ and
$\delta>0$.  Then, the sequence of probability measures
$\mathbb Q^N_{\nu^N}$ converges weakly to the probability measure
$\mathbb Q$ concentrated on the trajectory $\pi(t,dx) = u(t,x)\, dx$,
where $u$ is the unique weak solution of the heat equation with Robin
boundary conditions
\begin{equation}
\label{1-06}
\begin{cases}
\partial_tu \,=\, \Delta u\\
(\nabla  u) (t,0) \,=\, A^{-1} [\,u(t,0)-\alpha\,] \\
(\nabla  u)(t,1) \,=\, B^{-1} [\, \beta - u (t,1)] \\
u(0,\cdot) \,=\, \gamma (\cdot)\;.
\end{cases}
\end{equation}
\end{theorem}

In this formula, $\nabla u$ stands for the partial derivative in space
of $u$, $\partial_t u$ for its partial derivative in time and
$\Delta u$ for the Laplacian of $u$ in the space variable.  The
definition of weak solutions of equation \eqref{1-06} and the proof of
uniqueness of weak solutions is provided in Appendix \ref{sec06}.  It
is also presented in \cite{BMNS2017}.

\subsection*{The energy}

Denote by $\ms M_{\rm ac}$ the subset of $\ms M$ of all measures which
are absolutely continuous with respect to the Lebesgue measure and
whose density takes values in the interval $[0,1]$, that is,
$\color{bblue} \ms M_{\rm ac} = \{\pi \in \ms M : \pi(dx) = \gamma (x)
\, dx \;\;\text{and}\;\; 0\,\le\, \gamma(x) \,\le\, 1\,\}$.

For $T>0$, let the {\it energy}
$\mc Q_{[0,T]}: D([0,T], \ms M_{\rm ac}) \to [0,\infty]$ be given by
\begin{equation}
\label{energy}
\begin{split}
& \mc Q_{[0,T]} (u) \;:=\; \sup_{G} \mc Q_{[0,T]}^{G} \\
& \quad
:=\; \sup_{G} \Big\{ \int_0^Tdt \int_{0}^1  u(t,x) \,
(\nabla  G)(t,x) \; dx \;-\; \frac 12  \int_0^Tdt \int_{0}^1
G(t,x)^2  \; dx\, \Big\}\;,
\end{split}
\end{equation}
where the supremum is carried over all smooth functions
$G: [0,T]\times (0,1) \to\mathbb R$ with compact support.
\begin{remark}
\label{rm2}
In this definition and below, for a functional
$\Phi \colon D([0,T], \ms M_{\rm ac}) \to \mathbb R$, we often write
$\Phi(u)$ instead $\Phi(\pi)$ when $\pi(t,dx) = u(t,x)\, dx$.
\end{remark}

Clearly, the energy $\mc Q_{[0,T]}$ is convex and lower
semicontinuous. Moreover, if $\mc Q_{[0,T]}(u)$ is finite, $u$ has a
generalized space derivative, denoted by $\nabla u$, and
\begin{equation}
\label{energy1}
\mc Q_{[0,T]} (u) \;=\; \frac 12 \int_0^T dt\, \int_{0}^1
(\nabla  u_t)^2 \; dx \;\cdot
\end{equation}
Denote by $\color{bblue} D_{\mc E} ([0,T], \ms M_{\rm ac})$ the
trajectories in $D ([0,T], \ms M_{\rm ac})$ with finite energy.

\subsection*{The rate functional.}

For $T>0$ and positive integers $m,n$, denote by
$\color{bblue} C^{m,n}([0,T]\times[0,1])$ the space of functions
$G\colon [0,T]\times[0,1] \to\mathbb R$ with $m$ derivatives in time,
$n$ derivatives in space which are continuous up to the
boundary. Denote by $\color{bblue} C^{m,n}_0([0,T]\times[0,1])$ the
functions $G$ in $C^{m,n} ([0,T]\times[0,1])$ such that
$G(t,0) = G(t,1)=0$ for all $t\in[0,T]$, and by
$\color{bblue} C^{m,n}_c([0,T]\times(0,1))$ the functions in
$C^{m,n}([0,T]\times [0,1])$ with compact support in
$[0,T]\times (0,1)$.

For $0< \varrho <1$, $D>0$, $0<a<1$, $M\in \mathbb R$, let
\begin{equation}
\label{4-05}
\mf b_{\varrho, D} (a,M) \;=\; \frac{1}{D}\, \Big\{ [1-a]\, \varrho \,
[e^M-1] \;+\; a\, [1-\varrho] \, [e^{-M}-1] \, \Big\}  \;.
\end{equation}

Fix a trajectory
$\pi(t,dx) = u(t,x) \, dx \in D_{\mc E}([0,T], \ms M_{\rm ac})$. Then,
for almost all $t\in [0,T]$, $\int_{0}^1 (\nabla u_t)^2 \; dx$ is
finite, and therefore $u(t,\cdot)$ is H\"older-continuous. In
particular, $u(t,0)$ and $u(t,1)$ are well defined for almost all $t$.

Denote by $\<\,\cdot\,,\,\cdot\,\>$ the usual scalar product in
$\ms L^2([0,1])$: $\color{bblue} \< f , g \> = \int_0^1 f(x) g(x) dx$
for $f,g\in \ms L^2([0,1])$.  Recall the convention established in
Remark \ref{rm2}.  For each $H$ in $C^{1,2}([0,T]\times[0,1])$, let
$J_{T,H} \colon D_{\mc E}([0,T], \ms M_{\rm ac})\longrightarrow
\mathbb R$ be the functional given by
\begin{equation}
\label{1-01}
\begin{aligned}
& J_{T,H} ( u ) \; =\; \big\langle u_T, H_T \big\rangle
\;-\; \langle u_0, H_0  \rangle
\;-\; \int_0^{T} \big\langle u_t, \partial_t H_t  \big\rangle \; dt
\\
&\quad \; -\; \int_0^{T}  \big\langle u_t , \Delta H_t \big\rangle \; dt
\;+\;  \int_0^{T} u_t(1)\,  \nabla  H_t (1) \; dt
\; -\;   \int_0^{T} u_t(0)\,  \nabla  H_t(0) \; dt  \\
&\quad
\;-\; \int_0^{T}  \big\langle \sigma( u_t ),
\big( \nabla  H_t \big)^2 \big\rangle \; dt \\
&\quad  \;-\; \int_0^T \Big\{\,
\mf b_{\alpha, A} \big(\, u_t(0)\, ,\, H_t (0)\, \big)
\;+\; \mf b_{\beta, B} \big(\, u_t(1)\, ,\, H_t(1)\, \big)\,
\Big\}\, dt \; .
\end{aligned}
\end{equation}
In this formula and below, $\color{bblue} \sigma(a)=a(1-a)$ stands for
the mobility of the exclusion process. Since trajectories in
$D_{\mc E}([0,T], \ms M_{\rm ac})$ have generalized space-derivatives,
we may integrate by parts the second line and write the functional
$J_{T,H} ( \, \cdot\, )$ as
\begin{equation}
\label{1-01b}
\begin{aligned}
J_{T,H} ( u ) \; & =\; \big\langle u_T, H_T \big\rangle
\;-\; \langle u_0 , H_0  \rangle
\;-\; \int_0^{T} \big\langle u_t, \partial_t H_t  \big\rangle \; dt
\\
& +\; \int_0^{T} \big\langle \nabla  u_t , \nabla  H_t \big\rangle \; dt
-\; \int_0^{T}  \big\langle \sigma( u_t ),
\big( \nabla  H_t \big)^2 \big\rangle \; dt \\
&-\; \int_0^T \Big\{\,
\mf b_{\alpha, A} \big(\, u_t(0)\, ,\, H_t (0)\, \big)
\;+\; \mf b_{\beta, B} \big(\, u_t(1)\, ,\, H_t(1)\, \big)\,
\Big\}\, dt \; .
\end{aligned}
\end{equation}
We extend the definition of $J_{T,H} ( \, \cdot\, )$ to
$D ([0,T], \ms M)$ by setting
\begin{equation*}
J_{T,H} ( \pi ) \;=\; \infty\quad
\text{if}\;\; \pi \,\not\in\, D_{\mc E}([0,T], \ms M_{\rm
ac})\;.
\end{equation*}

\begin{remark}
\label{rm3}
This definition differs from the one presented in \cite{BLM09,
FLM2011, FGN_LPP} in the context of exclusion processes with Dirichlet
boundary conditions. There, one defines $J_{T,H} ( \, \cdot\,)$ in
$D ([0,T], \ms M_{\rm ac})$ by an equation similar to \eqref{1-01}
with $u_t(0)$, $u_t(1)$ replaced by the densities $\alpha$, $\beta$,
respectively.  Here, as the boundary values appear and are not fixed
by the dynamics, in the definition of the functional $J_{T,H}$, one is
forced to restrict the definition to trajectories with finite
energy. Otherwise, the boundary values of a density profile are not
defined.
\end{remark}

Let
$I_{[0,T]} (\, \cdot \, ) \colon D([0,T],\ms M_{\rm
ac})\longrightarrow [0,+\infty]$ be the functional defined by
\begin{equation}
\label{1-02}
I_{[0,T]} (u) \; :=\;
\sup_{H\in   C^{1,2} ([0,T]\times[0,1])} J_{T,H}(u)\; .
\end{equation}

Fix a density profile $\gamma$ in $\ms M_{\rm ac}$, and let
$I_{[0,T]} ( \, \cdot\, | \gamma) \colon D ([0,T], \ms M) \to \bb R$
be given by
\begin{equation}
\label{1-02b}
I_{[0,T]}  ( u | \gamma) \;=\; 
\left\{
\begin{aligned}
& I_{[0,T]} ( u ) \quad 
\text{if}\;\; u(0\,,\, \cdot\,) \,=\, \gamma(\,\cdot\,)\;\; \text{a.s.}\;, \\
& \infty\quad  \text{otherwise}\;.
\end{aligned}
\right.
\end{equation}

\begin{theorem}
\label{mt2}
Fix $T>0$ and a measurable function $\gamma: [0,1] \to [0,1]$.  The
function $I_{[0,T]}(\cdot|\gamma):D([0,T],\mathcal{M})\to[0,\infty]$
is convex, lower semicontinuous and has compact level sets.
\end{theorem}

This result is proved in Section \ref{sec03}, where we also show,
in Lemma \ref{l01}, that any path $\pi$ with finite rate function,
$I_{[0,T]} (\pi | \gamma)<\infty$, is weakly continuous in
time. Moreover, Proposition \ref{l06} states that there exists a
finite constant $C_0$ such that
\begin{equation*}
\int_0^T dt \int_0^1 \frac{[\nabla u]^2}{\sigma(u)} \, dx \;\le\;
C_0 \, \{\, I_{[0,T]} (u) +1\, \}
\end{equation*}
for all $u$ in $D_{\mc E}([0,T], \ms M_{\rm ac})$.

In Section \ref{sec07}, we obtain an explicit formula for the action
functional. Proposition~\ref{p01} states that $I_{[0,T]}(\,\cdot\,)$
can be expressed as
$I^{(1)}_{[0,T]}(\,\cdot\,) + I^{(2)}_{[0,T]}(\,\cdot\,)$. The first
term provides the contribution to the rate function due to the
evolution in the interior of the interval $[0,1]$, while the second
one the contribution due to the evolution at the boundary.

In Section \ref{sec05}, we show that trajectories with finite rate
function can be approximated by smooth ones.  The precise statement
requires some notation.

\begin{definition}
\label{d04}
Given $\gamma\in \ms M_{\rm ac}$, let $\Pi_\gamma$ be the collection
of all paths $\pi(t,dx) = u(t,x) dx$ in $D([0,T], \ms M_{\rm ac})$
such that
\begin{itemize}
\item[(a)] There exists $\mf t >0$, such that $u$ follows the hydrodynamic
equation \eqref{1-06} in the time interval $[0, \mf t]$. In
particular, $u(0,\cdot) = \gamma (\cdot)$.

\item[(b)] For every $0<\delta \le T$, there exists $\epsilon>0$ such that
$\epsilon \le u(t,x) \le 1-\epsilon$ for all $(t,x)$ in
$[\delta , T]\times[0,1]$;

\item[(c)] $u$ is smooth on $(0,T]\times [0,1]$.  
\end{itemize}
\end{definition}

Theorem \ref{mt3} states that for all $\gamma: [0,1] \to [0,1]$, the
set $\Pi_\gamma$ is $I_{[0,T]}(\cdot|\gamma)$-dense. This means that
any trajectory $\pi$ in $D([0,T],\ms{M})$ with finite rate function
can be approximated by a sequence of trajectories
$\pi^n\in \Pi_\gamma$ in such a way that $I_{[0,T]}( \pi^{n})$
converges to $I_{[0,T]}( \pi)$.  This is one of the main technical
difficulties in the proof of the lower bound.

We also provide in Section \ref{sec07} an explicit formula for the
rate function of trajectories in $\Pi_\gamma$.  For $0< \varrho <1$,
$D>0$, $0<a<1$, $M\in \mathbb R$, let
\begin{gather}
\label{5-02}
\mf p_{\varrho, D} (a,M) \;=\; \frac{1}{D}\, \Big\{ [1-a]\, \varrho \,
e^M \;-\; a\, [1-\varrho] \, e^{-M}  \, \Big\}  \;,
\\
\mf c_{\varrho, D} (a,M) \;=\; \frac{1}{D}\, \Big\{ [1-a]\, \varrho \,
[1 - e^M + M e^M ] \;+\; a\, [1-\varrho] \,
[\, 1 - e^{-M}  - M  e^{-M}] \, \Big\}  \;. \nonumber
\end{gather}

\begin{proposition}
\label{l09b}
Fix a density profile $\gamma: [0,1]\to [0,1]$ and a trajectory $u$ in
$\Pi_\gamma$. Then, for each $t>0$, the
elliptic equation
\begin{equation}
\label{5-01b}
\left\{
\begin{aligned}
& \partial_t u \;=\; \Delta u \,-\,
2\, \nabla  \{ \sigma(u) \ \nabla  H \}\; , \\
& \nabla  u_t (1) \,-\, 2\, \sigma(u_t(1)) \, \nabla  H_t(1) \,=\,
\mf p_{\beta, B} \big(\, u_t(1)\,,\, H_t(1)\, \big) \;, \\
& \nabla  u_t (0) \,-\, 2\, \sigma(u_t(0)) \, \nabla  H_t(0) \,=\,
-\, \mf p_{\alpha, A} \big(\, u_t(0)\,,\, H_t(0)\, \big) \;,
\end{aligned}
\right.
\end{equation}
has a unique solution, denoted by $H_t$. The function $H$ belongs to
$C^{1,2}([0,T]\times [0,1])$, and the rate functional $I_{[0,T]} (u)$
takes the form
\begin{equation}
\label{5-03}
\begin{aligned}
I_{[0,T]} (u) \; & =\; \int_0^T dt \int_{0}^1 \sigma(u_t)\,
( \nabla  H_t)^2\, dx \; +\; \int_0^T
\mf c_{\beta, B} \big(\, u_t(1)\,,\, H_t(1)\, \big)\; dt \\
\; & +\; \int_0^T
\mf c_{\alpha, A} \big(\, u_t(0)\,,\, H_t(0)\, \big)\; dt \;.
\end{aligned}
\end{equation}
\end{proposition}

\subsection*{{Dynamical large deviations principle.}}

The main result of this article reads as follows.

\begin{theorem}
\label{LDP}
Fix $T>0$, $\gamma\in \ms M_{\rm ac}$ and let
$\{\eta^N\}_{N\in \bb N}$ be a sequence of configurations. Assume that
$\delta_{\eta^N}$ is associated to $\gamma$ in the sense of
\eqref{eq:associated}.  Then, the sequence of probability measures
$\{\bb Q_{\eta^N}\}_{N\geq 1}$ satisfies a large deviation principle
with speed $N$ and good rate function $I_T(\cdot|\gamma)$.  Namely,
for each closed set $\mc C \subset D([0,T], \ms M)$ and each open set
$\mc O \subset D([0,T], \ms M)$
\begin{gather*}
\limsup_{N\to\infty} \frac 1N \log \bb P^N_{\eta^N} 
\big[\, \pi^N \in \mc C\, \big]
\;\leq\; - \, \inf_{\pi \in \mc C} I_T (\pi | \gamma)  
\\
\liminf_{N\to\infty} \frac 1N \log \bb P^N_{\eta^N} 
\big[\,  \pi^N\in \mc O \, \big] \;\geq\;
-  \, \inf_{\pi \in \mc O} I_T (\pi |
\gamma) \; . 
\end{gather*}
\end{theorem}

\begin{remark}
\label{nrm1}
In contrast to \cite{kov, BLM09}, the large deviations principle is
formulated for the empirical measure and not the empirical density. We
follow \cite{FLM2011} to show that the rate function can be set as
$+\infty$ for trajectories that do not belong to
$D([0,T], \ms M_{\rm ac})$.
\end{remark}

\section{The rate functional $I_{[0,T]} (\,\cdot\,)$}
\label{sec03}

In this section, we present some properties of the rate function
$I_{[0,T]} (\,\cdot\,)$ and prove Theorem~\ref{mt2}. Fix,
once for all, a measurable density profile $\gamma: [0,1] \to [0,1]$.

\smallskip\noindent{\it Note:} Throughout this article, given a
function $u: [0,T] \times [0,1]\to\mathbb R$, we represent by $u_t$ and
$u(t)$ the function defined on $[0,1]$ and such that $u_t(x) =
u(t,x)$. \smallskip

We start with two elementary bounds.  The first estimate asserts that
the cost of a trajectory in a interval $[0,T]$ is bounded by the sum
of its cost in the intervals $[0,S]$ and $[S,T]$.  Let
$\tau_r u: \mathbb R_+ \times [0,1] \to \mathbb R$, $r>0$, be the
function defined by $\color{bblue} \tau_r u (t,x) = u(t+r,x)$. For all
$\pi(t,dx) = u(t,x)\, dx$ in $D([0,T],\ms M_{\rm ac})$ and $0<S<T$,
\begin{equation}
\label{4-02}
I_{[0,T]} (u) \;\le\; I_{[0,S]} (u)
\;+\; I_{[0,T-S]} (\tau_S u)\;.
\end{equation}
The proof of this claim is elementary and left to the reader. It
relies on the fact that $\sup_n \{a_n + b_n\} \le \sup_n a_n
+ \sup_n b_n$.

The second assertion states that the cost of a trajectory on a
subinterval of $[0,T]$ is bounded by its total cost.  For all
$\pi(t,dx) = u(t,x)\, dx$ in $D([0,T],\ms M_{\rm ac})$ and $0<S<T$,
\begin{equation}
\label{4-04}
I_{[0,S]} (u) \;\le\; I_{[0,T]} (u)\;.
\end{equation}
To prove this claim, assume that $I_{[0,S]} (u)<\infty$,
and fix $\epsilon>0$. The same argument applies to the case
$I_{[0,S]} (u)=\infty$. By definition of the rate
function, there exists $H:[0,S]\times [0,1]\to \mathbb R$ smooth such that
\begin{equation*}
I_{[0,S]} (u) \;\le\; J_{S,H} (u) \;+\;
\epsilon\;.
\end{equation*}
Let $\sigma_n : [0,T]\to [0,1]$, $n\ge 1$, be a sequence of smooth,
monotone functions such that $\sigma_n(t) = 1$ for $0\le t\le S$ and
$\sigma_n(t) = 0$ for $S+(1/n)\le t\le T$. Define the function $H_n:
[0,T]\times [0,1]\to \mathbb R$ as $H_n(t,x) = H(t,x)$, $0\le t\le S$ and
$H_n(t,x) = H(S,x) \, \sigma_n(t)$, $S\le t\le T$. Then,
\begin{equation*}
J_{S,H} (u) \;=\;
\lim_{n\to\infty} J_{T,H_n} (u)
\;\le\; I_{[0,T]} (u) \;,
\end{equation*}
as claimed.

A similar argument yields that the cost of a trajectory $u$ in a
time-interval $[R,R+S]$ is bounded by the total cost. More precisely,
\begin{equation}
\label{4-06}
I_{[0,S]} (\tau_R u) \;\le\; I_{[0,T]} (u)\;.
\end{equation}
for all $S>0$, $R>0$ such that $R+S\le T$.

The proof of the next result is similar to the ones of \cite[Lemma
3.5]{BDGJL2003}, \cite[Lemma 4.1]{FLM2011}. We present it here in sake
of completeness.

\begin{lemma}
\label{l01}
Fix $T>0$ and $\gamma \in \ms M_{\rm ac}$.  Let $u$ be a path in
$D([0,T],\ms M_{\rm ac})$ such that $I_{[0,T]}(u|\gamma) < \infty$.
Then $u(0,x)=\gamma (x)$.  Moreover, for each $M>0$, $g$ in
$C^{2}([0,1])$ and $\epsilon>0$, there exists $\delta>0$ such that
\begin{equation*}
\sup_{u: I_T (u | \gamma) \le M}\,
\sup_{|t-s|\le \delta}
\big| \, \<u_t, g\> - \<u_s , g\> \, \big| \;\le\; \epsilon\;.
\end{equation*}
In particular, $u$ belongs to $C([0,T],\ms M_{\rm ac})$.
\end{lemma}

\begin{proof}
Fix $T>0$, $\gamma \in \ms M_{\rm ac}$ and $u$ in
$D([0,T], \ms M_{\rm ac})$ such that $I_{[0,T]} (u | \gamma) <\infty$.
We first show that $u (0, \cdot ) = \gamma (\cdot)$.

As $I_{[0,T]} (u | \gamma) <\infty$, $u$ has finite energy.  For
$\delta>0$, consider the function
$H_\delta (t,x) = h_\delta (t) g(x)$, where
$h_\delta (t) = (1- \delta^{-1}t)^+$ and $g$ is a $C^2([0,1])$
function which vanishes at the boundary of the interval $[0,1]$. Here
$a^+$ stands for the positive part of $a$. Of course, $H_\delta$ can
be approximated by smooth functions.  Since $u$ is bounded and since
$t \to u(t, \cdot)$ is right continuous for the weak topology,
\begin{equation*}
\lim_{\delta \downarrow 0} J_{T,H_\delta}(u) \;=\; \langle u(0),
g\rangle - \langle\gamma, g\rangle\; .
\end{equation*}
This proves that $u(0) = \gamma$ a.s.\ because $I_{[0,T]} (u | \gamma)
<\infty$.

We turn to the second assertion of the lemma. Fix $g$ in
$C^{2}([0,1])$ and $0 \leq s < t \le T$ such that $t-s<1$. A
convenient test function, depending only on time and similar to the one
proposed after equation (4.3) in \cite{FLM2011}, yields that
\begin{equation*}
\begin{split}
& \langle u(t) \,,\, g\rangle \,-\,
\langle u(s) \,,\, g \rangle  \;\leq\;
\frac{1}{a}  \, I_{[0,T]} (u|\gamma)
\\
& \quad
\;+\; C_1( \Vert \Delta g \Vert_\infty ,
\Vert \nabla  g \Vert_\infty) \, a \, (t-s)
\;+\; C _2( A, B, \Vert g \Vert_\infty) \, a \, (t-s)\,
e^{a \Vert g \Vert_\infty}
\end{split}
\end{equation*}
for all $a>0$. The exponential term comes from the
$\mf b_{\varrho, D}$ contribution to $J_{T,H}$ in the definition
\eqref{1-01}. Choose $ a = -\, (1/2) \, (1+ \Vert g \Vert_\infty)\,
\log (t-s)$ to get that there exists a finite positive constant $C_0$,
depending only on $A$, $B$, $g$, such that
\begin{equation*}
\langle u(t) \,,\, g\rangle \,-\,
\langle u(s) \,,\, g \rangle  \;\leq\;
\frac{C_0}{\log (t-s)^{-1}}  \,
\big\{ \, I_{[0,T]} (u|\gamma) \,+\, 1\,\big\}\;.
\end{equation*}
This completes the proof of the lemma.
\end{proof}

Let $\color{bblue} \mc H^1$ be the Sobolev space of measurable
functions $G:[0,1]\to \mathbb R$ with generalized derivatives
$\nabla G$ in $\ms L^{2}([0,1])$. $\mc H^1$ endowed with the scalar
product $\langle\cdot,\cdot\rangle_{1}$, defined by
\begin{equation}
\label{n08}
\langle G \,,\, H \rangle_{1}\;=\;
\langle G \,,\, H\rangle \;+\;
\langle \nabla  G \,,\, \nabla  H\rangle\;,
\end{equation}
is a Hilbert space. The corresponding norm is denoted by
$\|\cdot\|_{\mc H^1}$:
\begin{equation*}
\|G\|_{\mc H^1}^{2} \;:=\;
\int_0^1 |G(x)|^2 \;dx \;+\;
\int_0^1 |\nabla  G(x)|^2 \;dx\;.
\end{equation*}
Recall from \eqref{1-03} that any function $\gamma$ in $\mc H^1$ has a
continuous version.  Hereafter, we always replace $\gamma$ by its
continuous version $w$.

Consider the function $\phi:\mathbb R\to[0,\infty)$ defined by
\begin{equation*}
\begin{split}
\phi(r)\;:=\;
\begin{cases}
\dfrac{1}{Z}\exp{\{-\dfrac{1}{(1-r^{2})}\}}\ & \text{ if } |r|<1\;, \\
\ 0\ & \text{otherwise}\;,
\end{cases}
\end{split}
\end{equation*}
where the constant $Z$ is chosen so that $\int_{\mathbb R}\phi(r) dr =1$.
For each $\delta>0$, let
\begin{equation}
\label{n04}
\phi^{\delta}(r) \;:=\; \dfrac{1}{\delta}\,
\phi \Big(\dfrac{r}{\delta}\Big)\;,
\end{equation}
whose support is contained in $[-\delta,\delta]$.

Denote by $f*g$ the space or time convolution of two functions $f$,
$g$:
\begin{equation*}
(f*g)(a) \;=\; \int f(a-b)\, g(b) \, db\;,
\end{equation*}
where the integral runs over $\mathbb R$.

Throughout this section, we adopt the following notation. Recall from
Appendix \ref{sec04} that we denote by $(P^{(R)}_t : t\ge 0)$ the
semigroup associated to the Robin Laplacian. For a bounded measurable
function $ u:[0,T]\times[0,1]\to\mathbb R$, define the smooth
approximation in space, time and space-time by
\begin{gather*}
u^{\epsilon}(t,x) \;:=\;  [\, P^{(R)}_\epsilon u_t\, ](x)  \;, \quad
u^{\delta}(t,x) \;:=\; [ u(\cdot, x) * \phi^{\delta}](t)
\;=\; \int_{-\delta}^{\delta} u(t+r,x)\, \phi^{\delta}(r)\; dr\;, \\
u^{\epsilon, \delta}(t,x) \;:=\; [\, P^{(R)}_\epsilon u^\delta_t\,
](x) \;:=\; [\, P^{(R)}_\epsilon u_t\, ]^\delta (x)  \;.
\end{gather*}
In the above formulas, we extend the definition of $u$ to
$[-1,T+1] \times [0,1]$ by setting $ u_{t}= u_{0}$ for
$-1\le t \le 0$, $ u_{t}= u_{T}$ for $T \le t \le T+1$,

Note that we use the same notation, $ u^{\epsilon}$ and $ u^{\delta}$,
for different objects. However, $ u^{\epsilon}$ and $ u^{\delta}$
always represent a smooth approximation of $ u$ in space and time,
respectively. Moreover, the time-convolution commutes with the
operator which explains the identity in the last displayed equation.

We summarize some properties of $ u^{\epsilon}$ in the next result.

\begin{lemma}
\label{l02}
Let $ u:[0,T]\times[0,1]\to\mathbb R$ be a function in
$\ms L^{2}(0,T; \mc H^1)$.  Then, $u^{\epsilon}$ and
$\nabla u^{\epsilon}$ converge to $u$ and $\nabla u$ in
$\ms L^{2}([0,T]\times[0,1])$, respectively.  Moreover, if $u$ is
bounded in $[0,T]\times[0,1]$ and the application
$t\mapsto \langle u_{t}, g\rangle$ is continuous in the time interval
$[0,T]$ for any function $g$ in $C^{\infty}([0,1])$, then, for each
$\epsilon>0$, $n\ge 1$, $u^{\epsilon}$ and $\nabla ^n u^{\epsilon}$
are uniformly continuous in $[0,T]\times[0,1]$.
\end{lemma}

\begin{proof}
Recall the notation introduced in Appendix \ref{sec04}. As $u$ belongs
to $\ms L^{2}(0,T; \mc H^1)$ and the norms $\mc H^1$, $\mc H_R$ are
equivalent,
\begin{equation*}
\int_0^T \Vert \, u_t \, \Vert^2_{\mc H_R} \, dt \;<\; \infty \;.
\end{equation*}
This relation can be rewritten in terms of the eigenfunctions $(f_k :
k\ge 1)$ of the Robin Laplacian as
\begin{equation}
\label{4-03}
\int_0^T \sum_{k\ge 1} \lambda_k \;
\< u_t \,,\, f_k\>^2 \, dt \;<\; \infty \;.
\end{equation}

Since
\begin{equation*}
u_t \;=\;
\sum_{k\ge 1}  \< u_t \,,\, f_k\>\, f_k\;, \qquad
u^{\epsilon}_t \;=\;
\sum_{k\ge 1} e^{-\lambda_k\epsilon}\, \< u_t \,,\, f_k\>\, f_k\;,
\end{equation*}
we have that
\begin{equation*}
\int_0^T \Vert \, u^{\epsilon}_t - u_t \, \Vert^2_{2} \; dt
\;=\; \int_0^T
\sum_{k\ge 1} \big[\, e^{-\lambda_k\epsilon} \,-\, 1\,\big]^2
\, \< u_t \,,\, f_k\>^2\; dt
\end{equation*}
and, by \eqref{6-14},
\begin{equation*}
\begin{aligned}
\int_0^T \Vert \, \nabla  u^{\epsilon}_t - \nabla  u_t \, \Vert^2_{2} \, dt
\; & \le\; C_0\,
\int_0^T \Vert u^{\epsilon}_t - u_t \Vert^2_{\mc H_R} \, dt \\
\; & =\; C_0\, \int_0^T
\sum_{k\ge 1} \lambda_k\,
\big[\, e^{-\lambda_k\epsilon} \,-\, 1\,\big]^2
\, \< u_t \,,\, f_k\>^2 \, dt \;.
\end{aligned}
\end{equation*}
By \eqref{4-03}, the left-hand side of the previous two displayed
equations vanish as $\epsilon \to 0$, which proves the first assertion
of the lemma.

We turn to the second assertion. We may represent $u^{\epsilon}$,
$\nabla ^n u^{\epsilon}$ as
\begin{equation*}
u^{\epsilon}_t (x) \;=\;
\sum_{k\ge 1} e^{-\lambda_k\epsilon}\, \< u_t \,,\, f_k\>\, f_k(x)
\;, \quad
(\nabla ^n u_t^{\epsilon}) (x) \;=\;
\sum_{k\ge 1} e^{-\lambda_k\epsilon}\, \< u_t \,,\, f_k\>\,
(\nabla ^n f_k) (x) \;,
\end{equation*}
The second assertion follows from these identities and from the two
hypotheses of the lemma. Indeed, the bound \eqref{6-11} on the
eigenfunctions $f_k$ permits to restrict the sum to a finite number of
terms.
\end{proof}

For each $a>0$, define the functions $h_a$ and $\sigma_{a}$ on $[0,1]$
by
\begin{gather*}
h_a(x) \;:=\; \dfrac{1}{2(1+2a)} \, \Big\{ ( x+a)
\log{( x+a)}+(1- x + a)\log{(1- x + a)} \Big\}\;, \\
\sigma_{a}(x) \;:=\; (x+a)(1- x+a)\;.
\end{gather*}
Note that $h_a''= (2\, \sigma_{a})^{-1}$.

Until the end of this section, $0<C_0<\infty$ represents a constant
independent of $\epsilon$, $\delta$ and $a$ and that may change from
line to line.

Fix $T>0$ and a path $u$ in $D_{\mc E}([0,T],\ms M_{\rm ac})$.  For a
smooth function $G: [0,T]\times[0,1]\to\mathbb R$ and a bounded
function $H$ in $\ms L^{2}(0,T; \mc H^1)$, define the functionals
\begin{equation*}
\begin{split}
& L_{G}(u) \;=\; \langle u_{T}, G_{T}\rangle
\,-\,  \langle u_0 , G_{0}\rangle
\,-\, \int_{0}^{T}  \langle u_{t}, \partial_{t}G_{t} \rangle\; dt \;, \\
&\quad B^{1}_{H}(u) \;=\; \int_{0}^{T}
\langle\nabla  u_{t}, \nabla  H_{t}\rangle\; dt
\;-\; \int_{0}^{T}
\langle \sigma( u_{t}), (\nabla  H_{t})^{2}\rangle\; dt \;, \\
&\qquad B^{2}_{H}(u) \;=\;
\int_0^T \Big\{\,
\mf b_{\alpha, A} \big(\, u_t(0)\, ,\, H_t (0)\, \big)
\;+\; \mf b_{\beta, B} \big(\, u_t(1)\, ,\, H_t(1)\, \big)\,
\Big\}\, dt \;.
\end{split}
\end{equation*}
By \eqref{1-01}, for paths $u$ such that $u(0) = \gamma$,
\begin{equation*}
\sup_{H\in C^{1,2}([0,T]\times[0,1])}
\Big\{ L_{H}(u)+B_{H}^{1}(u)-B_{H}^{2}(u)\Big\}
\;=\; I_{[0,T]}(u|\gamma)\;.
\end{equation*}

\begin{lemma}
\label{l03}
For $a>0$, $\epsilon>0$, $\delta >0$, let
$H_{\epsilon,\delta} = h_a'( u^{\epsilon,\delta})$,
\begin{equation*}
R^{\epsilon,\delta} \;=\; L_{H_{\epsilon,\delta}}(u^{\epsilon,\delta})
\;-\; L_{(H_{\epsilon,\delta})^{\epsilon,\delta}}(u)\;.
\end{equation*}
Then, for any fixed $a>0$, $\epsilon>0$, $R^{\epsilon,\delta}$
converges to $0$ as $\delta\downarrow0$.
\end{lemma}

\smallskip\noindent{\bf Warning:} Until the end of Proposition
\ref{l06} proof's, we drop the dependence of
$H = H_{\epsilon,\delta} = h_a'( u^{\epsilon,\delta})$ on $\epsilon$,
$\delta$. Hence, $H$ always stands for $H_{\epsilon,\delta}$.

\begin{proof}[Proof of Lemma \ref{l03}]
Recall that $C_0$ represents a constant independent of
$\epsilon$, $\delta$ and $a$, that may change from line to line.  As
$P^{(R)}_\epsilon$ is a self-adjoint operator in $\ms L^2([0,1])$ and
commutes with the time-derivative,
\begin{equation*}
\begin{split}
L_{H}(u^{\epsilon,\delta})
& \;=\;\langle  u^{\delta}_{T}, H_{T}^{\epsilon} \rangle
- \langle u_{0}^{\delta}, H_{0}^{\epsilon} \rangle
- \int_{0}^{T} \langle  u^{\delta}_{t},
\partial_{t}H_{t}^{\epsilon}\rangle  \;dt \\
& \;=\; \langle  u_{T}, H_{T}^{\epsilon,\delta} \rangle
- \langle u_{0}, H_{0}^{\epsilon,\delta} \rangle
- \int_{0}^{T} \langle  u^{\delta}_{t},
\partial_{t}H_{t}^{\epsilon}\rangle \;dt
\;+\;  R_{1}^{\epsilon,\delta}\;,
\end{split}
\end{equation*}
where
\begin{equation*}
R_{1}^{\epsilon,\delta} := R^{\epsilon,\delta,T} -
R^{\epsilon,\delta,0} \quad\text{and}\quad R^{\epsilon,\delta,t}
:=\langle u_{t}^{\delta}- u_{t}, H_{t}^{\epsilon}\rangle +\langle
u_{t}, H_{t}^{\epsilon}-H^{\epsilon,\delta}_{t} \rangle\;
\end{equation*}
for $0\le t\le T$.

A simple computation yields that
\begin{equation*}
\int_{0}^{T} \langle  u_{t}^{\delta},
\partial_{t}H_{t}^{\epsilon}\rangle \;dt
\;=\; \int_{0}^{T} \langle u_{t},
\partial_{t}H_{t}^{\epsilon,\delta} \rangle \; dt
\;+\; R_{2}^{\epsilon,\delta}  \;,
\end{equation*}
where
$|R_{2}^{\epsilon,\delta}|\le
C_0\delta\|\partial_{t}H^{\epsilon}\|_{\infty}$.  To conclude the
proof, it is enough to show that, for each fixed $a>0$, $\epsilon>0$,
$R_{1}^{\epsilon,\delta}$ and
$\delta\|\partial_{t}H^{\epsilon}\|_{\infty}$ converge to zero as
$\delta\downarrow0$.

We first prove that
\begin{equation}
\label{lim4.1}
\lim_{\delta\downarrow0}R^{\epsilon,\delta,t} \;=\; 0 \quad\text{for}\quad
t=0\;\text{ and } \;t=T\;.
\end{equation}
We consider the case $t=T$, the argument being similar for
$t=0$. As $P^{(R)}_t$ is symmetric,
\begin{equation*}
R^{\epsilon,\delta,T} \;=\;
\langle  u_{T}^{\epsilon,\delta}- u_{T}^{\epsilon}, H_{T}\rangle
+\langle  u^{\epsilon}_{T}, H_T
-H^{\delta}_T \rangle \;.
\end{equation*}
By Lemma \ref{l02}, for each $x\in[0,1]$, $u^{\epsilon}(\cdot,x)$ is
continuous.  Therefore, by definition of $H$, for any
$(t,x)\in[0,T]\times[0,1]$,
\begin{gather*}
\lim_{\delta\downarrow0} u^{\epsilon,\delta}(t,x)
\;=\;  u^{\epsilon}(t,x) \; ,  \\
\lim_{\delta\downarrow0} H^{\delta}(T,x)
\;=\; \lim_{\delta\downarrow 0} h_a'( u^{\epsilon,\delta})^{\delta}
(T,x) \;=\; h_a'( u^{\epsilon}) (T,x) \;=\;
\lim_{\delta\downarrow0} h_a'( u^{\epsilon,\delta})
\;=\; \lim_{\delta\downarrow0} H (T,x)
\nonumber
\end{gather*}
because $h'_a$ is bounded and continuous on $[0,1]$. Note that the
dependence on $\delta$ of the last term on the right-hand side is
hidden, as $H$ is actually $h_a'( u^{\epsilon,\delta})$. Claim
\eqref{lim4.1} follows from these results, from the boundedness of $u$
and $h'_a$, and the bounded convergence theorem.

It remains to show that $\delta\|\partial_{t}H^{\epsilon}\|_{\infty}$
converges to $0$ as $\delta \downarrow 0$.  An elementary computation
gives that, for any $t\in[0,T]$,
\begin{equation*}
\partial_{t}H^{\epsilon}(t)
\;=\; P^{(R)}_\epsilon \Big[\,
h''_a( u^{\epsilon,\delta}(t)) \,
\int_{-\delta}^{\delta}  u^{\epsilon}(t-r) \,
(\phi^{\delta})'(r)\; dr \, \Big] \;.
\end{equation*}
Since $\phi^{\delta}$ is a symmetric function, a change of variables
shows that
\begin{equation*}
\int_{-\delta}^{\delta}
 u^{\epsilon}(t-r)\, (\phi^{\delta})'(r) \; dr
\;=\;\int_{0}^{\delta}
\{ \, u^{\epsilon}(t-r) \,-\,
u^{\epsilon}(t+r)\, \} \, (\phi^{\delta})'(r) \; dr \;.
\end{equation*}
By Lemma \ref{l02}, $ u^{\epsilon}$ is uniformly continuous on
$[-1,T+1]\times [0,1]$. On the other hand,
$\delta\int_0^{\delta} (\phi^{\delta})'(r) \, dr = - \, \phi(0)$.
Therefore, the last expression multiplied by $\delta$ converges to $0$
as $\delta\downarrow0$ uniformly in $[0,T]\times[0,1]$. Since $h''_a$ is
uniformly bounded, by the bounded convergence theorem,
$\delta\|\partial_{t}H_{\epsilon}\|_{\infty}$ converges to $0$ as
$\delta\downarrow0$.
\end{proof}

\begin{lemma}
\label{l04}
There exists a positive constant $C_0<\infty$ such that
\begin{equation}
\label{bound4.4}
\int_{0}^{T} dt \ \int_0^1
\frac{(\nabla  u(t,x))^{2}}{\sigma_{a}( u(t,x))}
\; dx \;\le\; C_0 \, B^{1}_{h_a'( u)}(u)\;,\quad
|B^{2}_{h_a'( u)}(u)| \;\le\; C_0
\end{equation}
for all $u\in D_{\mc E} ([0,T], \ms M_{\rm ac})$ and
$0<a<1$. Moreover, for each $u\in D_{\mc E} ([0,T], \ms M_{\rm ac})$
and $i=1,2$,
\begin{equation*}
\lim_{\epsilon\downarrow0}\lim_{\delta\downarrow0}
B^{i}_{H^{\epsilon,\delta}}(u) \;=\; B^{i}_{h_a'(u)}(u)\;.
\end{equation*}
\end{lemma}

\begin{proof}
Let $u$ be a path in $D_{\mc E} ([0,T], \ms M_{\rm ac})$.  We first
show that
\begin{equation}
\label{lim4.2}
\lim_{\epsilon\downarrow0}\lim_{\delta\downarrow0}
B^{1}_{H^{\epsilon,\delta}}(u) \;=\;
B^{1}_{h_a'( u)}(u)\;.
\end{equation}

By Lemma \ref{l02}, $\nabla  u^{\epsilon}$ is uniformly continuous in
$[0,T]\times [0,1]$.  Therefore, for any $(t,x)\in[0,T]\times[0,1]$,
\begin{gather*}
\lim_{\delta\downarrow0}\nabla  u^{\epsilon,\delta}(t,x)
\;=\;\nabla  u^{\epsilon}(t,x)\;.
\end{gather*}
Recall from the end of the Appendix \ref{sec06} the definition of the
semigroup $P^{(M)}_t$. By \eqref{6-23b}
$\nabla P^{(R)}_\epsilon = P^{(M)}_\epsilon \nabla$. Hence,
\begin{gather*}
\lim_{\delta\downarrow0}\nabla  H^{\epsilon,\delta} (t,x)
\;=\; \lim_{\delta\downarrow0} P^{(M)}_\epsilon
\Big[\, h_a''( u^{\epsilon, \delta})\, \nabla  u^{\epsilon, \delta}
\,\Big]^\delta (t,x)
\;=\; P^{(M)}_\epsilon \big[\,
h_a''( u^{\epsilon}_t )\, \nabla  u^{\epsilon}_t\, \big] (x) \;.
\end{gather*}
Hence, as $(\nabla u)(t,x) \, dx\, dt$ is a finite measure on
$[0,T]\times [0,1]$, by the bounded convergence theorem,
\begin{equation}
\label{01}
\lim_{\delta\downarrow0}B^{1}_{H^{\epsilon,\delta}}(u) \;=\;
\int_{0}^{T}
\Big\{\, \big\< \, \nabla   u_{t} \,,\, G^\epsilon_t \,\big \>
\,-\, \big\<\, \sigma( u_{t}) \,,\,
[\, G^\epsilon_t \, ]^{2} \big \> \, \Big\}\; dt \;,
\end{equation}
where $G^\epsilon (t,x) = P^{(M)}_\epsilon [\,
h_a''( u^{\epsilon}_t )\, \nabla  u^{\epsilon}_t\, ] (x)$.

On the one hand, since $P^{(M)}_\epsilon$ is a contraction in
$\ms L^2 ([0,1])$, $h_a''$ is bounded, and since, by Lemma \ref{l02},
$\nabla u^{\epsilon}$ converges to $\nabla u$ in
$\ms L^{2}([0,T]\times[0,1])$,
\begin{equation*}
\lim_{\epsilon\downarrow0} \int_{0}^{T} dt\ \int_0^1
\big\{ \, P^{(M)}_\epsilon \big [\, h_a''( u^{\epsilon}_{t}) \, \big(
\nabla  u^{\epsilon}_{t} - \nabla  u_{t} \,\big) \,\big] \big\}^2
\; dx  \;=\; 0\;.
\end{equation*}
Therefore, on the right-hand side of \eqref{01}, in the formula for
$G^{\epsilon}$ we may replace $\nabla u^{\epsilon}_{t}$ by
$\nabla u_{t}$ at a cost that vanishes as $\epsilon\to 0$.

Since $h_{a}''$ is Lipschitz continuous, by Lemma \ref{l02}, as
$\epsilon\downarrow 0$, $h_{a}''( u^{\epsilon})$ converges in measure
to $h_{a}''( u)$. In other words, for any $b>0$, the Lebesgue measure
of the set
$\{(t,x)\in [0,T]\times[0,1]; |h_{a}''( u^{\epsilon}(t,x))-h_{a}''(
u(t,x))|\ge b\}$ converges to $0$ as $\epsilon\downarrow0$.
Therefore, as $\nabla  u_t$ belongs to $\ms L^2 ([0,T]\times[0,1])$,
\begin{equation*}
\lim_{\epsilon\downarrow0} \int_{0}^{T}
\big\< h_a''( u^{\epsilon}_{t}) \, \big[\nabla  u_{t}\big]^2 \big\>
\;dt \;=\; \int_{0}^{T}
\big\< h_a''( u_{t}) \, \big[\nabla  u_{t}\big]^2 \big\>
\;dt \;.
\end{equation*}
In consequence, on the right-hand side of \eqref{01}, in the formula
for $G^{\epsilon}$ we may further replace $h_{a}''( u^{\epsilon})$ by
to $h_{a}''( u)$.

To complete the proof of \eqref{lim4.2}, it remains to recall that for
any $f$ in $\ms L^2 ([0,1])$, $P^{(M)}_\epsilon f$ converges in
$\ms L^2 ([0,1])$ to $f$ as $\epsilon \to 0$.

We turn to the proof that
\begin{equation*}
\lim_{\epsilon\downarrow0}\lim_{\delta\downarrow0}\,
B^{2}_{H^{\epsilon,\delta}}(u) \;=\;  B^{2}_{h_a'( u)}(u) \;.
\end{equation*}
We examine the boundary condition at $x=0$, the other one being
similar.

By Lemma \ref{l02}, $u^{\epsilon}$ is uniformly continuous. Hence, as
$h_a'$ is continuous in the interval $[0,1]$, as $\delta \to 0$,
$H^{\epsilon,\delta} (t,0)$ converges to 
$P^{(R)}_\epsilon [ \, h_a'(u^{\epsilon}_t)\, ] (0)$. Therefore,
\begin{equation}
\label{n01}
\lim_{\delta\downarrow0}
\int_0^T \mf b_{\alpha, A} \big(\, u_t(0)\, ,\,
H_{\epsilon,\delta} (t,0)\, \big) \, dt
\;=\; \int_0^T \mf b_{\alpha, A} \big(\, u_t(0)\, ,\,
P^{(R)}_\epsilon [ \, h_a'(u^{\epsilon}_t)\, ] (0) \, \big) \, dt\;.
\end{equation}

To conclude the proof, we first replace on the right-hand side
$P^{(R)}_\epsilon [ \, h_a'(u^{\epsilon}_t)\, ] (0)$ by
$h_a'(u^{\epsilon}_t (0))$. Since $ h_a'$ is bounded, there exists a
finite constant $C_1 = C_1(a, A, \alpha)$ such that
\begin{equation*}
\begin{aligned}
& \Big|\, \int_0^T \mf b_{\alpha, A} \big(\, u_t(0)\, ,\,
P^{(R)}_\epsilon [ \, h_a'(u^{\epsilon}_t)\, ] (0) \, \big) \, dt
\;-\;
\int_0^T \mf b_{\alpha, A} \big(\, u_t(0)\, ,\,
h_a'(u^{\epsilon}_t (0)) \, \big) \, dt
\,\Big| \\
&\quad \le\;
C_1 \int_0^T \big|\,
P^{(R)}_\epsilon [ \, h_a'(u^{\epsilon}_t)\, ] (0) \, -\,
h_a'(u^{\epsilon}_t (0))\,\big| \, dt\;.
\end{aligned}
\end{equation*}

By \eqref{6-05}, Lemma \ref{l08} and \eqref{6-14},
\begin{equation*}
\big|\, P^{(R)}_\epsilon [ \, h_a'(u^{\epsilon}_t)\, ] (0) \;-\;
h_a'(u^{\epsilon}_t (0)\, )\, \big| \;\le\;
C_0\, \sqrt{ A} \, \epsilon^{1/5}
\big\Vert\, h_a'(u^{\epsilon}_t \, )\, \big\Vert_{\mc H^1}
\end{equation*}
for some finite constant $C_0$. Hence, the term on the right-hand side
in the penultimate displayed equation is bounded by
\begin{equation*}
C_1 \,  \epsilon^{1/5}\, \int_0^T
\big\Vert\, h_a'(u^{\epsilon}_t \, )\, \big\Vert_{\mc H^1}
\, dt\;.
\end{equation*}
By Lemma \ref{l02}, $u^{\epsilon}$ converges to $u$ in
$\ms L^2(0,T, \mc H^1)$. Since $h_a'$ and $h_a''$ are bounded, the
previous integral is bounded uniformly in $\epsilon$. In particular,
the previous expression vanishes as $\epsilon \to 0$.

It remains to estimate the right-hand side of \eqref{n01} with
$P^{(R)}_\epsilon [ \, h_a'(u^{\epsilon}_t)\, ] (0)$ replaced by
$h_a'(u^{\epsilon}_t (0))$. By \eqref{6-05} and since, by Lemma
\ref{l02}, $u^{\epsilon}$ converges to $u$ in
$\ms L^2(0,T, \mc H^1)$, $\lim_{\epsilon \to 0}
u^{\epsilon}_t(0) \,=\, u_t(0)$. Hence, as $h_a'$ is continuous,
and since, by \eqref{6-17}, $u^{\epsilon}$ is uniformly bounded, by
the bounded convergence theorem,
\begin{equation*}
\lim_{\epsilon \to 0}
\int_0^T \mf b_{\alpha, A} \big(\, u_t(0)\, ,\,
h_a'(u^{\epsilon}_t(0)) \, \big) \, dt
\;=\; \int_0^T \mf b_{\alpha, A} \big(\, u_t(0)\, ,\,
h_a'(u_t(0)) \, \big) \, dt \;,
\end{equation*}
which completes the proof of first assertion of the lemma.

We turn to the bounds \eqref{bound4.4}. As $\sigma(x) \le
\sigma_a(x)$, for each $\epsilon >0$,
\begin{equation*}
\int_{0}^{T}
\langle\nabla  u^{\epsilon}_{t} \,,\, \nabla  h_a'(u^{\epsilon}_t) \rangle\; dt
\;-\; \int_{0}^{T}
\langle \sigma_a ( u^{\epsilon}_{t}) \,,\, 
(\nabla  h_a'(u^{\epsilon}_t))^{2}\rangle\; dt \;\le\;
B^1_{h_a'(u^{\epsilon})} (u^{\epsilon}) \;.
\end{equation*}
Compute the left-hand side to get that
\begin{equation*}
\frac{1}{4}\, \int_{0}^{T} dt \ \int_0^1
\frac{(\nabla  u^{\epsilon})^{2}}{\sigma_{a}( u^{\epsilon})}
\; dx \;\le\; B^1_{h_a'(u^{\epsilon})} (u^{\epsilon}) \;.
\end{equation*}
The arguments presented in the first part of the proof permit to let
$\epsilon \to 0$ on both sides of this inequality and yield the first
estimate in \eqref{bound4.4}.

To estimate $B^{2}_{h_a'( u)}(u)$, note that
\begin{equation*}
\begin{aligned}
& \mf b_{\varrho, D} \big(\, v\, ,\, h_a'(v) \, \big) \;=\;  \\
& \frac{1}{D}\, \Big\{\, \varrho\, [1-v]\,
\Big[ \, \Big( \frac{v+a}{1+a-v}\Big)^{1/2+4a} \,-\, 1\,\Big]
\,+\, [1-\varrho]\, v \,
\Big[ \, \Big( \frac{1+a-v}{v+a}\Big)^{1/2+4a} \,-\,
1\,\Big]\,\Big\}\;.
\end{aligned}
\end{equation*}
In particular, $\mf b_{\varrho, D}$, as a function of $v$ and $a$ is
bounded: for all $0<\varrho<1$, $D>0$, there exists a finite constant
$C_0 = C_0(\varrho, D)$ such that
\begin{equation*}
\sup_{0\le a<1} \sup_{v\in [0,1]}
\big|\, \mf b_{\varrho, D} \big(\, v\, ,\, h_a'(v) \, \big)\,\big|
\;\le\; C_0\;.
\end{equation*}
The second inequality in \eqref{bound4.4} follows from this estimate
and the definition of $B^{2}_{h_a'( u)}(u)$.
\end{proof}

\begin{proposition}
\label{l06}
There exists a constant $C_0>0$ such that
\begin{equation*}
\int_0^{T} dt \int_0^1 \frac{|\nabla  u(t,x)|^{2}}{\sigma(u(t,x))}
\; dx
\;\le\; C_0\,\{ \, I_{[0,T]}(u) \,+\, 1 \, \}
\end{equation*}
for any path $u$ in $D_{\mc E}([0,T],\ms M_{\rm ac})$.
\end{proposition}

\begin{proof}
We may assume, without loss of generality, that $I_{[0,T]}(u)$ 
is finite.  By the variational formula \eqref{1-02} and with the
notation of Lemma \ref{l03},
\begin{align}
\label{bound4.2}
L_{H}(u^{\epsilon,\delta})
+B^{1}_{H^{\epsilon,\delta}}(u)-B^{2}_{H^{\epsilon,\delta}}(u) - R^{\epsilon,\delta}
\;\le\; I_{[0,T]}(u)\;,
\end{align}
where, recall, $H$ stands for the function
$h_a'( u^{\epsilon,\delta})$.

Since $ u^{\epsilon,\delta}$ is smooth, an integration by parts
yields that
\begin{equation*}
L_{H}(u^{\epsilon,\delta})\;=\;
\int_0^1  h_a( u^{\epsilon,\delta}_{T}) \; dx
\;-\;  \int_0^1  h_a( u^{\epsilon,\delta}_{0})\; dx \;.
\end{equation*}
There exists, therefore, a constant $C_0$, independent of $\epsilon$,
$\delta$ and $a$, such that
\begin{equation*}
|L_{H}(u^{\epsilon,\delta})| \;\le\; C_0\;.
\end{equation*}

In \eqref{bound4.2}, let $\delta\downarrow 0$ and then
$\epsilon\downarrow 0$. It follows from the previous bound, and from
Lemmata \ref{l03} and \ref{l04} that
\begin{equation*}
B^{1}_{h'_a(u)}(u) \;-\;
B^{2}_{h'_a(u)}(u) \;\le\; I_{T}(u) \;+\; C_0 \;.
\end{equation*}
Thus, by \eqref{bound4.4},
\begin{equation*}
\int_0^{T} dt\ \int_0^1 \frac{|\nabla  u(t,x)|^{2}}
{\sigma_{a}( u(t,x))} \; dx \;\le\; C_0 \, \{ I_{T}(u)+1\}\;.
\end{equation*}
It remains to let $a\downarrow0$ and to apply Fatou's lemma.
\end{proof}

\noindent{\bf Note:} Since the rate function is declared to be
infinite on trajectories with infinite energy, this result is not
meant to show that a trajectory has finite energy. Its interest lies
on the fact that it provides a uniform bound of a strong version of
the energy for trajectories with rate function bounded by a constant.

\begin{corollary}
\label{zero}
The density $u$ of a path $\pi (t,dx) = u(t,x)\, dx$ in
$D([0,T], \ms M_{\rm ac})$ is the weak solution of the
initial-boundary value problem \eqref{1-06} if, and only if
$I_{T}(u|\gamma) = 0$.
\end{corollary}

\begin{proof}
Suppose that the density $u$ of a path $\pi (t,dx)= u(t,x)\, dx$ in
$D([0,T], \ms M_{\rm ac})$ is the weak solution of the
initial-boundary value problem \eqref{1-06}. Then, by Lemma \ref{l14},
$u$ has finite energy. On the other hand, by Definition
\ref{d02} and the equation following it, $u(0) = \gamma$ a.s. and for
any $G$ in $C^{1,2}([0,T]\times[0,1])$,
\begin{equation*}
\begin{split}
J_{T,G}(u) \;=\;
& - \, \int_{0}^{T}  \< \sigma( u_{t}), (\nabla  G_{t})^{2}\> \; dt \\
& -\int_{0}^{T} \big \{\,  \mf q_{\beta, B} (u_t(1), G_t(1))
\,+\, \mf q_{\alpha, A} (u_t(0), G_t(0)) \,\big\} \; dt\;,
\end{split}
\end{equation*}
where
\begin{equation}
\label{6-04}
\mf q_{\varrho, D} (a, M) \;=\; \frac{1}{D}\, \Big\{ [1-a] \, \varrho \,
\big[\, e^{M} \,-\, M \, -\, 1\,\big] \;+\;
a \, [1-\varrho] \,  \big[\, e^{-M} \,+\, M \, -\, 1\,\big]\, \Big\} \;.
\end{equation}
Here we used the fact that $u-a$ can be written as $u(1-a) - (1-u)a$.
As $\mf q_{\varrho, D} (a, M) \ge 0$, $J_{T,G}(u) \le 0$. Hence, the
supremum in the variational problem \eqref{1-02} is attained at $H=0$
and $I_{[0,T]}(u)=0$. Since $u(0) = \gamma$, $I_{[0,T]}(u|\gamma)=0$.

On the other hand, if $I_{[0,T]}(u|\gamma) = 0$, then, for any $G$ in
$C^{1,2}([0,T]\times[0,1])$ and $\epsilon$ in $\mathbb R$,
$J_{T,\epsilon G}(u)\le 0$. Since $J_{T,0}(u) = 0$, the derivative in
$\epsilon$ of $J_{T,\epsilon G}(u)$ at $\epsilon=0$ is equal to
$0$. Therefore, by Definition \ref{d02}, the density $u$ is a weak
solution of the initial-boundary value problem \eqref{1-06}.
\end{proof}

Let $E_{q}$, $q\ge 0$, be the level set of the rate function
$I_{[0,T]}(\cdot|\gamma)$:
\begin{equation*}
E_{q}\;:=\; \big\{\, \pi\in D([0,T], \ms M) \,|\,
I_{[0,T]}(\pi|\gamma) \,\le\, q\, \}\;.
\end{equation*}

\begin{proof}[Proof of Theorem \ref{mt2}]
The rate function $I_{[0,T]}(\, \cdot\, |\, \gamma)$ is convex because
the energy $\mc Q_{[0,T]} (\,\cdot\,)$ and the functionals
$J_{T,H}(\,\cdot\,)$ are convex.

Let $\{\pi^{n}:n\ge1\}$ be a sequence in $D([0,T],\ms M)$ such that
$\pi^{n}$ converges to some element $\pi$ in $D([0,T], \ms M)$.  We
show that
$I_{[0,T]}(\pi|\gamma)\le\liminf_{n\to\infty}
I_{[0,T]}(\pi^{n}|\gamma)$.  If $\liminf I_{[0,T]}(\pi^{n}|\gamma)$ is
equal to $\infty$, the conclusion is clear.  Therefore, we may assume
that the set $\{\pi^{n}: n\ge1\}$ is contained in $E_{q}$ for some
$q>0$. In particular, by definition of $I_{[0,T]}(\,\cdot\, |\gamma)$
and by Lemma \ref{l01}, $\pi^n (t,dx) = u^n(t,x)\, dx$ for some
$u^n \in C([0,T], \ms M_{\rm ac})$ with finite energy.

Since $u^{n}$ belongs to $C([0,T], \ms M_{\rm ac})$ and
$\pi^n(t,dx) = u^n (t,x) \,dx$ converges to $\pi(t,dx)$ in
$D([0,T], \ms M)$, $\pi(t,dx) = u(t,x)\, dx$ for some
$u\in C([0,T], \ms M_{\rm ac})$.  Moreover, by the lower
semicontinuity of the energy $\mathcal{Q}_{[0,T]}$ and by Proposition
\ref{l06},
\begin{equation*}
\mathcal{Q}_{[0,T]}(u)\;\le\;\varliminf_{n\to\infty}
\mathcal{Q}_{[0,T]}(u^{n}) \;\le\; C_0(q+1) \;<\; \infty
\end{equation*}
for some finite constant $C_0$.

\smallskip\noindent{\it Claim 1:} The sequence $\{ u^{n} : n\ge1\}$
converges to $ u$ in $\ms L^{2}([0,T]\times[0,1])$.  \\
\noindent Indeed, by the triangle inequality,
\begin{equation*}
\begin{split}
& \frac{1}{3}\, \int_{0}^{T} \  \| u_{t} -  u^{n}_{t}\|^2_{2}\ dt \\
&\quad  \;\le\; \int_{0}^{T}\ \| u_{t} -  u^{\epsilon}_{t}\|^2_{2}\ dt \;+\;
\int_{0}^{T} \ \| u_{t}^{\epsilon} -  u^{n,\epsilon}_{t}\|^2_{2}\ dt
\;+\; \int_{0}^{T} \ \| u_{t}^{n,\epsilon} -  u^{n}_{t}\|^2_{2}\ dt  \;,
\end{split}
\end{equation*}
where $u_{t}^{\epsilon}= P^{(R)}_\epsilon u_{t}$,
$u_{t}^{n,\epsilon}= P^{(R)}_\epsilon u_{t}^{n}$.  By Lemma \ref{l07}
and \eqref{6-14}, and since $\Vert u_t \Vert_\infty \le 1$, the first
and the last terms are bounded by
\begin{equation*}
C_0 \, \epsilon^{2/3} \, \int_0^T \{\,  \| u_{t}\|^2_{\mc H^1}
\,+\, \| u^n_{t}\|^2_{\mc H^1}\,\}\, dt \;\le\;
C_0 \, \epsilon^{2/3} \, \{\, q + T + 1 \,\} \;.
\end{equation*}

On the other hand,
\begin{equation*}
\int_{0}^{T} \ \| u_{t}^{\epsilon} -  u^{n,\epsilon}_{t}\|^2_{2}
\; dt  \;\le\;
\int_{0}^{T} \sum_{k\ge 1} e^{- 2 \lambda_k \epsilon}\,
\< u^n_t - u_t \,,\, f_k\>^2 \; dt \;.
\end{equation*}
As $\pi^n$ converges to $\pi$ in $D([0,T], \ms M)$, for all
$g\in C([0,1])$, $\< u^n_t - u_t \,,\, g\> \to 0$ for almost all
$t\in [0,T]$. In particular, for every $\epsilon>0$, the right-hand
side of the previous displayed equation vanishes as $n\to \infty$,
which proves Claim 1.

\smallskip\noindent{\it Claim 2:} We have that
\begin{equation}
\label{5-07}
\lim_{n\to\infty} \int_{0}^{T} \big\{\,
| \, u_{t}(0) -  u^{n}_{t}(0)\, |^2 \,+\,
| \, u_{t}(1) -  u^{n}_{t}(1) \, |^2 \,\big\} \ dt \;=\; 0 \;.
\end{equation}
We consider the boundary $x=0$, the argument for $x=1$ being
identical. The proof is similar to the one of Claim 1 and relies on
Lemma \ref{l08} instead of Lemma \ref{l07}. By the triangle
inequality, the previous integral, for $x=0$ only and divided by $3$,
is bounded by
\begin{equation*}
\int_{0}^{T}  | u_{t}(0) -  u^{\epsilon}_{t}(0)|^2  \ dt
\;+\; \int_{0}^{T}  | u^{\epsilon}_{t}(0) -  u^{n, \epsilon}_{t}(0)|^2 \ dt
\;+\; \int_{0}^{T}  | u^{n, \epsilon}_{t}(0) - u^{n}_{t}(0) |^2  \ dt
\;.
\end{equation*}
As $u_t$, $u^n_t$ are continuous for almost all $t$ [because they have
finite energy], we may repeat the argument of Claim 1, using Lemma
\ref{l08} instead of Lemma \ref{l07}, to show that the first and third
integrals in the previous equation are bounded by
$C_0 \, \epsilon^{2/5} \, \{\, q + T + 1 \,\}$.

By \eqref{6-15}, \eqref{6-11} and Schwarz inequality,
\begin{equation*}
\begin{aligned}
| u^{\epsilon}_{t}(0) -  u^{n, \epsilon}_{t}(0)|^2
\; & \le\; \sum_{k\ge 1} e^{- \lambda_k \epsilon}\,
\< u^n_t - u_t \,,\, f_k\>^2 \sum_{k\ge 1} e^{- \lambda_k \epsilon} \\
\; & =\; C_0(\epsilon) \, \sum_{k\ge 1} e^{- \lambda_k \epsilon}\,
\< u^n_t - u_t \,,\, f_k\>^2 \;.
\end{aligned}
\end{equation*}
At this point, we may repeat the arguments presented in Claim 1 to
complete the proof of Claim 2.

\smallskip
By Claims 1, 2 and \eqref{1-01b}, for any function $G$ in
$C^{1,2}([0,T]\times[0,1])$,
\begin{equation}
\label{5-08}
\lim_{n\to\infty}J_{G}(\pi^{n})\;=\;J_{G}(\pi)\;.
\end{equation}
Therefore,
$I_{[0,T]}(\pi|\gamma)\le\liminf_{n\to\infty}
I_{[0,T]}(\pi^{n}|\gamma)$, proving that $I_{[0,T]}(\,\cdot\,|\gamma)$
is lower semicontinuous.

The same argument shows that $E_{q}$ is closed in $D([0,T],\ms M)$.
By Lemma \ref{l11} below, $E_{q}$ is relatively compact in
$D([0,T],\ms M)$. Thus, $E_{q}$ is compact in $D([0,T],\ms M)$, as
claimed.
\end{proof}

The proof of the next result is similar to the one contained in the
proof of Theorem 4.2 in \cite{BLM09}.

\begin{lemma}
\label{l11}
For each $q>0$, the set $E_{q}$ is relatively compact in
$D([0,T],\ms M)$.
\end{lemma}

\begin{proof}
Fix $q>0$ and let $\pi^n$ be a sequence in $E_q$. By Lemma \ref{l01},
$\pi^n(t,dx) = u^n(t,x)\, dx$ for some
$u^n \in C([0,T], \ms M_{\rm ac})$. Since $0\le u^n(t,x) \le 1$, there
exists a subsequence, still denoted by $(u^n : n\ge 1)$, which
converges weakly in $\ms L^2 ([0,T]\times [0,1])$ to some trajectory
$u$. by the lower semicontinuity of $\mc Q_{[0,T]}$,
$\mc Q_{[0,T]}(u) <\infty$.

The proofs of Claims 1 and 2 in Theorem \ref{mt2} yield that $u^n$
converges strongly to $u$ in $\ms L^2 ([0,T]\times [0,1])$ and that
\eqref{5-07} holds. Therefore, by \eqref{5-08} and the fact that
$\pi^n$ belongs to $E_q$,
$I_{[0,T]}(\pi|\gamma)\le\liminf_{n\to\infty}
I_{[0,T]}(\pi^{n}|\gamma) \le q$. By Lemma \ref{l01}, $u^n$, $u$ are
uniformly weakly continuous in time. In particular, strong convergence
in $\ms L^2 ([0,T]\times [0,1])$ implies convergence in
$C([0,T], \ms M_{\rm ac})$.
\end{proof}

\section{Deconstructing the rate functional}
\label{sec07}

The main result of this section, stated in Proposition \ref{p01}
below, shows that the rate function $I_{[0,T]}(\,\cdot\,)$ can
be decomposed as the sum of two rate functions. The first one measures
the cost of the trajectory due to its evolution in the bulk, while the
second one measures the costs due to the boundary evolution. This
decomposition of the rate function is the main tool in the proof that
any trajectory $u$ with finite rate function can be approximated by a
sequence of regular trajectories $(u^n:n\ge 1)$ in such a way that
$I_{[0,T]}(u^n\,|\gamma) \to I_{[0,T]}(u\,|\gamma)$, the content of
the next section.

\subsection*{Weighted Sobolev spaces}

Let $\color{bblue} \Omega_T$ be the cylinder $[0,T]\times [0,1]$.  Fix
a non-negative weight $\kappa : \Omega_T \to \mathbb R_+$, and denote by
$\color{bblue} \ms L^2(\kappa)$ the Hilbert space induced by the smooth
functions in $C^\infty (\Omega_T)$ endowed with the scalar product
defined by
\begin{equation*}
\<\!\< G,H \>\!\>_{\kappa} \;=\;
\int_0^T dt \int_0^1 \kappa_t\, \, G_t \, H_t \; dx \;.
\end{equation*}
Above and hereafter, induced means that we first declare two functions
$F$, $G$ in $C^{\infty} (\Omega_T)$ to be equivalent if
$\<\!\< F-G, F-G \>\!\>_{\kappa} =0$ and then we complete the quotient
space with respect to the scalar product.

Denote by $\color{bblue} C^\infty_K (\Omega_T)$ the space of smooth
functions $H : \Omega_T \to \mathbb R$ with support contained in
$(0,T) \times (0,1)$. Let $\color{bblue} \mc H^{1} (\kappa)$,
$\color{bblue} \mc H^{1}_0 (\kappa)$ be the Hilbert spaces induced by
the sets $C^\infty (\Omega_T)$, $C^\infty_K (\Omega_T)$ endowed with
the scalar products, $\<\!\< G,H \>\!\>_{1,2,\kappa}$,
$\<\!\< G,H \>\!\>_{1,\kappa}$, respectively defined by
\begin{equation*}
\begin{gathered}
\<\!\< G,H \>\!\>_{1,2,\kappa} \;=\;
\<\!\< G,H \>\!\>_{\kappa} \;+\;
\<\!\< \nabla  G, \nabla  H \>\!\>_{\kappa} \;, \\
\<\!\< G,H \>\!\>_{1,\kappa} \;=\;
\<\!\< \nabla  G, \nabla  H \>\!\>_{\kappa}  \;.
\end{gathered}
\end{equation*}
The Poincar\'e's inequality yields that the norms induced by the
scalar products $\<\!\< G,H \>\!\>_{1,2,\kappa}$,
$\<\!\< G,H \>\!\>_{1,\kappa}$ are equivalent in
$\mc H^{1}_0 (\kappa)$.

Denote by $\color{bblue} \Vert \cdot \Vert_{\kappa}$,
$\color{bblue} \Vert \cdot \Vert_{1, \kappa}$ the norm associated to
the scalar product $\<\!\<\cdot, \cdot \>\!\>_{\kappa}$,
$\<\!\<\cdot, \cdot \>\!\>_{1,\kappa}$, respectively.  Let
$\color{bblue} \mc H^{-1} (\kappa)$ be the dual of
$\mc H^{1}_0 (\kappa)$; it is a Hilbert space equipped with the norm
$\Vert \cdot \Vert_{-1, \kappa}$ defined by
\begin{equation}
\label{9-0}
\|L\|^2_{-1,\kappa} \;=\; \sup_{G\in C_{K}^\infty (\Omega_T)}
\big\{ \, 2\,  L(G) \;-\;
\Vert G \Vert^2_{1, \kappa} \, \big \}\;.
\end{equation}
By Riesz' representation theorem, an element $L$ of
$\mc H^{-1} (\kappa)$ can be written as
$L(H) =$\break $ \<\!\< \nabla  G \,,\, \nabla  H \>\!\> _{\kappa}$ for some $G$
in $\mc H^{1}_0 (\kappa)$.

When $\kappa \equiv 1$, we represent $\ms L^2(\kappa)$,
$\mc H^{1} (\kappa)$, $\mc H^{1}_0 (\kappa)$, $\mc H^{-1} (\kappa)$ as
$\color{bblue} \ms L^2(\Omega_T)$,
$\color{bblue} \mc H^{1}(\Omega_T)$,
$\color{bblue} \mc H^{1}_0(\Omega_T)$
$\color{bblue} \mc H^{-1} (\Omega_T)$, respectively.  Next result is
\cite[Lemma 4.8]{BLM09}. It states that $\mc H^{-1} (\kappa)$ is
formally the space $\{\nabla  P : P\in \ms L^2 (\kappa^{-1})\}$. For an
integrable function $H: [0,1] \to\mathbb R$, let
$\color{bblue} \<H\> = \int_{0}^1 H(x) \; dx$.

\begin{lemma}
\label{l20}
A linear functional $L: \mc H^1_0(\kappa)\to\mathbb R$ belongs to
$\mc H^{-1} (\kappa)$ if, and only if, there exists $P$ in
$\ms L^2 (\kappa^{-1})$ such that
$L(H) = \int_0^T dt \int_0^1 P_t \, \nabla  H_t dx$ for every $H$ in
$C^\infty_K(\Omega_T)$. In this case,
\begin{equation*}
\Vert L\Vert^2_{-1,\kappa}
\;=\; \int_0^T \, \big\{\,
\< P_t , P_t\>_{\kappa(t)^{-1}} \;-\; c_t \,\big\}\; dt \;,
\end{equation*}
where $c_t = \{\, \< P_t /\kappa_t \>^2 \,/\,  \<  1/ \kappa_t\>\,\}
\, \mb 1\{\< 1/\kappa_t\> <\infty\}$.
\end{lemma}

\subsection*{Representation theorems}

Until the end of this section, {\color{bblue}
$\pi(t,dx) = u(t,x)\, dx$ is a path in
$D_{\mc E}([0,T], \ms M_{\rm ac})$. We assume that $u$ is continuous
on $\Omega_T$ and smooth in time, there exists $\epsilon>0$ such that
$\epsilon \le u(t,x) \le 1-\epsilon$ for all $(t,x) \in \Omega_T$, and
$I_{[0,T]} (u)<\infty$}. These conditions are fulfilled in sets of the
form $[\delta, T]\times [0,1]$, $\delta>0$, by paths in $\Pi_3$, a
class of trajectories to be introduced in Section \ref{sec05}. As $u$
is bounded away from $0$ and $1$, the spaces $\ms L^2(\sigma(u))$ and
$\ms L^2(\Omega_T)$ coincide, as well as, the other Hilbert spaces
introduced in the previous subsection with $\kappa = \sigma(u)$.

Denote by $\mf W \colon C^{0,1}(\Omega_T) \to \mathbb R$ the
functional given by
\begin{equation*}
\mf W (H) \; =\; \int_0^T dt \int_0^1 \sigma(u_t)\, |\nabla  H_t|^2
\, dx
\;+\; \int_0^T \Psi (t, H_t(0), H_t(1)) \; dt \;,
\end{equation*}
where
\begin{equation*}
\Psi (t,M, N) \;=\; \mf b_{\alpha, A} (u_t(0), M)
\;+\; \mf b_{\beta, B} (u_t(1), N) \;,
\end{equation*}
and $\mf b_{\varrho, D} (a,M)$ has been introduced in
\eqref{4-05}. For each $0\le t\le T$, $(M,N) \mapsto \Psi (t,M, N)$ is
a smooth, convex function which takes negative values.
\smallskip

Fix a linear functional $L: C^{0,1}(\Omega_T) \to \mathbb R$.  Denote
by $L_0$ its restriction to $C^{0,1}_0(\Omega_T)$:
\begin{equation}
\label{9-09}
L_0(H) \;=\; L(H)\;, \quad H \,\in\,  C^{0,1}_0(\Omega_T)\;,
\end{equation}
where
\begin{equation*}
C^{0,1}_0(\Omega_T) \;:=\; \big\{\, H \in
C^{0,1}(\Omega_T) : H(t,0) = H(t,1) = 0 \,,\, 0\le t\le T \,\big\}
\;.
\end{equation*}

Let $\Xi: \Omega_T \to \mathbb R$ the function given by
\begin{equation}
\label{9-02}
\Xi(t,x) \;=\; \frac{1}
{\int_0^1 1/\sigma(u(t,y)) \; dy} \,
\int_0^x \frac{1}{\sigma(u(t,y))}\; dy \;.
\end{equation}
Note that $\Xi$ belongs to $C^{\infty, 1}(\Omega_T)$, and that
$\Xi(t,0)=0$, $\Xi(t,1)=1$ for all $0\le t\le T$.  Let $\ell^{(0)}$,
$\ell^{(1)} \colon C([0,T]) \to \mathbb R$ be the linear functionals given
by
\begin{equation}
\label{9-04}
\ell^{(0)} (h) \;=\; L \big(\, h(t) \, [1-\Xi(t,x)]\,\big)\;, \quad
\ell^{(1)} (h) \;=\; L \big(\, h(t) \, \Xi(t,x)  \,\big)\;.
\end{equation}
Note that the right-hand sides of the previous identities are well
defined because $\Xi$ belongs to $C^{0,1}(\Omega_T)$.

\smallskip\noindent{\it Note:} The definition of $\ell^0$, $\ell^1$
explains why we defined $L$ in $C^{0,1}(\Omega_T)$ and not in
$C^{\infty, \infty}(\Omega_T)$. For
$L \big(\, h(t) \, [1-\Xi(t,x)]\,\big)$ to make sense, we need the map
$(t,x) \mapsto h(t) \, [1-\Xi(t,x)]$ to belong to the domain of
definition of $L$. \smallskip

Decompose a function $H\colon \Omega_T \to \bb R$ as
$H = H^{(0)} + H^{(1)}$, where
\begin{equation}
\label{9-14}
H^{(1)}(t,x) \;=\; H(t,0) \;+\; [\, H(t,1) - H(t,0)\,] \; \Xi(t,x)
\;.
\end{equation}
Note that $H^{(0)}(t,0) = H^{(0)}(t,1) =0$ for all $0\le t\le T$. In
particular, $H^{(0)}$ belongs to $C^{0,1}_0(\Omega_T)$ so that
$L_0(H^{(0)})$ is well defined and $L_0(H^{(0)}) = L(H^{(0)})$.

By linearity and the previous paragraph,
$L(H) = L_0(H^{(0)}) + L(H^{(1)})$. By definition of $H^{(1)}$,
$\ell^0$, $\ell^1$,
$L(H^{(1)}) = L(\, H(t,0) \, [\, 1- \Xi\,]\,) \,+\, L(\, H(t,1) \, \Xi
\,) \;=\; \ell^0 (H(\cdot ,0)) \,+\, \ell^1 (H(\cdot ,1))$, Hence, for
all $H$ in $C^{0,1}(\Omega_T)$,
\begin{equation*}
L(H) \;=\; L_0(H^{(0)}) \;+\;
\ell^0 (H(\cdot ,0)) \;+\, \ell^1 (H(\cdot ,1))\;.
\end{equation*}

\begin{lemma}
\label{l18}
Let $L: C^{0,1}(\Omega_T) \to \mathbb R$ be a linear functional.
Then,
\begin{equation}
\label{9-01}
\sup_{H \in C^{0,1}(\Omega_T)} \big\{ L(H) \,-\, \mf W(H)\, \big\}
\;=\; S_1 \;+\; S_2\;,
\end{equation}
where
\begin{equation}
\label{9-03}
S_1 \;=\; \sup_{G \in C^{0,1}_0(\Omega_T)} \big\{ \,
L_0(G) \;-\; \int_0^T dt \int_0^1 \sigma(u_t)\,
|\nabla  G_t|^2 \, dx \, \big\} \;,
\end{equation}
and
\begin{equation*}
S_2\;=\; \sup_{h, g \in C ([0,T])}
\big\{ \,  \ell(g,h)
\;-\; \int_0^T \zeta_t\, [\, h(t) - g(t)\,]^2 \; dt
\,-\, \int_0^T  \Psi (t,g_t,h_t) \; dt \, \big\}\;.
\end{equation*}
In this formula, $\color{bblue} \zeta_t = 1/\<1/\sigma(u_t)\>$ and
$\color{bblue} \ell(g,h) = \ell^{(0)} (g) \,+\, \ell^{(1)} (h)$.
\end{lemma}

The first variational problem concerns the interior of $\Omega_T$,
while the second one the boundary of the cylinder $\Omega_T$.

\begin{proof}[Proof of Lemma \ref{l18}.]
Fix a linear functional $L: C^{0,1}(\Omega_T) \to \mathbb R$.  Write
$H = H^{(0)} + H^{(1)}$, as in \eqref{9-14}. Since $H^{(0)}$ belongs
to $C^{0,1}_0(\Omega_T)$, $L_0(H^{(0)})$ is well defined and
$L_0(H^{(0)}) = L(H^{(0)})$.

By linearity, $L(H) = L_0(H^{(0)}) + L(H^{(1)})$. On the other hand,
an elementary computation yields that $\nabla  H^{(0)}$ and
$\nabla  H^{(1)}$ are orthogonal in $\ms L^2(\sigma(u))$:
\begin{equation*}
\int_0^T dt \int_0^1 \sigma(u_t)\, \nabla  H^{(0)}_t\,
\nabla  H^{(1)}_t\; dx \;=\; 0\;.
\end{equation*}
Therefore, the supremum appearing in \eqref{9-01} can be written as
\begin{equation*}
\begin{aligned}
& \sup_{H \in C^{0,1}(\Omega_T)} \Big\{ \,
L_0(H^{(0)}) \;-\; \int_0^T dt \int_0^1 \sigma(u_t)\,
|\nabla  H^{(0)}_t|^2 \, dx \\
& \quad \,+\; L(H^{(1)})
\;-\; \int_0^T \zeta_t\, [\, H(t,1) - H(t,0)\,]^2 \; dt
\,-\, \int_0^T \Psi (t, H_t(0), H_t(1))  \; dt
\, \Big\}\;.
\end{aligned}
\end{equation*}
The first line depends only on $H^{(0)}$, while the second one only on
$H_t(0)$, $H_t(1)$. We may, therefore, split the supremum in two
pieces. Recall the definition of the functionals $\ell$
to rewrite the previous supremum as
\begin{equation*}
\begin{aligned}
& \sup_{G \in C^{0,1}_0(\Omega_T)} \big\{ \,
L_0(G) \;-\; \int_0^T dt \int_0^1 \sigma(u_t)\,
|\nabla  G_t|^2 \, dx \, \big\} \\
& +\;
\sup_{h, g \in C([0,T])}
\big\{ \,  \ell (g,h)
\;-\; \int_0^T \zeta_t\, [\, h(t) - g(t)\,]^2 \; dt
\,-\, \int_0^T \Psi (t,g_t,h_t) \; dt \, \big\}\;,
\end{aligned}
\end{equation*}
as claimed.
\end{proof}

We apply Lemma \ref{l18} to the linear functionals appearing in the
definition of the rate functional $I_{[0,T]}(\,\cdot\,)$.
Denote $L^{(\partial_t)}$, $L^{(\nabla )}: C^{0,1}(\Omega_T) \to \mathbb R$
the linear functionals given by
\begin{equation}
\label{9-08}
L^{(\partial_t)} (G) \;=\; \int_{0}^{T} \langle \, \partial_{t} u_{t},
G_{t} \, \rangle\; dt \;, \quad L^{(\nabla )} (H) \;=\; \int_{0}^{T}
\langle\, \nabla  u_{t}, \nabla  H_{t} \, \rangle\; dt \;,
\end{equation}
and let $\color{bblue} \mf L = L^{(\partial_t)} + L^{(\nabla )}$.
Denote by $\mf L_0$, $\mf l^0$, $\mf l^1$ the linear functionals
associated to $\mf L$ by \eqref{9-09}, \eqref{9-04}, so that
\begin{equation*}
\begin{gathered}
\mf L_0 (G) \;=\;
\int_{0}^{T} \langle \, \partial_{t} u_{t},
G_{t} \, \rangle\; dt \;+\; \int_{0}^{T}
\langle\, \nabla  u_{t}, \nabla  G_{t} \, \rangle\; dt \;,
\\
\mf l^0 (g) \;=\; \int_0^T a(t) \, g(t) \; dt \;,
\quad
\mf l^1 (g) \;=\; \int_0^T b(t)\,\, g(t) \; dt \;,
\end{gathered}
\end{equation*}
where
\begin{equation}
\label{9-10}
a(t) \;=\; \< \, \partial_{t}  u_{t}, [1-\Xi_{t}] \, \rangle
\;-\;
\< \, \nabla   u_{t}, \nabla  \Xi_{t} \, \rangle \;, \;\;
b(t) \;=\; \< \, \partial_{t}  u_{t} \,,\, \Xi_{t} \, \rangle
\;+\;
\< \, \nabla   u_{t}, \nabla  \Xi_{t} \, \rangle \;.
\end{equation}
With this notation,
\begin{equation}
\label{9-16}
\mf L (H) \;=\; \mf L_0 (H^{(0)}) \;+\;
\mf l^0 (H(\cdot, 0)) \;+\; \mf l^1 (H(\cdot, 1))\;.
\end{equation}

Denote by $\Upsilon_t \colon \mathbb R^2 \to \mathbb R$, $0\le t\le T$, the
strictly convex map defined by
\begin{equation*}
\Upsilon_t (x,y) \;=\; \zeta_t\, [\, x - y \,]^2
\,+\, \mf b_{\alpha, A}
(u_t(0), x) \,+\, \mf b_{\beta, B} (u_t(1), y)\;,
\end{equation*}
and let $\Phi_t: \mathbb R^2 \to \mathbb R$, $t\ge 0$ be its Legendre
transform:
\begin{equation}
\label{9-15}
\Phi_t(a,b) \;=\;
\sup_{x,y\in \mathbb R} \big\{ \,  a\, x \,+\, b\, y
\;-\; \Upsilon_t (x,y) \,\big\} \;.
\end{equation}

\begin{lemma}
\label{l21}
Under the hypotheses stated at the beginning of this subsection,
\begin{equation*}
I_{[0,T]}(u) \;=\;
I^{(1)}_{[0,T]}(u) \;+\; I^{(2)}_{[0,T]}(u) \;,
\end{equation*}
where
\begin{equation*}
I^{(1)}_{[0,T]}(u) \;=\;
\frac{1}{4}\, \Vert\, \mf L_0  \,\Vert^2_{-1, \sigma(u)} \;, \quad
I^{(2)}_{[0,T]}(u) \;=\;
\int_0^T \Phi_t (a_t, b_t) \; dt \;.
\end{equation*}
\end{lemma}

\begin{proof}
By definition of the rate functional $I_{[0,T]}$, given in
\eqref{1-02},
\begin{equation*}
I_{[0,T]}(u) \;=\; \sup_{H \in C^{0,1}(\Omega_T)}
\big\{ \mf L (H) \,-\, \mf W(H)\, \big\}\;.
\end{equation*}
Hence, by Lemma \ref{l18}, \eqref{9-0} and the definition of
$\mf l^0$, $\mf l^1$, given above \eqref{9-10},
\begin{equation*}
I_{[0,T]}(u) \;=\; I^{(1)}_{[0,T]}(u)
\;+\; I^{(2)}_{[0,T]}(u) \;,
\end{equation*}
where
\begin{equation}
\label{9-19}
\begin{gathered}
I^{(1)}_{[0,T]}(u) \;=\;
\frac{1}{4}\, \Vert\, \mf L_0 \,\Vert^2_{-1, \sigma(u)}
\;, \\
I^{(2)}_{[0,T]}(u) \;=\;
\sup_{h, g \in C ([0,T])}
\big\{ \,  \mf l (g,h) \;-\;
\int_0^T \Upsilon_t (g_t,h_t) \; dt \, \big\}
\end{gathered}
\end{equation}
and $\mf l (g,h) = \mf l^0 (g) + \mf l^1 (h)$.  The second term can
be written as
\begin{equation*}
\begin{aligned}
& \sup_{h, g \in C ([0,T])}
\int_0^T \big \{\, a(t)\, g(t) \,+\,
b(t) \, h(t) \;-\;   \Upsilon_t (g_t, h_t) \, \big\} \; dt \,  \\
& \quad =\;
\int_0^T  \sup_{x, y\in \mathbb R} \big\{\, a(t)\, x \,+\,
b(t) \, y  \;-\;    \Upsilon_t (x, y) \, \big\} \; dt
\;=\; \int_0^T \Phi_t (a_t, b_t) \; dt \;.
\end{aligned}
\end{equation*}
This completes the proof of the lemma.
\end{proof}

The function $\Phi_t$ is convex and continuous.  Moreover,
$\Phi_t(a,b) \ge 0$ [take $x=y=0$ in the supremum] and
$\Phi_t(a,b) \le \Phi^0_t(a) + \Phi^1_t(b)$, where $\Phi^0_t$,
$\Phi^1_t$ are the Legendre transform of
$\mf b_{\alpha, A} (u_t(0), \cdot)$,
$\mf b_{\beta, B} (u_t(1), \,\cdot\,)$, respectively:
\begin{equation*}
0\;\le\; \Phi_t(a,b) \;\le\; \Phi^0_{u_t(0)}(a)
\;+\; \Phi^1_{u_t(1)}(b)\;,
\end{equation*}
where
\begin{equation*}
\Phi^0_u(a) \;=\; a\, \ln \Big\{ \frac{\sqrt{a^2
+ 4 \, \mf f_{0,u} \, \mf g_{0,u}} \,+\, a}
{2\, \mf f_{0,u}} \,\Big\} \,-\,
\sqrt{a^2 + 4 \, \mf f_{0,u}\, \mf g_{0,u}}
\,+\, \mf f_{0,u} \,+\, \mf g_{0,u}\;,
\end{equation*}
$\mf f_{0,u} = (1/A) \, [1-u]\, \alpha$,
$\mf g_{0,u} = (1/A) \, u\, [1-\alpha]$. The formula for $\Phi^1_u$ is
similar. One just needs to replace $A$, $\alpha$ by $B$, $\beta$,
respectively. In particular,
\begin{equation}
\label{9-11}
0\;\le\; \Phi_t(a,b) \;\le\; C_0 \, \{ 1
\,+\, |a| \,\ln^+ |a| \,+\, |b| \,\ln^+ |b| \,\} \;,
\end{equation}
where $\ln^+ x =0 $ for $0<x\le 1$ and $\ln^+ x = \ln x$ for $x\ge 1$.

\smallskip\noindent{\it Note:} It might be disconcerting that
$\Phi_t(0,0)$ is not equal to $0$. This is a consequence of the fact
that $\mf b_{\beta, B} (a, \cdot)$ takes negative values. To remedy,
one can add and subtract a linear term to
$\mf b_{\beta, B} (a, \cdot)$, transforming
$\mf b_{\beta, B} (a, \cdot)$ into $\mf q_{\beta, B} (a, \cdot)$,
given by \eqref{6-04}.  In constrast with
$\mf b_{\beta, B} (a, \cdot)$, $\mf q_{\beta, B} (a, \cdot)$ is
nonnegative and attains its minimum at $0$.

After these modifications, $\Phi_t$ becomes
\begin{equation*}
\Phi_t (a ,b)\;=\;
\widehat \Phi_t
\Big(\, a \,+\, \frac{1}{A}\, [\alpha - u_t(0)]
\,,\, b \,+\, \frac{1}{B}\, [\beta - u_t(1)] \, \Big)\;,
\end{equation*}
where
\begin{equation*}
\widehat \Phi_t(a,b) \;=\;
\sup_{x,y\in \mathbb R}
\Big\{ \,  a\, x \,+\, b\, y
\;-\; \zeta_t\, [\, x - y \,]^2  \,-\, \mf q_{\alpha, A} (u_t(0), x)
\,-\, \mf q_{\beta, B} (u_t(1), y) \,\Big\} \;,
\end{equation*}
and $\widehat \Phi_t(a,b) \ge \widehat \Phi_t(0,0) =0$.
\smallskip

Both functions $\Phi_t$ and $\widehat \Phi_t$ are convex and
continuous.  As $\Phi_t$, $\widehat \Phi_t$ depend on the trajectory
$u$, whenever we wish to stress this dependence, we represent
$\Phi_t(a,b)$, $\widehat \Phi_t(a,b)$, by
$\color{bblue} \Phi^{(u)}_t(a,b)$,
$\color{bblue} \widehat \Phi^{(u)}_t(a,b)$, respectively.

\smallskip

Lemma \ref{l21} decomposes the rate function as the sum of two
independent functionals. The first piece can still be simplified.
This is the content of the next result.  Under the
hypotheses of this subsection,
$\Vert\, L^{(\nabla )}_0 \,\Vert^2_{-1, \sigma(u)} < \infty$. Since
$\int_0^T \Phi_t (a_t, b_t) \; dt$ is finite as well, it follows from
the previous lemma that
\begin{equation}
\label{9-06}
I_{[0,T]}(u) \;<\; \infty
\;\;\text{if, and only if,}\;\;
\Vert\, L^{(\partial_t)}_0  \,\Vert^2_{-1,\sigma(u)}
\;<\; \infty\;.
\end{equation}

Suppose that $L^{(\partial_t)}_0$ belongs to $\mc H^{-1}(\sigma(u))$.
By Lemma \ref{l20}, there exists $P$ in $\ms L^2(\sigma(u)^{-1})$ such
that
\begin{equation*}
L^{(\partial_t)}_0 (H) \;=\; \int_0^T \<\, P_t \,,\, \nabla  H_t\,\> \; dt
\end{equation*}
for all $H$ in $C^{\infty}_K(\Omega_T)$.  This identity extends to
$C^{0,1}_0(\Omega_T)$. Since $H_t$ vanishes at the boundary $x=0$,
$x=1$, the same identity holds if we replace $P_t$ by $P_t-c_t$ for
some function $c$ in $\mc L^1([0,T])$. By choosing the right constant
[that is $c_t = \<P_t/\sigma(u_t)\>/\<1/\sigma(u_t)\>$], we may assume
that $\<P_t/\sigma(u_t)\>=0$ for almost all $0\le t\le T$. We denote
by $M$ the element of $\ms L^2(\sigma(u)^{-1})$ satisfying this
condition and the previous displayed equation:
\begin{equation}
\label{9-07}
L^{(\partial_t)}_0 (H) \;=\;
\int_0^T \<\, M_s \,,\, \nabla  H_s\,\> \; ds
\;, \quad \int_0^1
\frac{M_t}{\sigma(u_t)}\; dx \;=\; 0
\end{equation}
for all $H\in C^{0,1}_0(\Omega_T)$ and almost all $0\le t\le
T$. Moreover, as $\<M_t/\sigma(u_t)\>=0$ for almost all $t$, by Lemma
\ref{l20},
\begin{equation*}
\Vert\, L^{(\partial_t)}_0  \,\Vert^2_{-1,\sigma(u)}
\;=\; \int_0^T \, dt \int_0^1 \frac{M_t^2}{\sigma(u_t)} \; dx \;.
\end{equation*}

\begin{lemma}
\label{l22}
Fix a trajectory satisfying the hypotheses stated at the beginning of
this section.  Then,
\begin{equation*}
I^{(1)}_{[0,T]}(u) \;=\; \frac{1}{4}\, \int_0^T \big\{\,
\Vert\, M_t \,+\, \nabla  u_t\,\Vert^2_{\sigma(u_t)^{-1}}
\,-\, R_t\,\big\} \; dt \;,
\end{equation*}
where
\begin{equation*}
R_t \;=\; \big \<\, \frac{\nabla  u_t}{\sigma(u_t)}\,\big\>^2
\, \frac{1}{\<\sigma(u_t)^{-1}\>}
\;=\; \Big\{ \, \log \frac{u_t(1)}{1-u_t(1)}  \,-\,
\log \frac{u_t(0)}{1-u_t(0)} \,\Big\}^2
\, \frac{1}{\<\sigma(u_t)^{-1}\>}\;\cdot
\end{equation*}
\end{lemma}

\begin{proof}
As $I_{[0,T]}(u)$ is finite, by \eqref{9-06} and the paragraph
preceding the statement of the lemma, $L^{(\partial_t)}_0$ belongs to
$\mc H^{-1}(\sigma(u))$ and there exists $M$ in
$\ms L^2(\sigma(u)^{-1})$ satisfying \eqref{9-07}.  Therefore, for all
$H$ in $C^{0,1}_0(\Omega_T)$,
\begin{equation*}
\mf L_0 (H) \;=\; L^{(\nabla )}_0 (H) \;+\; L^{(\partial_t)}_0 (H)
\;=\; \int_0^T \<\, M_t
\,+\, \nabla  u_t \,,\, \nabla  H_t\,\> \; dt\;.
\end{equation*}
By Lemma \ref{l20},
\begin{equation*}
\Vert\, \mf L_0 \,\Vert^2_{-1,\sigma(u)}
\;=\; \int_0^T
\Big\{\, \Vert\, M_t \,+\, \nabla  u_t\,\Vert^2_{\sigma(u_t)^{-1}}
\,-\, R_t\, \,\big\}\; dt \;,
\end{equation*}
where $R_t$ has been introduced in the statement of the lemma. This
completes the proof of the lemma.
\end{proof}

We summarize the last two results in the next proposition.

\begin{proposition}
\label{p01}
Fix a path $\pi(t,dx) = u(t,x)\, dx$ in $D([0,T], \ms M_{\rm
ac})$. Assume that $u$ is continuous on $\Omega_T$ and smooth in time,
that there exists $\epsilon>0$ such that
$\epsilon \le u(t,x) \le 1-\epsilon$ for all $(t,x) \in \Omega_T$, and
that $I_{[0,T]} (u)<\infty$.  Then,
\begin{equation*}
I_{[0,T]}(u) \;=\;
I^{(1)}_{[0,T]}(u) \;+\; I^{(2)}_{[0,T]}(u) \;,
\end{equation*}
where Lemma \ref{l22} provides a formula for first term  and Lemma
\ref{l21} for the second.
\end{proposition}

\begin{remark}
\label{rm1}
In the statement of Proposition \ref{p01}, we imposed many regularity
assumptions on $u$ because this is the context in which this result is
applied in the next section. The proof shows that they can be relaxed.
\end{remark}

\section{$I_{[0,T]}(\,\cdot\,)$-density}
\label{sec05}

In this section, we prove that any trajectory $\pi\in D([0,T], \ms M)$
with finite rate function can be approximated by a sequence of smooth
trajectories $\{\pi^{n}: n\ge1\}$ such that
\begin{equation*}
\pi^{n}\longrightarrow\pi
\;\; \text{ and }\;\; 
I_{[0,T]}(\pi^{n}|\gamma)\longrightarrow I_{[0,T]}(\pi|\gamma)\;.
\end{equation*}
We follow an approach proposed in \cite{qrv, BLM09, FLM2011}. Here,
and throughout this section, $\gamma\colon [0,1]\to [0,1]$ is a fixed
density profile.  We first introduce some terminology.

\begin{definition}
\label{d05}
A subset $A$ of $D([0,T],\ms M)$ is said to be
$I_{[0,T]}(\cdot|\gamma)$-dense if for any $\pi$ in $D([0,T],\ms M)$
such that $I_{[0,T]}(\pi|\gamma)<\infty$, there exists a sequence
$\{\pi^{n}:n\ge1\}$ in $A$ such that $\pi^{n}$ converges to $\pi$ in
$D([0,T],\ms M)$ and $I_{[0,T]}(\pi^{n}|\gamma)$ converges to
$I_{[0,T]}(\pi|\gamma)$.
\end{definition}

\begin{theorem}
\label{mt3} 
For all $\gamma: [0,1] \to [0,1]$, the set $\Pi_\gamma$ is
$I_{[0,T]}(\cdot|\gamma)$-dense. If there exists $\epsilon_0 >0$ such
that $\epsilon_0 \le \gamma \le 1-\epsilon_0$, condition (b) in
Definition \ref{d04} can be replaced by the existence of $\epsilon >0$
such that $\epsilon \le u(t,x) \le 1-\epsilon$ for all
$(t,x) \in [0,T] \times [0,1]$.
\end{theorem}

The proof of Theorem \ref{mt3} is divided into several steps.
Throughout this section, denote by
$\color{bblue} u^{(\gamma)}:[0,T]\times [0,1] \to [0,1]$ the unique
weak solution of the boundary-initial valued problem \eqref{1-06} with
initial profile $u_0 = \gamma$.

Let $\color{bblue} \Pi_{1}$ be the set of all paths
$\pi(t,dx)=u(t,x)\, dx$ in $D_{\mc E}([0,T],\ms M_{\rm ac})$, whose
density $u$ is a weak solution of the Cauchy problem \eqref{1-06} in
some positive time interval. In other words, there exists $\delta>0$
such that $u_t = u^{(\gamma)}_t$ for $0\le t\le \delta$.

\begin{lemma}
\label{pi1}
The set $\Pi_{1}$ is $I_{[0,T]}(\cdot|\gamma)$-dense.
\end{lemma}

\begin{proof}
Fix $\pi$ in $D([0,T],\ms M)$ such that
$I_{[0,T]}(\pi|\gamma)<\infty$. By definition of the rate function,
$\pi$ belongs to $D([0,T],\ms M_{\rm ac})$, $\pi(t,dx) = u(t,x)\, dx$,
and $\mathcal{Q}_{[0,T]} (u) < \infty$.

For each $\delta>0$, consider the path
$\pi^{\delta}(t,dx) = u^{\delta}(t,x)\, dx$ defined by
\begin{align*}
u^{\delta}(t,x) \;=\; 
\begin{cases}
u^{(\gamma)} (t,x) & \text{ if } t\in[0,\delta]\;, \\
u^{(\gamma)}(2\delta-t,x) & \text{ if } t\in[\delta,2\delta]\;, \\
u(t-2\delta,x) & \text{ if } t\in[2\delta,T]\;. 
\end{cases}
\end{align*}

\noindent{\it Claim A:} The trajectory $\pi^{\delta}$ belongs to
$\Pi_{1}$.  Indeed, by definition, $u^{\delta}$ is the weak solution
of the Cauchy problem \eqref{1-06} in the time-interval
$[0,\delta]$. On the other hand, by definition of $u^\delta$,
$\mathcal{Q}_{[0,T]} (u^{\delta}) \,\le\, 2\, \mathcal{Q}_{[0,\delta]}
(u^{(\gamma)}) \,+\, \mathcal{Q}_{[0,T]} (u)$.  By Corollary
\ref{zero}, $\mathcal{Q}_{[0,\delta]}(u^{(\gamma)})<\infty$. On the
other hand, $\mathcal{Q}_{[0,T]} (u)$ is finite because
$I_{[0,T]}(\pi|\gamma)<\infty$. Therefore, $u^\delta$ has finite
energy, which completes the proof of Claim A.

It is clear that $\pi^{\delta}$ converges to $\pi$ in $D([0,T],\ms M)$
as $\delta \downarrow 0$. To conclude the proof of the lemma it is
enough to show that $I_{[0,T]}(\pi^{\delta}|\gamma)$ converges to
$I_{[0,T]}(\pi|\gamma)$ as $\delta \downarrow 0$.

Since the rate function is lower semicontinuous,
$I_{[0,T]}(\pi|\gamma)\le\liminf_{\delta\to0}I_{[0,T]}(\pi^{\delta}|\gamma)$.
To prove that
$\limsup_{\delta\to0}I_{[0,T]}(\pi^{\delta}|\gamma)\le
I_{[0,T]}(\pi|\gamma)$, decompose the rate function
$I_{[0,T]}(\pi^{\delta}|\gamma)$ into the sum of the contributions on
each time interval $[0,\delta]$, $[\delta, 2\delta]$ and
$[2\delta, T]$.

Recall the notation introduced at the beginning of Section
\ref{sec03}. By \eqref{4-02}, \eqref{4-04}, since in the interval
$[2\delta, T]$ $\pi^{\delta}$ is a time translation of the path
$\pi$,
\begin{equation*}
I_{[0,T]}(\pi^{\delta}|\gamma) \;\le\;
I_{[0,\delta]}(\pi^{\delta}|\gamma) \;+\;
I_{[0,\delta]} (\tau_\delta u^{\delta} \, |\, u^{(\gamma)}_\delta) \;+\;
I_{[0,T]}(\pi|\gamma)\;.
\end{equation*}
Since the density $u^{\delta}$ is a weak solution of the equation
\eqref{1-06} on the interval $[0,\delta]$, by Corollary \ref{zero},
the first contribution is equal to $0$. It remains to show that the
second term on the right-hand side vanishes as $\delta\to 0$.

Let $v^\delta = \tau_\delta u^{\delta}$.  As
$v^\delta (t) = u^{(\gamma)} (\delta-t)$, the density $v^{\delta}$
solves the backward heat equation:
$\partial_{t}v^{\delta} \,=\, -\, \Delta v^{\delta}$. Thus, by
Definition \ref{d01} and \eqref{1-01b}, for each $H$ in
$C^{1,2}([0,T]\times[0,1])$,
\begin{equation*}
\begin{split}
J_{\delta,H} (v^\delta) \; & =\;
\int_0^{\delta} \big\{\, 2\, \langle \, \nabla
u^{(\gamma)}_{t} \,,\, \nabla H_{t} \, \rangle
\,-\,\langle \, \sigma(u^{(\gamma)}_{t})
\,,\, (\nabla H_{t})^{2} \, \rangle \big\} \; dt \\
& -\; \int_0^T \Big\{\, 
\widehat {\mf b}_{\alpha, A}
\big(\, u^{(\gamma)}_t(0)\, ,\, H_t (0)\, \big)
\;+\; \widehat{\mf b}_{\beta, B}
\big(\, u^{(\gamma)}_t(1)\, ,\, H_t(1)\, \big)\,
\Big\}\, dt \; ,
\end{split}
\end{equation*}
where
\begin{equation*}
\widehat {\mf b}_{\varrho, D} (a,M) \;=\;
\frac{1}{D}\, \Big\{ [1-a]\, \varrho \,
[e^M-1 + M] \;+\; a\, [1-\varrho] \, [e^{-M}-1 - M] \, \Big\}  \;.
\end{equation*}
By Schwarz inequality, the first integral on the right-hand side is
bounded above by
\begin{equation*}
\int_{0}^{\delta} dt\ \int_0^1
\frac{|\nabla u^{(\gamma)} (t,x)|^{2}}
{\sigma(u^{(\gamma)}(t,x))} \, dx \;.
\end{equation*}
By \eqref{n06}, this expression vanishes as $\delta\to 0$.  On the
other hand, maximizing $\widehat {\mf b}_{\varrho, D} (a,M)$ over $M$
yields that the second integral is bounded above by
\begin{equation*}
\begin{aligned}
& \int_{0}^{\delta} \Big\{\,
\frac{1}{B} \, [\, \beta - u^{(\gamma)}_t(1)\,]
\, \log \frac {[1-u^{(\gamma)}_t(1)] \, \beta}  
{u^{(\gamma)}_t(1) \, [1-\beta]} \,+\,
\frac{1}{A} \, [\, \alpha - u^{(\gamma)}_t(0)\,]
\, \log \frac {[1-u^{(\gamma)}_t(0)] \, \alpha}  
{u^{(\gamma)}_t(0) \, [1-\alpha]}\, \Big\}\, dt \\
&\le\;
\int_{0}^{\delta} \Big | \, \frac{\beta - u^{(\gamma)}_t(1)}{B} 
\, \log \frac {[1-u^{(\gamma)}_t(1)]}  
{u^{(\gamma)}_t(1)} \,+\,
\frac{\alpha - u^{(\gamma)}_t(0)}{A} 
\, \log \frac {[1-u^{(\gamma)}_t(0)]}  
{u^{(\gamma)}_t(0)}\, \Big | \, dt
\;+\; C_0 \delta
\end{aligned}
\end{equation*}
for some finite constant $C_0 = C_0 (\alpha, \beta, A, B)$. By
\eqref{n06}, this expression vanishes as $\delta\to 0$.

Putting together the previous estimates shows that there exists a
function $c(\delta)$, independent of $H$, such that
$\lim_{\delta\to 0} c(\delta) =0$ and
\begin{equation*}
J_{\delta,H} (\tau_\delta u ^\delta) \; \le \;
c(\delta)
\end{equation*}
for all $H\in C^{1,2}([0,T]\times[0,1])$. This shows that
$\lim_{\delta\to 0} I_{[0,\delta]}(\tau_\delta u^\delta \,|\,
u^{(\gamma)}_\delta) =0$, and completes the proof of the lemma.
\end{proof}

Let $\color{bblue} \Pi_{2}$ be the set of all paths
$\pi(t,dx)=u(t,x)\, dx$ in $\Pi_{1}$ with the property that for every
$\delta>0$ there exists $\epsilon>0$ such that
$\epsilon\le u(t,x)\le 1-\epsilon$ for all
$(t,x)\in[\delta,T]\times [0,1]$.

\begin{lemma}
\label{l15} 
The set $\Pi_{2}$ is $I_{T}(\,\cdot\, |\gamma)$-dense.
\end{lemma}

\begin{proof}
Fix $\pi(t,dx)=u(t,x)\, dx$ in $\Pi_{1}$ such that
$I_{[0,T]}(\pi|\gamma)<\infty$.  For each $0<\epsilon<1$, define the
path $\pi^{\epsilon}(t,dx) = u^{\epsilon}(t,x) dx$ by
$u^{\epsilon} = (1-\epsilon) u + \epsilon u^{(\gamma)}$.

\smallskip
\noindent{\it Claim A:} For each $0<\epsilon< 1$, the trajectory
$\pi^{\epsilon}$ belongs to $\Pi_1$. Since $\pi$ belongs to $\Pi_1$,
by definition, there exists $\delta>0$ such that
$\pi^{\epsilon}_t = \pi_t$ for $0\le t\le \delta$. Therefore,
$\pi^{\epsilon}$ follows the hydrodynamic equation in the
time-interval $[0,\delta]$.  On the other hand, by the convexity of
the energy,
$\mathcal{Q}_{[0,T]}(u^{\epsilon}) \le \epsilon\, \mathcal{Q}_{[0,T]}
(u^{(\gamma)}) + (1-\epsilon)\, \mathcal{Q}_{[0,T]}(u)$. Hence, by
lemma \ref{l14}, $\mathcal{Q}_{[0,T]}(u^{\epsilon})<\infty$.
Therefore, $\pi^{\epsilon}$ belongs to $\Pi_{1}$, as claimed.

\smallskip
\noindent{\it Claim B:} For each $0<\epsilon< 1$, the trajectory
$\pi^{\epsilon}$ belongs to $\Pi_2$.  By Theorem \ref{mt4}, for every
$\delta>0$ there exists $\kappa >0$ such that
$\kappa \le u^{(\gamma)}_t\le 1-\kappa$ for all $\delta \le t \le
T$. This property is inherited by $u^{\epsilon}$ for a different
$\kappa = \kappa(\epsilon) $ because $0\le u\le 1$, which proves Claim
B.

It is clear that $\pi^{\epsilon}$ converges to $\pi$ in
$D([0,T],\ms M)$ as $\epsilon \downarrow 0$. Therefore, to conclude
the proof it is enough to show that $I_{[0,T]}(\pi^{\epsilon}|\gamma)$
converges to $I_{[0,T]}(\pi|\gamma)$ as $\epsilon \downarrow 0$. Since
the rate function is lower semicontinuous,
$I_{[0,T]}(\pi|\gamma)\le\liminf_{\epsilon\downarrow0}
I_{[0,T]}(\pi^{\epsilon}|\gamma)$.  On the other hand, as the rate
function $I_{[0,T]}(\,\cdot\,|\, \gamma)$ is convex, by Corollary
\ref{zero},
\begin{equation*}
I_{[0,T]}(\pi^{\epsilon}|\gamma)\;\le\;
(1-\epsilon)\, I_{[0,T]}(\pi|\gamma) \;+\;
\epsilon \, I_{[0,T]}( u^{(\gamma)}|\gamma)
\;\le\;
(1-\epsilon)\, I_{[0,T]}(\pi|\gamma) 
\;.
\end{equation*}
This completes the proof of the lemma.
\end{proof}

Let $\color{bblue} \Pi_{3}$ be the set of all paths
$\pi(t,dx)=u(t,x)\, dx$ in $\Pi_{2}$ whose density $u$ is continuous
in $(0,T]\times [0,1]$ and smooth in time: For all $x\in [0,1]$,
$u(x, \,\cdot\,)$ belongs to $C^{\infty}( (0,T])$.

\begin{lemma}
\label{l19} 
The set $\Pi_{3}$ is $I_{[0,T]}(\cdot|\gamma)$-dense.
\end{lemma}

\begin{proof}
Fix $\pi(t,dx)=u(t,x)dx$ in $\Pi_{2}$ such that
$I_{[0,T]}(\pi|\gamma)<\infty$.  Since $\pi$ belongs to the set
$\Pi_{1}$, the density $u$ solves the equation \eqref{1-06} in a
time interval $[0,3\delta]$ for some $\delta>0$. Let
$\varphi:\bb R\to\bb R$ be a smooth, nonnegative function such that
\begin{align*}
\text{supp}\ \varphi\;\subset\; (0,1)  \;\; \text{ and }\;\; 
\int_{0}^{1}\varphi(s)\, ds \;=\; 1\;.
\end{align*}
Set
$\color{bblue} \varphi_\epsilon(s) \,=\,
\epsilon^{-1}\varphi(s/\epsilon)$. 

Let $\chi:[0,T] \to [0,1]$ be
a smooth, nondecreasing function such that
\begin{align}
\label{6-22}
\begin{cases}
\chi(t)\;=\;0\    &\text{ if }\ t\in[0, \delta]\;, \\
0\;<\chi(t)\;<1\  &\text{ if }\ t\in (\delta,2\delta)\;, \\
\chi(t)\;=\;1\    &\text{ if }\ t\in[2\delta, T]\;,
\end{cases}
\end{align}
and set $\color{bblue} \chi_n (t) = \chi (t)/n$ for $n\ge 1$. Hence,
$\chi_n(t) = 1/n$ for $t\ge 2\delta$. 

Let $\pi^{n}(t,dx) = u^{n}(t,x)\, dx$ where
\begin{align*}
u^{n}(t,x) \;=\; \int_{0}^{1}u(t+\chi_{n}(t)\, s)\,
\varphi(s)\; ds
\;=\; \int_{\bb R} u(t+s)\, \varphi_{\chi_{n}(t)} (s)\;  ds\;.
\end{align*}
In the above formula, we extend the definition of $u$ to $[0,T+1]$ by
setting $u_{t}=u^{(u_T)}_{t-T}$ for $T \le t \le T+1$. This means that
on the interval $[T,T+1]$, $u_t$ follows the hydrodynamic equation
\eqref{1-06} starting from the initial condition $u_T$. [If $w$
represents the solution of equation \eqref{1-06} with $\gamma=u_T$,
$u_{T+t} = w_t$ for $0\le t\le 1$].

\smallskip\noindent{\it Claim A:} The trajectory $\pi^{n}$ belongs to
$\Pi_{1}$ for all $n>\delta^{-1}$.

Fix such $n\in \bb N$. By construction, the density $u^n$ coincides
with the solution $u^{(\gamma)}$ of the hydrodynamic in the
time-interval $[0,\delta]$. To estimate the energy of $u^n$, we
consider the time-intervals $[0,\delta]$, $[\delta, 2\delta]$ and
$[2\delta, T]$ separately. On $[0,\delta]$, $u^n$ coincides with
$u^{(\gamma)}$. Therefore, by Lemma \ref{l14}, the energy of $u^n$ in
this interval is bounded (uniformly in $n$). In the interval
$[\delta, 2\delta]$, $u^n_t$ is a convex combination of $u_{t+s}$ for
$0\le s\le 1/n\le \delta$. Since $u$ coincides with $u^{(\gamma)}$ in the
interval $[\delta, 3\delta]$, and since the solution is smooth in this
interval and bounded away from $0$ and $1$, the energy of $u^n$ in
this interval is bounded (uniformly in $n$). Finally, for $2\delta \le
t\le T$,
\begin{align*}
u^{n}(t,x) \;=\; 
\int_{0}^{1/n} u(t+s)\, \varphi_{1/n} (s)\;  ds\;.
\end{align*}
By convexity of the energy,
\begin{equation*}
Q_{[2\delta ,T]} (\pi^n)
\;\le\; \int_{0}^{1/n} Q_{[2\delta ,T]} (\tau _s \pi ) \, \varphi_{1/n}
(s)\;  ds \;\le\;
Q_{[2\delta ,T+1/n]} (\pi ) \;.
\end{equation*}
where the translation $\tau_s$ has been introduced in \eqref{4-02}.
This quantity is finite because $u$ has finite energy and by Lemma
\ref{l14}. This proves Claim A.

\smallskip\noindent{\it Claim B:} The trajectory $\pi^{n}$ belongs to
$\Pi_{3}$ for all $n>\delta^{-1}$.

As $\pi$ belongs to $\Pi_2$, by construction, so does $\pi^n$. By
Definition \ref{d02} and Theorem \ref{t01}, the function $u$ is smooth
in the set $(0, 3\delta) \times [0,1]$. Therefore, by definition, the
function $u^{n}$ is smooth in time on $(0,T]\times [0,1]$. As
$n>\delta^{-1}$, and since $u = u^{(\gamma)}$ is continuous in
$(0,3\delta)\times [0,1]$, by definition, $u^n$ is continuous in
$(0,2\delta)\times [0,1]$. We turn to the set
$[2\delta , T]\times [0,1]$. By convexity, for all
$2\delta \le t\le T$,
\begin{equation*}
\begin{aligned}
& \int_0^1 (\nabla u^n_t)^2\, dx \;\le\;
\int_{0}^{1/n}  ds \, \varphi_{1/n} (s)
\int_0^1 [\,\nabla u_{t+s} \,]^2 \: dx \\
&\quad \le\;
C_n\, \int_t^{t+(1/n)} ds 
\int_0^1 [\,\nabla u_{s} \,]^2 \: dx
\le\;
C_n\, \int_0^{T+1} ds 
\int_0^1 [\,\nabla u_{s} \,]^2 \: dx 
\end{aligned}
\end{equation*}
for some finite constant $C_n$. The last integral is finite for two
reasons. By Lemma \ref{l14}, the integral restricted to $[T,T+1]$ is
finite. The integral on $[0,T]$ is finite because $\pi$ has finite
energy as all elements of $\Pi_2$. It follows from this bound and from
its definition that $u^n_t$ is continuous on
$[2\delta, T]\times [0,1]$, which proves Claim B. \smallskip

It is clear that $\pi^{n}$ converges to $\pi$ in
$D([0,T],\mathcal{M})$. It remains to show that
$I_{[0,T]}(\,u^n\,|\gamma) \to I_{[0,T]}(\,u\,|\gamma)$.  As the
rate-function $I_{[0,T]}(\,\cdot\,|\gamma) $ is lower semicontinuous,
we turn to the bound
$\limsup_{n\to\infty} I_{[0,T]}(\pi^{n}|\gamma) \le
I_{[0,T]}(\pi|\gamma)$.

By \eqref{4-04}, the cost of the trajectory
$\pi^n$ in the interval $[0,T]$ is bounded by the sum of its cost in
the intervals $[0,\delta]$, $[\delta, 2\delta]$, $[2\delta, T]$. As
$u^n = u$ in the time-interval $[0,\delta]$, and as $u$ is the
solution of the hydrodynamic equation in this interval,
\begin{equation}
\label{6-20}
I_{[0,\delta]}(\pi^{n}|\gamma) \;=\; 0\; .
\end{equation}

Consider the contribution to $I_{[0,T]}(\pi^{n}|\gamma)$ of the piece
of the trajectory corresponding to the time interval $[2\delta,
T]$. Recall the definition of the functional $\tau_t$, introduced just
above \eqref{4-02}.  Since $\chi_n(t)=1/n$ in this interval, by the
concavity of $\sigma(\cdot)$, for any smooth function
$H:[0,T-2\delta] \times [0,1]\to \bb R$,
\begin{equation*}
\begin{aligned}
J_{T-2\delta,H} (\, \tau_{2\delta} u^n \,) \; & \le\;
\int \varphi_{1/n}(s)\,
J_{T-2\delta,H} (\, \tau_{2\delta +s} u \,)\; ds \\
& \le\;
\int \varphi_{1/n}(s)\, I_{[0,T-2\delta]}
(\, \tau_{2\delta +s} u \,)\; ds
\;. 
\end{aligned}
\end{equation*}
By \eqref{4-02}, the right-hand side is bounded by
\begin{equation*}
\int \varphi_{1/n}(s)\,  \big\{\, I_{[0,T-2\delta-s]}
(\, \tau_{2\delta +s} u \,) 
\;+\; I_{[0,s]}
(\, \tau_{T} u \,)\,\} \; ds\;.
\end{equation*}
Since $u$ solves the hydrodynamic equation on the interval $[T,T+1]$,
by Corollary \ref{zero}, $I_{[0,s]} (\, \tau_{T} u \,) =0$ for
$s\le 1$. Hence, by \eqref{4-06}, the previous integral is bounded by
\begin{equation*}
\int \varphi_{1/n}(s)\,  I_{[0,T]} (\,  u  \,)  \; ds
\;\le\; I_{[0,T]} (\,  u\, \,) \;.
\end{equation*}
Therefore, optimizing over $H$,
\begin{equation}
\label{6-19}
I_{[0,T-2\delta]}
(\,  \tau_{2\delta} \pi^n \,)
\;\le\; I_{[0,T]} (\,  u \,)\;.
\end{equation}

We turn to the contribution to $I_{[0,T]}(\pi^{n}|\gamma)$ of the
piece of the trajectory corresponding to the time interval
$[\delta, 2\delta]$. Since $u$ solves the hydrodynamic equation
\eqref{1-06} on the time interval $[\delta,3\delta]$, it is smooth in
$(0,3\delta) \times [0,1]$. Hence, by definition of $u^n$,
\begin{equation*}
\partial_{t}u^{n}(t,x) \;=\;
\int_{\bb R} \partial_{t}u(t+s,x) \,
\varphi_{\chi_{n}(t)} (s) \; ds
\;+\; \int_{\bb R}u(t+s,x ) \;
\partial_{t} \varphi_{\chi_{n}(t)} (s) \; ds\;.
\end{equation*}
As $u$ solves the hydrodynamic equation \eqref{1-06} on the time
interval $[\delta,3\delta]$, for any function $G$ in
$C^{1,2}([0,T]\times [0,1])$,
\begin{equation*}
\begin{aligned}
& \langle u_{2\delta}^{n}, G_{2\delta}\rangle \,-\,
\langle u_{\delta}^{n}, G_{\delta}\rangle
\,-\, \int_{\delta}^{2\delta}
\langle u_{t}^{n}, \partial_{t}G_{t}\rangle \; dt
\;=\;
-\, \int_{\delta}^{2\delta} \langle \nabla u^n_{t}, \nabla G_{t}\rangle \;
dt \\
&\qquad +\; 
\int_{\delta}^{2\delta} \big\{\, \frac{1}{B} \,[\beta - u^n_{t}(1)]\,
G_t(1) \,+\,\frac{1}{A} \,[\alpha - u^n_{t}(0)] \, G_t(0) \,\big\}\; dt
\;+\; \int_{\delta}^{2\delta}
\langle r_{t}^{n}, G_{t}\rangle \; dt \;,
\end{aligned}
\end{equation*}
where
\begin{equation*}
r^n_t(x)  \;=\; \int_{\bb R}
u (t+s,x) \,
\partial_t \varphi_{\chi_{n}(t)} (s)\; ds\;.
\end{equation*}
Therefore,
\begin{equation*}
\begin{aligned}
J_{\delta, G} (\tau_\delta u^n) \; &\le\; 
\int_{\delta}^{2\delta} \langle r_{t}^{n}, G_{t}\rangle \; dt
\;-\; \int_{\delta}^{2\delta} dt \int_0^1
\sigma (u^n_{t})  \, [\nabla G_{t}]^2 \; dx \\
& -\; \int_{\delta}^{2\delta} \big\{\,
\mf q_{\beta, B} (u^n_{t}(1), G_t(1))
\,+\, \mf q_{\alpha, A} (u^n_{t}(0), G_t(0))
 \,\big\}\; dt \;,
\end{aligned}
\end{equation*}
where $\mf q_{\varrho, D} (a, M)$ has been introduced in \eqref{6-04}.
Since $u$ belongs to $\Pi_2$, there exists $\epsilon>0$ such that
$\epsilon \le u(t,x) \le 1-\epsilon$ for all $\delta\le t\le T$,
$0\le x\le 1$. By Theorem \ref{mt4}, this bound extends to
$T\le t\le T+1$, $0\le x\le 1$. By definition, it is inherited by
$u^n$. Therefore, there exists a positive constant
$c_0 = c_0(\epsilon)$ such that
\begin{equation*}
\begin{aligned}
J_{\delta, G} (\tau_\delta u^n) \; &\le\; 
\int_{\delta}^{2\delta} \langle r_{t}^{n}, G_{t}\rangle \; dt
\;-\; c_0  \int_{\delta}^{2\delta} dt \int_0^1
[\nabla G_{t}]^2 \; dx \\
& -\; c_0  \,  \int_{\delta}^{2\delta} \big\{\,
G_t(1)^2 \,+\, G_t(0)^2  \,\big\}\; dt \;,
\end{aligned}
\end{equation*}
Adding and subtracting $G_t(0)$ to $G_t$ in
$\langle r_{t}^{n}, G_{t}\rangle$ yields, by Young's inequality, that
this scalar product is bounded by
$(1/2A_1) \< ( r_{t}^{n})^2\> + A_1\< [G_t-G_t(0)]^2\> + A_1 G_t(0)^2$
for all $A_1>0$. Hence, by choosing $A_1$ appropriately,
\begin{equation*}
J_{\delta, G} (\tau_\delta u^n) \; \le\; 
C_0 \, \int_{\delta}^{2\delta} dt \int_0^1  (r_{t}^{n})^2 \; dx  \;,
\end{equation*}
so that
\begin{equation}
\label{6-21}
I_{[0,\delta]} (\tau_\delta u^n) \; \le\; 
C_0 \, \int_{\delta}^{2\delta} dt \int_0^1  (r_{t}^{n})^2 \; dx  \;,
\end{equation}

It remains to show that $r^{n}(t,x)$ converges to $0$, as
$n\to\infty$, in $\ms L^2[(\delta,2\delta)\times [0,1])$.  Fix a point
$(t,x)$ in this set.  Since
$\int_{\bb R}\partial_{t}\, \varphi_{\chi_{n}(t)}(s) \, ds \,=\,
\partial_{t}\, \int_{\bb R}\varphi_{\chi_{n}(t)}(s) \, ds \,=\, 0$,
$r^{n}(t,x)$ can be written as
\begin{equation*}
\int_{\bb R} [\,  u(t+s,x)- u(t,x) \,] \,
\partial_{t} \, \varphi_{\chi_{n}(t)}(s) \, ds\; .
\end{equation*}
Since $ u$ is Lipschitz continuous on $[\delta, 3\delta]\times [0,1]$,
there exists a positive constant $C(\delta)>0$, depending only on
$\delta$, such that
\begin{equation*}
| \, u(t+s,x)- u(t,x)\, | \;\le\; C(\delta) \, s \;,
\end{equation*}
for any $(t,x)\in[\delta, 2\delta]\times [0,1]$ and $s\in[0,\delta]$.
Therefore $r^{n}(t,x)$ is bounded above by
\begin{equation*}
C(\delta)\int_{\bb R} s \, \big|\, 
\partial_{t} \, \varphi_{\chi_{n}(t)}(s) \, \big| \, ds \;.
\end{equation*}
By the change of variables $s' = s/\chi_{n}(t)$, 
\begin{equation*}
\int_{\bb R} s \, \big|\, 
\partial_{t} \, \varphi_{\chi_{n}(t)}(s) \, \big| \, ds
\;\le\; \frac{\|\, \chi'\,\|_{\infty}}{n} \int_{0}^{1} \big\{ 
\, s\, \varphi(s) + s^{2}\, |\varphi'(s)| \,\big\}\, ds\; .
\end{equation*}
Therefore,  as $n\to\infty$, $r^{n}$ converges to $0$ uniformly
in $(\delta,2\delta)\times [0,1]$, and, by \eqref{6-21},
\begin{equation*}
\lim_{n\to\infty} I_{[0,\delta]}
(\,  \tau_{\delta} \pi^n \,)
\;=\; 0\;.
\end{equation*}

By \eqref{4-02}, \eqref{6-20}, \eqref{6-19} and the previous estimate,
$\limsup_{n\to\infty} I_{[0,T]} (\, \pi^n\,|\, \gamma \,) \,\le\,
I_{[0,T]} (\, \pi\,|\, \gamma \,)$, which completes the proof of the
lemma.
\end{proof}

Let $\color{bblue} \Pi_{4}$ be the set of all paths
$\pi(t,dx) = u(t,x)\, dx$ in $\Pi_{3}$ whose density $u(t,\cdot)$
belongs to the space $C^{\infty}([0,1])$ for any $t\in (0,T]$. Note
that $\Pi_4 = \Pi_\gamma$, introduced in Definition
\ref{d04}. \smallskip

Denote by $\color{bblue} (P^{(D)}_t:t\ge 0)$,
$\color{bblue} (P^{(N)}_t:t\ge 0)$ the semigroup associated to the
Laplacian on $[0,1]$ with Dirichlet, Neumann boundary conditions,
respectively. The following property will be used many times
below. For all $s\ge 0$ and function $f$ in $C^1([0,1])$,
\begin{equation}
\label{6-23}
\nabla P^{(D)}_{s} f \;=\; P^{(N)}_{s} \nabla f\;.
\end{equation}
To check this identity, fix $f$ in $C^1([0,1])$, and let
$u_s := P^{(D)}_{s} f$. Clearly $u_s$ is the solution of the heat
equation on $[0,1]$ with boundary conditions $u_s(0)=u_s(1)=0$ and
initial condition $u_0=f$. Let $v_s := \nabla u_s$, Then, $v_s$ solves
the heat equation on $[0,1]$ with boundary conditions
$\nabla v_s(0)= \nabla v_s(1)=0$ and initial condition $v_0=\nabla
f$. Hence, $v_s$ can be represented as $v_s = P^{(N)}_{s} \nabla f$,
that is,
$P^{(N)}_{s} \nabla f = v_s = \nabla u_s = \nabla P^{(D)}_{s} f$, as
claimed.

Fix $\pi(t,dx) = u(t,x)\, dx$ in $\Pi_{3}$ such that
$I_{[0,T]}(\pi|\gamma)<\infty$.  Since $\pi$ belongs to the set
$\Pi_{1}$, the density $u$ solves the equation \eqref{1-06} in some
time interval $[0,3\delta]$, $\delta>0$.  Recall the definition of the
function $\chi_n(\cdot)$ introduced in \eqref{6-22}.  Let
$\pi^{n}(t,dx) = u^{n}(t,x)\, dx$, where
\begin{equation}
\label{6-29}
u^n_t \;=\; w_t \;+\; P^{(D)}_{\chi_n(t)} [\,u_t - w_t\,]\;.
\end{equation}
In this formula, $w_t(\cdot)$ is the smooth function given by
$\color{bblue} w_t(x) \,=\, u_t(0) \,+\, [\, u_t(1) \,-\, u_t(0)\,] \,
x$.

\begin{lemma}
\label{l17}
Fix $\pi(t,dx) = u(t,x)\, dx$ in $\Pi_{3}$ such that
$I_{[0,T]}(\pi|\gamma)<\infty$. Define $u^n$, $n\ge 1$, by
\eqref{6-29}.  For each $n\ge 1$, $\pi^{n} (t,dx)= u^n(t,x)\, dx$
belongs to $\Pi_4$ and the trajectory $u^n$ has finite energy.
\end{lemma}

\begin{proof}
\noindent{\it Claim A:} The trajectory $\pi^{n}$ belongs to
$\Pi_{1}$.

By definition, $u^n_t = u_t = u^{(\gamma)}_t$ for $0\le t\le
\delta$. It remains to estimate its energy.  As $u^n_t=u^{(\gamma)}_t$
for $0\le t\le \delta$, by Lemma \ref{l14}, the contribution to the
total energy of the evolution of $u^n$ in the time interval
$[0,\delta]$ is bounded. We turn to the contribution in the time
interval $[\delta,T]$.

By definition and \eqref{6-23},
$\nabla u^n_t = \nabla w_t \,+\, P^{(N)}_{\chi_n(t)} \nabla [\,u_t -
w_t\,]$, so that
$(\nabla u^n_t)^2 \le 2\,  (\nabla w_t)^2 \,+\, 2\, \{\,
P^{(N)}_{\chi_n(t)} \nabla [\,u_t -w_t\,]\, \}^2$. Therefore, as
$\epsilon(\delta) \le u^n_t \le 1 - \epsilon(\delta)$ for $\delta \le
t\le T$,  
\begin{equation*}
\begin{aligned}
& \int_\delta^T dt\int_0^1 \frac{|\nabla u^n_t|^2}{\sigma(u^n_t)}\; dx \;\le\;
C_0(\epsilon)\, \int_\delta^T dt\int_0^1 |\nabla u^n_t|^2 \; dx \\
& \;\; \;\le\; C_0(\epsilon)\, \int_\delta^T dt\int_0^1
(\nabla w_t)^2 \; dx \;+\;
C_0(\epsilon)\, \int_\delta^T dt\int_0^1
\{\, P^{(N)}_{\chi_n(t)} \nabla [\,u_t -w_t\,]\, \}^2  \; dx \;,
\end{aligned}
\end{equation*}
where the constant $C_0(\epsilon)$ changed from line to line. The
first term is bounded by the definition of $w_t$. As $P^{(N)}_{s}$ is
a contraction in $\ms L^2([0,1])$, the second term is bounded by
\begin{equation*}
C_0(\epsilon)\, \int_\delta^T dt\int_0^1 (\nabla u_t)^2  \; dx
\;+\; C_0(\epsilon)\, \int_\delta^T dt\int_0^1 (\nabla w_t)^2  \; dx\;.
\end{equation*}
The first term is bounded because $\pi_t(dx) = u(t,x)\, dx$ belongs to
$\Pi_3$. We already estimated the second one. This completes the proof
of Claim A.

\noindent{\it Claim B:} The trajectory $\pi^{n}$ belongs to $\Pi_{2}$.
By Theorem \ref{mt4}, and since $\pi$ belongs to $\Pi_{2}$, for every
$\delta'>0$, there exists $\epsilon>0$ such that
$\epsilon \le u_t\le 1-\epsilon$ for all $t \in [\delta',T]$.  Denote
by $\color{bblue} \epsilon (\delta)$ the constant $\epsilon$ when
$\delta'=\delta$.  As $u^n_t = u_t$ for $0\le t\le \delta$, this
property extends to $u^n_t$ in the interval $[0,\delta]$: for every
$0<\delta' \le \delta$, there exists $\epsilon>0$ such that
$\epsilon \le u^n_t\le 1-\epsilon$ for all $t \in [\delta',\delta]$.

We turn to the interval $[\delta, T]$. Fix $\delta \le t \le T$.  Let
$\color{bblue} v_s = v^{(t)}_s = w_t + P^{(D)}_{s} [u_t- w_t]$,
$s\ge 0$. Note that $u^n_t = v^{(t)}_{\chi_n(t)}$. By definition, $v$ is the
unique solution of the heat equation with Dirichlet boundary
conditions:
\begin{equation*}
\left\{
\begin{aligned}
& \partial_s v \;=\; \Delta v \; , \\
& v_s (0) \,=\, u_t(0) \;,\;\; v_s (1) \,=\, u_t(1) \\
& v(0, \cdot) = u_t (\cdot)\; .
\end{aligned}
\right.
\end{equation*}
Here we used the fact that $w(t,0) = u(t,0)$, $w(t,1) = u(t,1)$ and
that $\Delta w_t= 0$.
By the maximum principle, for all $s\ge 0$,
$\min_{0\le x \le 1} u_t(x) \le \min_{0\le x \le 1} v_s(x) \le
\max_{0\le x \le 1} v_s(x) \le \max_{0\le x \le 1} u_t(x)$. Hence, the
bound $\epsilon(\delta) \le u_t\le 1-\epsilon(\delta)$, which holds
for all $t \in [\delta,T]$ by definition of $\epsilon(\delta)$,
extends to $v^{(t)}_{\chi_n(t)} = u^n_t$. Therefore, $\pi^n$ belongs
to $\Pi_2$, as claimed.

The condition $\Delta w_t= 0$ selects $w_t$ among other possible
choices. More precisely, in principle one could define $w_t$ as
$w_t(x) \,=\, u_t(0) \,+\, [\, u_t(1) \,-\, u_t(0)\,] \, f(x)$ for any
smooth function $f(x)$ such that $f(0) =0$, $f(1)=1$. However, the
proof that $u^n$ belongs to $\Pi_2$ is based on the maximum principle
for the heat equation with Dirichlet boundary conditions. For $v$ to
be a solution we need $\Delta w_t=0$ which imposes the choice
$f(x) =x$.

It remains to examine the regularity in space and time of the
trajectory $u^n_t$.  Since $u_t$ belongs to $\Pi_3$ and as the
time-derivative commutes with the operator $P^{(D)}_s$, by definition,
the trajectory $u^n_t$ also belongs to $\Pi_3$. Furthermore, as $w_t$
is smooth in space, by Theorem \ref{mt4} and its equivalent version
for the heat equation with Dirichlet boundary conditions,
$u^n_t \in C^{\infty}([0,1])$ for all $0<t\le T$, and $u^n_t$ belongs to
$\Pi_4$. This completes the proof of the lemma.
\end{proof}

\begin{lemma}
\label{l16} 
The set $\Pi_{4}$ is $I_{[0,T]}(\cdot|\gamma)$-dense.
\end{lemma}

\begin{proof}
Fix $\pi(t,dx) = u(t,x)\, dx$ in $\Pi_{3}$  such that
$I_{[0,T]}(\pi|\gamma)<\infty$. Keep in mind that $u$ is
continuous in $(0,T]\times [0,1]$. Define $u^n$, $n\ge
1$, by \eqref{6-29}, and let $\pi^{n} (t,dx) = u^n(t,x)\,
dx$. By Lemma \ref{l17}, $\pi^n$ belongs to $\Pi_{4}$.

By definition, $\pi^{n}$ converges to $\pi$ in $D([0,T],\ms M)$.
Hence, by the lower semicontinuous of the rate function, it remains to
show that
$\limsup_{n\to\infty}I_{[0,T]}(\pi^n|\gamma) \le
I_{[0,T]}(\pi|\gamma)$.

By \eqref{4-02}, the cost of the trajectory $\pi^n$ in the interval
$[0,T]$ is bounded by the sum of its cost in the intervals
$[0,\delta]$, $[\delta, T]$:
\begin{equation}
\label{6-30}
I_{[0,T]} (u^n) \;\le\; I_{[0,\delta]} (u^n)
\;+\; I_{[0,T-\delta]} (\tau_\delta u^n)\;.
\end{equation}

As $u^n = u^{(\gamma)}$ in the
time-interval $[0,\delta]$,
\begin{equation}
\label{6-31}
I_{[0,\delta]}(u^{n}) \;=\; 0\; .
\end{equation}
We turn to the interval $[\delta, T]$. Recall the notation introduced
in \eqref{4-02}. The cost of the trajectory in this interval is given
by $I_{[0,T-\delta]}(\tau_{\delta}\, u^n)$. Let
$\color{bblue} \widehat \chi_n(t) = \chi_n(t-\delta)$,
$\color{bblue} T_\delta = T-\delta$,
$\color{bblue} v= \tau_{\delta}\, u$,
$\color{bblue} v^n= \tau_{\delta}\, u^n$,
$\color{bblue} \widehat w= \tau_{\delta}\, w$, and observe that
$v^n_t = \widehat w_t + P^{(D)}_{\widehat \chi_n(t)} [v_t - \widehat
w_t]$, $0\le t\le T_\delta$.  Moreover,
\begin{equation}
\label{6-25}
\epsilon(\delta) \;\le\; v^n_t \;\le\;  1 \,-\, \epsilon(\delta)
\end{equation}
for $0 \le t\le T_\delta$, where $\epsilon(\delta)$ has been
introduced at the beginning of the proof of Lemma \ref{l17}.  With
this notation,
$I_{[0,T-\delta]} (\tau_\delta u^n) = I_{[0,T_\delta]}(v^n)$.

By Lemma \ref{l21},
$I_{[0,T_\delta]}(v^n) = I^{(1)}_{[0,T_\delta]}(v^n) +
I^{(2)}_{[0,T_\delta]}(v^n)$.  We estimate each term of this sum
separately.  The next observation will be useful in the argument.

Let $L^{(\partial_t)}$, $L^{(\partial_t)}_0$ be the functional
introduced in \eqref{9-08}, \eqref{9-09} with $T$, $u_t$ replaced by
$T_\delta$, $v_t$, respectively. Keep in mind that these linear
functionals depend on the trajectory $u(\cdot, \cdot)$, that is, on
$v$. Since
$I_{[0,T_\delta]}(v ) = I_{[0,T-\delta]}(\tau_{\delta}\, u) \le
I_{[0,T]}(u) < \infty$, by \eqref{9-06}, \eqref{9-07},
$L^{(\partial_t)}_0$ belongs to $\mc H^{-1}(\sigma (v))$ and there
exists $M$ in $\ms L^2(\sigma(v)^{-1})$ such that
\begin{equation}
\label{6-24}
L^{(\partial_t)}_0 (H) \;=\;
\int_0^{T_\delta} \<\, M_s \,,\, \nabla H_s\,\> \; ds \;, \quad
\int_0^1 \frac{M_s}{\sigma(v_s)}\; dx \;=\; 0
\end{equation}
for all $H$ in $C^\infty_K(\Omega_{T_\delta})$, and almost all
$0\le s\le T_\delta$.

We turn to $I^{(1)}_{[0,T_\delta]}(v^n)$.  By Lemma \ref{l21},
$I^{(1)}_{[0,T_\delta]}(v^n) = (1/4) \Vert \mf L_0 \Vert^2_{-1,
\sigma(v^n)}$. The linear functional $\mf L_0 $ introduced just below
\eqref{9-08} is the sum of $L^{(\partial_t)}_0$ with
$L^{(\nabla)}_0$. We first examine $L^{(\partial_t)}_0$.

\smallskip\noindent{\it The linear functional $L^{(\partial_t)}_0$}.
By definition, since $P^{(D)}_{s}$ is a symmetric
operator in $\ms L^2([0,1])$, for every $H\in C^\infty_K(\Omega_T)$,
\begin{equation*}
\begin{aligned}
& \int_0^{T_\delta} \< \, \partial_t v^n_t \,,\, H_t\,\>\, dt 
\;=\;
\int_0^{T_\delta} \< \,  (\, I \,-\, P^{(D)}_{\widehat \chi_n(t)} \,)\, 
\partial_t \widehat w_t \,,\, H_t  \,\> \, dt \\
&\quad \:+\;
\int_0^{T_\delta} \< \, \partial_t v_t \,,\, 
P^{(D)}_{\widehat \chi_n(t)} H_t \,\> \, dt
\:+\; \int_0^\delta \widehat \chi'_n(t) \,
\<\, \Delta P^{(D)}_{\widehat \chi_n(t)}  [ \, v_t - \widehat w_t\,]
\,,\, H_t\,\>\, dt \;.
\end{aligned}
\end{equation*}
The last integral runs from $0$ to $\delta$ because $\widehat
\chi'_n(t)$ vanishes for $t\ge \delta$.

As $(t,x) \mapsto (P^{(D)}_{\widehat \chi_n(t)} H_t)(x)$ is a smooth
function which vanishes at $x=0$ and $x=1$, by \eqref{6-24}, the
second term on the right-hand side is equal to the time integral of
$\< \, M_t \,,\, \nabla \, P^{(D)}_{\widehat \chi_n(t)} H_t \,\>$.  By
\eqref{6-23}, this scalar product is equal to
$\< \, M_t \,,\, P^{(N)}_{\widehat \chi_n(t)} \nabla H_t \,\> \,=\, \<
\, P^{(N)}_{\widehat \chi_n(t)} \, M_t \,,\, \nabla H_t \,\>$ because
the operator $P^{(N)}_{\widehat \chi_n(t)} $ is symmetric in
$\ms L^2([0,1])$.

On the other hand, as $H_t$ vanishes at the boundary, an integration
by parts yields that the third term on the right-hand side is equal to
\begin{equation*}
-\, \int_0^\delta \widehat \chi'_n(t) \,
\<\, \nabla P^{(D)}_{\widehat \chi_n(t)}  [ \, v_t - \widehat w_t\,]
\,,\, \nabla H_t\,\>\, dt \;=\;
-\, \int_0^\delta \widehat \chi'_n(t) \,
\<\,  P^{(N)}_{\widehat \chi_n(t)}  \nabla [ \, v_t - \widehat w_t\,]
\,,\, \nabla H_t\,\>\, dt\;,
\end{equation*}
where we apllied the identity \eqref{6-23} once more.

In conclusion,
\begin{equation}
\label{6-28}
\begin{aligned}
& \int_0^{T_\delta} \< \, \partial_t v^n_t \,,\, H_t\,\>\, dt 
\;=\;
\int_0^{T_\delta} \< \, (\, I \,-\, P^{(D)}_{\widehat \chi_n(t)} \,)\, 
\partial_t \widehat w_t \,,\, H_t  \,\> \, dt \\
&\quad \:+\;
\int_0^{T_\delta} \< \, P^{(N)}_{\widehat \chi_n(t)} \, M_t
\,,\, \nabla H_t \,\>\, dt
\;-\;  \int_0^\delta \widehat \chi'_n(t) \,
\<\,  P^{(N)}_{\widehat \chi_n(t)}  \nabla [ \, v_t - \widehat w_t\,]
\,,\, \nabla H_t\,\>\, dt \;.
\end{aligned}
\end{equation}

We estimate the first and the last term on the right-hand side.
By Young's inequality $xy \le (1/2A_1) x^2 + (A_1/2) y^2$, $A_1>0$,
the first term on the right-hand side is bounded by
\begin{equation*}
\frac{1}{2A_1}\, \int_0^{T_\delta} \< \,
[\, (\, I \,-\, P^{(D)}_{\widehat \chi_n(t)} \,) \, \partial_t \widehat w_t\, ]^2 \,\> \, dt
\;+\; \frac{A_1}{2}\, \int_0^{T_\delta}
\< \, H_t^2 \,\> \, dt
\end{equation*}
for all $A_1>0$. As $H_t$ vanishes at the boundary of $[0,1]$, by
Poincar\'e's inequality and \eqref{6-25}, this sum is bounded by
\begin{equation*}
\begin{aligned}
& \frac{1}{2A_1}\, \int_0^{T_\delta} \< \,
[\, (\, I \,-\, P^{(D)}_{\widehat \chi_n(t)} \,) \, \partial_t \widehat w_t\, ]^2 \,\> \, dt
\;+\; C_0\, A_1 \, \int_0^{T_\delta} \< \, (\, \nabla H_t \,)^2 \,\> 
dt \\
& \quad\le \frac{1}{2A_1}\, \int_0^{T_\delta} \< \,
[\, (\, I \,-\, P^{(D)}_{\widehat \chi_n(t)} \,) \, \partial_t \widehat w_t\, ]^2 \,\> \, dt
\;+\; C_0\, A_1 \, \int_0^{T_\delta} dt\,
\int_0^1 \sigma(v^n_t) \,  (\, \nabla H_t \,)^2 \; dx
\end{aligned}
\end{equation*}
for some finite constant $C_0 = C_0(u)$ which may change from line to line.

Since $\widehat \chi'_n(t) = (1/n) \, \chi'(t-\delta)$, by Young's
inequality, the third term on the right-hand side of \eqref{6-28} is
bounded by
\begin{equation*}
\frac{C_0}{n}\, \int_0^\delta 
\<\, \{\, P^{(N)}_{\widehat \chi_n(t)}
\nabla  [ \, v_t - \widehat w_t\,]\,\}^2 \,\>\, dt
\;+\;
\frac{1}{n}\, \int_0^\delta \<\, (\, \nabla H_t\,)^2 \,\>\, dt
\end{equation*}
for some finite constant $C_0$ which depends on $\chi(\cdot)$.  As
$P^{(N)}_{s}$, $s\ge 0$, is a contraction in $L^2([0,1])$ and since
$\epsilon(\delta) \le v^n_t \le 1- \epsilon(\delta)$, this
expression is less than or equal to
\begin{equation*}
\frac{C_0}{n}\, \int_0^\delta dt \int_0^1 
[ \, \nabla v_t - \nabla \widehat w_t\,]^2 \; dx
\;+\;
\frac{C_1}{n}\, \int_0^\delta dt \int_0^1 \sigma(v^n_t)
[\, \nabla H_t\,]^2 \; dx
\end{equation*}
for some finite constant $C_1=C_1(u)$. We turn to the linear
functional $L^{(\nabla)}_0$.

\smallskip\noindent{\it The linear functional $L^{(\nabla)}_0$}.
By definition of $v^n_t$,
\begin{equation*}
\int_0^{T_\delta} \< \, \nabla v^n_t \,,\, \nabla H_t\,\>\, dt 
\;=\;
\int_0^{T_\delta} \< \, (\, I \,-\, P^{(N)}_{\widehat \chi_n(t)} \,)\,
\nabla \widehat w_t \,,\, \nabla  H_t  \,\> \, dt
\:+\;
\int_0^{T_\delta} \< \, P^{(N)}_{\widehat \chi_n(t)} \, \nabla v_t
\,,\, \nabla H_t \,\>\, dt\;.
\end{equation*}
For similar reasons to the ones presented above, the first term on the
right-hand side is bounded by
\begin{equation*}
\frac{1}{2A_2}\, \int_0^{T_\delta} \< \,
[\, (\, I \,-\, P^{(N)}_{\widehat \chi_n(t)} \,) \, \nabla \widehat w_t\, ]^2 \,\> \, dt
\;+\; C_0\, A_2 \, \int_0^{T_\delta} dt\,
\int_0^1 \sigma(v^n_t) \, (\, \nabla H_t \,)^2  \; dx
\end{equation*}
for all $A_2>0$ and some finite constant $C_0=C_0(u)$.

\smallskip\noindent{\it The linear functional $\mf L_0$}.  We are now
in a position to estimate
$I^{(1)}_{[0,T_\delta]}(v^n) = (1/4) \Vert \mf L_0 \Vert^2_{-1,
\sigma(v^n)}$.  Let
\begin{equation*}
\begin{gathered}
r_1(n) \;=\; \int_0^{T_\delta} \< \,
[\, (\, I \,-\, P^{(D)}_{\widehat \chi_n(t)} \,) \, \partial_t \widehat w_t\, ]^2
\,\> \, dt \;, \quad
r_2(n) \;=\; \frac{C_0}{n}\, \int_0^\delta dt \int_0^1 
[ \, \nabla v_t - \nabla \widehat w_t\,]^2 \; dx\;, \\
r_3(n) \;=\; \int_0^{T_\delta} \< \,
[\, (\, I \,-\, P^{(N)}_{\widehat \chi_n(t)} \,) \, \nabla \widehat w_t\, ]^2 \,\> \,
dt \;.
\end{gathered}
\end{equation*}
As both semigroups are continuous, $\lim_{n\to\infty} r_j(n) = 0$ for
$j=1$, $3$. As $u$ (and, thus, $v$) has finite energy, by definition
of $\widehat w$, $\lim_{n\to\infty} r_2(n) = 0$.  Set
$\color{bblue} A_j = \sqrt{r_j(n)} =: c_j(n)$, $j=1$, $3$, and
$c_2(n) := r_2(n)$, to get from the bounds obtained above that
\begin{equation*}
\begin{aligned}
& 2\, \int_0^{T_\delta} \< \, \partial_t v^n_t \,,\, H_t\,\>\, dt  \;+\;
2\, \int_0^{T_\delta} \< \, \nabla v^n_t \,,\, \nabla H_t\,\>\, dt
\;-\; \int_0^{T_\delta} dt\,
\int_0^1 \sigma(v^n_t) \,  (\, \nabla H_t \,)^2 \; dx \\
& \quad \le\; 2\, \int_0^{T_\delta} \< \, P^{(N)}_{\widehat \chi_n(t)} \, M_t
\,,\, \nabla H_t \,\>\, dt \;+\;
2\, \int_0^{T_\delta} \< \, P^{(N)}_{\widehat \chi_n(t)} \, \nabla v_t
\,,\, \nabla H_t \,\>\, dt \\
& \quad -\; [1 - \epsilon_n] \, \int_0^{T_\delta} dt\,
\int_0^1 \sigma(v^n_t) \,  (\, \nabla H_t \,)^2 \; dx
\; + \; c_n \;,
\end{aligned}
\end{equation*}
where $c_n = \sum_{1\le j\le 3} c_j(n)$,
$\epsilon_n = C_0 \, [\, c_1(n) \,+\, c_2(n) \, +\, (1/n)\,]$
so that $\lim_{n \to \infty} \epsilon_n = 0$. Note that $c_n$,
$\epsilon_n$ do not depend on $H$.  Hence, by definition of
$I^{(1)}_{[0,T_\delta]}(\,\cdot \,)$ and \eqref{9-0},
\begin{equation*}
I^{(1)}_{[0,T_\delta]}(v^n) \;\le\;
\frac{1}{4(1 - \epsilon_n)}\,
\Vert\, L^n_0 \,\Vert^2_{-1,\sigma(v^n)}
\;+\; c_n \;,
\end{equation*}
where $L^n_0 (H) = \int_0^{T_\delta} \< P^{(N)}_{\widehat \chi_n(t)} \, M_t \,+ \,
P^{(N)}_{\widehat \chi_n(t)} \, \nabla v_t \,,\, \nabla H_t\, \> \, dt$.

By Lemma \ref{l20}, the first term on the right-hand side of the
previous displayed equation is equal to
\begin{equation*}
\frac{1}{4(1 - \epsilon_n)}\, \int_0^{T_\delta} \Big\{
\int_0^1 \frac{1}{\sigma(v^n_t)} \,
[\, P^{(N)}_{\widehat \chi_n(t)} \, M_t \,+ \, P^{(N)}_{\widehat \chi_n(t)} \, \nabla v_t\,]^2
\; dx \, - R^n_t\, \Big\} \, dt \;,
\end{equation*}
where
$R^n_t = \<\, [P^{(N)}_{\widehat \chi_n(t)} \, M_t \,+ \, P^{(N)}_{\widehat \chi_n(t)} \, \nabla
v_t]/ \sigma(v^n_t)\, \>^2 \,/\, \<\, 1/ \sigma(v^n_t)\>$.

Consider the limit, as $n\to\infty$, of the two previous displayed
equations.  Since $\epsilon_n\to 0$ and $c_n \to 0$ we may ignore
these constants. On the other hand, by \eqref{6-25},
$\epsilon(\delta) \;\le\; v^n_t \;\le\; 1 \,-\, \epsilon(\delta)$.
Therefore, as the semigroup $(P^{(N)}_{t} : t\ge 0)$ is continuous in
$\ms L^2([0,1])$, we may replace in the previous equations
$P^{(N)}_{\widehat \chi_n(t)} \, M_t$,
$P^{(N)}_{\widehat \chi_n(t)} \, \nabla v_t$ by $M_t$, $\nabla v_t$,
respectively, at a cost which vanishes as $n\to\infty$. Finally, as
$v^n_t \to v_t$ a.e., we conclude that
\begin{equation*}
\limsup_{n\to \infty}
I^{(1)}_{[0,T_\delta]}(v^n) \;\le\;
\frac{1}{4}\, \int_0^{T_\delta} \Big\{
\int_0^1 \frac{1}{\sigma(v_t)} \,
[\, M_t \,+ \, \nabla v_t\,]^2 \; dx \, - R'_t\, \Big\} \, dt \;, 
\end{equation*}
where
$R'_t = \<\, [\, M_t \,+ \, \nabla v_t]/ \sigma(v_t)\>^2 \,/\, \<\, 1/
\sigma(v_t)\>$. By \eqref{6-24}, this expression is equal to $R_t$,
where
$R_t \,=\, \<\, \nabla v_t/ \sigma(v_t)\>^2 \,/\, \<\, 1/
\sigma(v_t)\>$. Hence, by Lemma \ref{l22} [with $u_t$ replaced by
$v_t$],
\begin{equation*}
\limsup_{n\to \infty}
I^{(1)}_{[0,T_\delta]}(v^n) \;\le\;
I^{(1)}_{[0,T_\delta]}(v)\;.
\end{equation*}

We turn to $I^{(2)}_{[0,T_\delta]}(v^n)$. By Lemma
\ref{l21},
\begin{equation*}
I^{(2)}_{[0,T_\delta]}(v^n)
\;=\; \int_0^{T_\delta} \Phi^{v_n}(a^n_t, b^n_t)\; dt\;,
\end{equation*}
where $a^n_t$, $b^n_t$ are given by \eqref{9-10}, \eqref{9-02} with
$u$ replaced by $v^n$. To stress the dependence of $\Phi$ on $v_n$, we
denoted this functional by $\Phi^{v_n}$. However, as
$v^n(t,1) = v(t,1)$, $v^n(t,0) = v(t,0)$, $\Phi^{v_n}= \Phi^{v}$. 

Let $\Xi^n$, $\Xi$ be given by \eqref{9-02} with $v^n$, $v$ in place
of $u$, respectively. As $v^n \to v$ almost everywhere, and since
$\epsilon (\delta) \le v^n \le 1 - \epsilon (\delta)$, the continuous
function $\Xi^n$ converges to $\Xi$ pointwisely.

An elementary computation, similar to the one presented above when we
examined the rate function $I^{(1)}_{[0,T_\delta]}$, yields that
\begin{equation*}
\begin{aligned}
a^n_t \;& =\; \< \, (\, I \,-\, P^{(D)}_{\widehat \chi_n(t)} \,) \,
\partial_t \widehat w_t \,,\, 1 \,-\, \Xi^n_t \,\>
\;-\; \< \, (\, I \,-\, P^{(N)}_{\widehat \chi_n(t)} \,) \,
\nabla \widehat w_t \,,\, \nabla \Xi^n_t \,\> \\
& \quad +\; \< \, P^{(D)}_{\widehat \chi_n(t)}  \,
\partial_t v_t \,,\, 1 \,-\, \Xi^n_t \,\>
\;-\; \< \, P^{(N)}_{\widehat \chi_n(t)} \,
\nabla v_t \,,\, \nabla \Xi^n_t \,\> \\
&\quad \;+\; \widehat \chi_n'(t)\,
\< \,  v_t \, -\, \widehat w_t \,,\,
\Delta P^{(D)}_{\widehat \chi_n(t)} [1 \,-\, \Xi^n_t] \,\> \;.
\end{aligned}
\end{equation*}
Note that in the last term the operator
$\Delta P^{(D)}_{\widehat \chi_n(t)}$ is acting on $[1 \,-\, \Xi^n_t]$
instead of $v_t \, -\, \widehat w_t$, as in the first part of the
proof. Here, we simply used the fact that the semigroup $P^{(D)}_{r}$
is symmetric.

As $v_t \, -\, \widehat w_t$ vanishes at the boundary, an integration by
parts and \eqref{6-23} yield that the last term is equal to
\begin{equation*}
-\, \widehat \chi_n'(t)\,
\< \,  \nabla [ v_t \, -\, \widehat w_t] \,,\,
\nabla P^{(D)}_{\widehat \chi_n(t)} [1 \,-\, \Xi^n_t] \,\>
\;=\; \widehat \chi_n'(t)\,
\< \,  P^{(N)}_{\widehat \chi_n(t)}
\nabla [ v_t \, -\, \widehat w_t] \,,\,
\nabla  \Xi^n_t \,\>\;,
\end{equation*}
where we used that the semigroup $P^{(N)}_s$ is symmetric in
$\ms L^2([0,1])$.

Since $\epsilon (\delta) \le v^n \le 1 - \epsilon (\delta)$, there
exists a finite constant $C_0$ such that $|\, \Xi^n_t\,|\,\le\, C_0$,
$|\, \nabla \Xi^n_t\,|\,\le\, C_0$ for all $n\ge 1$,
$0\le t\le T_\delta$. Therefore, as
$\widehat \chi_n'(t)\,=\, (1/n) \chi'(t-\delta)$ and since the
operators $P^{(N)}_{s}$, $P^{(D)}_{s}$ are contractions in
$\ms L^2([0,1])$, there exists a finite constant $C_0$ such that
$|\, a^n_t\,|^2 \,\le\, C_0 \, \{\, 1\,+\, \< \, (\partial_t v_t)^2 \,\>
\,+\, \< \, (\nabla v_t)^2 \,\>\,\}$ for all $n\ge 1$,
$0\le t\le T_\delta$. Moreover,
\begin{equation*}
\lim_{n\to\infty} \, \widehat \chi_n'(t)\,
\< \,  P^{(N)}_{\widehat \chi_n(t)}
\nabla [ v_t \, -\, \widehat w_t] \,,\,
\nabla  \Xi^n_t \,\> \;=\; 0\;,
\end{equation*}
and, as $\Xi^n_t \to \Xi_t$, $\nabla \Xi^n_t \to \nabla \Xi_t$ in
$\ms L^2([0,1])$, for all $0\le t\le T_\delta$,
\begin{equation}
\label{6-27}
\lim_{n\to\infty} a^n_t \;=\; a_t \;:=\;
\< \, \partial_t v_t \,,\, 1 \,-\, \Xi_t \,\>
\;-\; \< \,  \nabla v_t \,,\, \nabla \Xi_t \,\>\;.
\end{equation}
A similar bound and limit hold for the sequence $b^n_t$. Since
$\Phi^{v}$ is continuous and the map $t \mapsto \< \, (\partial_t v_t)^2 \,\>
\,+\, \< \, (\nabla v_t)^2 \,\>$ is integrable, by \eqref{9-11},
\eqref{6-27} and the dominated convergence theorem,
\begin{equation*}
\lim_{n\to\infty}
\int_0^{T_\delta} \Phi^{v}(a^n_t, b^n_t)\; dt \;=\;
\int_0^{T_\delta} \Phi^{v}(a_t, b_t)\; dt\;.
\end{equation*}
By Lemma \ref{l21}, the right-hand side is equal to
$I^{(2)}_{[0,T_\delta]}(v)$. Therefore,
\begin{equation*}
\lim_{n\to\infty} I^{(2)}_{[0,T_\delta]}(v^n)
\;=\; I^{(2)}_{[0,T_\delta]}(v) \;.
\end{equation*}

Since $v^n = \tau_\delta u^n$, $v = \tau_\delta u$, adding together
the estimates on $I^{(1)}_{[0,T_\delta]}(v^n)$ and
$I^{(2)}_{[0,T_\delta]}(v^n)$, yield that
\begin{equation*}
\limsup_{n\to\infty} I_{[0,T_\delta]}(\tau_\delta u^n)
\;\le \; I_{[0,T_\delta]}(\tau_\delta u)\;.
\end{equation*}
By \eqref{4-06}, this expression is bounded by
$I_{[0,T]}( u \,|\, \gamma)$, which completes the proof of the lemma
in view of \eqref{6-30}, \eqref{6-31}.
\end{proof}

\begin{proof}[Proof of Theorem \ref{mt3}]
The first assertion follows from Lemma \ref{l16} and the definition of the
set $\Pi_4$.

Assume that there exists $\epsilon_0 >0$ such that
$\epsilon_0 \le \gamma \le 1-\epsilon_0$. Fix
$\pi \in D([0,T], \ms M)$ such that $I_{[0,T]}
(\pi|\gamma)<\infty$. Let $\pi^n (t,dx) = u^n(t,x)\, dx$ be the
sequence in $\Pi_\gamma$ which
$I_{[0,T]} (\,\cdot\,|\gamma)$-approximates $\pi$ in the sense of
Definition \ref{d05}. Since $\pi^n$ belongs to $\Pi_\gamma$, there
exists $\delta>0$ and $\epsilon>0$ such that $u^n_t = u^{(\gamma)}_t$
for $0\le t\le \delta$ and $\epsilon \le u^n(t,x) \le 1-\epsilon$ for
all $(t,x) \in [\delta, T]\times [0,1]$. By \eqref{1-08}, there
exists $\epsilon_1>0$ such that
$\epsilon_1 \le u^n(t,x) \le 1-\epsilon_1$ for all
$(t,x) \in [0,\delta]\times [0,1]$. This completes the proof of the
theorem.
\end{proof}

\begin{remark}
\label{rm4}
The difference between the present context and \cite{BLM09} is that
the rate function is convex. We used this property to restrict our
attention to trajectories bounded away from $0$ and $1$ and smooth in
time {\rm [}that is to paths in $\Pi_3${\rm ]}. 
\end{remark}

We conclude this section deriving the explicit formula for the rate
functions of trajectories in $\Pi_\gamma$.

\begin{proof}[Proof of Proposition \ref{l09b}]
As $u$ belongs to $\Pi_\gamma$, $u$ is smooth in
$(0,T]\times [0,1]$, and for each $0<t\le T$, there exists
$\delta = \delta(t)>0$ such that $\delta \le u(t,x)\le 1-\delta$.
Therefore, equation \eqref{5-01b} is strictly elliptic and can be
solved explicitly. The solution $H$ inherits the smoothness from
$u$. In particular, it belongs to $C^{1,2}((0,T]\times [0,1])$.

As $u$ belongs to $\Pi_\gamma$, $u$ follows the hydrodynamic equation
in a time-interval $[0, \mf t]$ for some $\mf t>0$. Hence, for
$0<t\le \mf t$, the solution of \eqref{5-01b} vanishes: $H(t,x)=0$ for
all $(t,x) \in [0,\mf t]\times [0,1]$. Hence, $H$ actually belongs to
$C^{1,2}([0,T]\times [0,1])$.

We turn to the formula for the rate function. For a function $G$ in
$C^{1,2}([0,T]\times [0,1])$, let
\begin{equation*}
L_{T,G} ( u ) \; =\; \big\langle u_T, G_T \big\rangle
\;-\; \langle u_0 , G_0  \rangle
\;-\; \int_0^{T} \big\langle u_t, \partial_t G_t  \big\rangle \; dt
\; +\; \int_0^{T} \big\langle \nabla  u_t , \nabla  G_t \big\rangle \; dt
\; .
\end{equation*}
Multiply equation \eqref{5-01b} by $G$, integrate over space and time,
integrate by parts in space, and recall the boundary conditions to
get that
\begin{equation*}
\begin{aligned}
L_{T,G} ( u ) \; & =\; 2 \, \int_0^{T} \big\langle
\sigma( u_t) \, \nabla H_t \, ,\,  \nabla  G_t \big\rangle \; dt \\
& +\;  \int_0^{T} \big\{\, G_t(1)\, \mf p_{\beta, B} \big(\,
u_t(1)\,,\, H_t(1)\, \big) \;+\; G_t(0)\,
\mf p_{\alpha, A} \big(\, u_t(0)\,,\, H_t(0)\, \big)\,\big\}\; dt
\; .
\end{aligned}
\end{equation*}
Insert this expression in \eqref{1-01b} and add and subtract some
terms to get that
\begin{equation*}
\begin{aligned}
J_{T,G} ( u ) \; =\; & -\, \int_0^{T} \big\langle
\sigma( u_t) \, [\, \nabla H_t \, -\,  \nabla  G_t\,]^2 \big\rangle \; dt \\
& - \;  \frac{1}{A} \, \int_0^{T} [\,1-u_t(0)\,] \, \alpha\,
\big[\, e^{G_t(0)} \,-\, e^{H_t(0)} \,-\, [G_t(0) - H_t(0)]\,
e^{H_t(0)} \,\big] \; dt \\
& - \;  \frac{1}{A} \, \int_0^{T}
u_t(0) \, (1-\alpha) \,
\big[\, e^{-G_t(0)} \,-\, e^{-H_t(0)} \,-\, [G_t(0) - H_t(0)]\,
e^{-H_t(0)} \,\big] \; dt \\
&- \; \mf B \; +\; \bb I_{[0,T]} (u)\;, 
\end{aligned}
\end{equation*}
where $\bb I_{[0,T]} (u)$ is the expression appearing on the
right-hand side of \eqref{5-03} and $\mf B$ is a term similar to the
second and third lines of this formula with the left boundary
conditions replaced by the right ones. Since the expressions inside the
integrals are all positive, the supremum in $G$ is attained at $G=H$,
so that
\begin{equation*}
I_{[0,T]} (u) \;=\; \sup_{G} J_{T,G} ( u ) \; =\; \bb I_{[0,T]} (u)\;,
\end{equation*}
which completes the proof of the lemma.
\end{proof}

\section{Proof of Theorem \ref{LDP}}
\label{sld}

In this section, we prove the dynamical large deviations. The strategy
is by now classical and we just indicate the main steps. The main
point here is that the dynamics can be considered as a small
perturbation of the exclusion process with Neumann boundary conditions
(the process induced by the generator $L^{\rm bulk}_{N} $) because the
boundary dynamics is speeded-up only by $N$.

The reversible stationary measures for the exclusion process with
Neumann boundary conditions are the uniform measures with a fixed
total number of particles. The grand canonical versions are the
product Bernoulli measures with a fixed density. For this reason, we
take one of these measures as reference measure.

There is an important difference between our model and the exclusion
process with Dirichlet boundary conditions. Recall the definition the
functional $J_{T,H}$ introduced in \eqref{1-01} and \eqref{1-01b}. For
the sake of this discussion, denote by $J^{\rm DBC}_{T,H}$ the
corresponding functional in the context of exclusion dynamics with
Dirichlet boundary conditions \cite{BDGJL2003, BLM09, FLM2011,
FGN_LPP}. While $J^{\rm DBC}_{T,H}$ is defined on the set
$D([0,T], \ms M_{\rm ac})$, in the present context, the functionals
$J_{T,H}$ are defined only on the subset
$D_{\mc E}([0,T], \ms M_{\rm ac})$ of trajectories with finite energy
because only for such trajectories are the boundary densities well
defined. As a consequence, in the two-blocks estimate, the usual
empirical density,
$(2N\epsilon +1)^{-1} \sum_{y\in \Lambda_N , |y-x|<\epsilon} \eta_y$
which, as a function of $x$, has jumps needs to be replaced by a
smooth approximation.  See \eqref{n03} below.

\subsection*{A super-exponential estimate}

We follow the proofs presented in \cite[Section 3]{BLM09},
\cite[Section 6]{FLM2011}, \cite{FGN_LPP}. Denote by
$\color{bblue} \nu_N$ the Bernoulli product measure on $\Omega_N$ with
density $1/2$ and by $\bb D_N$ the Dirichlet form given by
\begin{equation*}
\bb D_N(f) \;:=\;
\langle\,  -\, L^{\rm bulk}_{N} f  \,,\, f  \,
\rangle_{\nu_N}\;,
\quad f\colon \Omega_N \to \bb R_+\;.
\end{equation*}
Next result is \cite[Lemma 3.1]{BLM09} adapted to the present context.
The proof is elementary and left to the reader. It relies on a Schwarz
inequality.

\begin{lemma}
\label{eq:dir_car}
There exists a finite constant $C_0$, which only depends on the
parameters $\alpha$, $\beta$, $A$, $B$, such that
\begin{equation*}
\langle\, \ms L_{N} f \,,\, f  \,
\rangle_{\nu_N}
\;\le\; -\, \bb D_N(f) \;+\; C_0\, N\, E_{\nu_N}[\, f^2\, ]
\end{equation*}
for all $ f\colon \Omega_N \to \bb R_+$.
\end{lemma}

Given a cylinder function $h$, that is a function on $\{0,1\}^{\bb
Z}$ depending on $\eta_x$, $x\in\bb Z$, only through finitely many
$x$, denote by $\widetilde h(\alpha)$ the expectation of $h$ with
respect to $\nu_\alpha$, the Bernoulli product measure with density
$\alpha$:
\begin{equation*}
\widetilde h(\alpha) \;=\; E_{\nu_\alpha}[h]\; . 
\end{equation*}

Denote by $\color{bblue} \{\tau_x : x\in \bb Z\}$ the group of
translations in $\{0,1\}^{\bb Z}$ so that
$(\tau_x \zeta)_z = \zeta_{x+z}$ for all $x$, $z$ in $\bb Z$ and
configuration $\zeta$ in $\{0,1\}^{\bb Z}$.  Translations are extended
to functions and measures in a natural way. They are also extended to
configurations, functions and measures in $\Omega_N$. In this case,
for $x$, $y\in \{k/N : k\in \bb Z\}$ such that $y$,
$x+y\in \Lambda_N$, $(\tau_x \eta)_y = \eta_{x+y}$.

Fix a strictly decreasing sequence $\{ U_\epsilon : \epsilon>0\}$
converging to $1$: $U_\epsilon > U_{\epsilon'} >1$ for
$\epsilon >\epsilon'>0$, $\lim_{\epsilon\to 0} U_\epsilon =1$.  Recall
from \eqref{n04} the definition of the approximation of the unity
$\phi^\delta$. For $\epsilon>0$, $\pi\in \ms M$, denote by
$\Xi_\epsilon (\pi)$ the measure in $\ms M_{\rm ac}$ defined by
\begin{equation}
\label{n03}
\Xi_\epsilon (\pi) \, (dx)  \;=\; \frac{1}{U_{\epsilon}}
\int_0^1 \phi^\epsilon(y-x) \, \pi(dy) \, dx \;.
\quad\text{Let}\;\; \pi^{N,\epsilon} \;=\; \Xi_\epsilon (\pi^N)\;.
\end{equation}
Clearly, $\pi^{N,\epsilon}$ belongs to $\ms M_{\rm ac}$ for $N$
sufficiently large because $U_\epsilon>1$. Denote its density by
$\color{bblue} u^{N,\epsilon}$. We have just pointed out that
$0\le u^{N,\epsilon}(x) \le 1$ for $N$ large. The map
$x\mapsto u^{N,\epsilon}(x)$ is smooth, and, if
$x$ is at distance less than $\epsilon$ from the boundary of the
interval $[0,1]$, $u^{N,\epsilon}(x)$ does not represent the density
of particles around $x$ because the integral is carried over an
interval which does not contain the support of $\phi (\,\cdot\,-x\,)$.

Let $H\in C([0,T]\times [0,1])$ and $h$ a cylinder
function.  For $\epsilon>0$ and $N$ large enough, define
$V_{N,\epsilon}^{H,h} :[0,T]\times \Omega_N \to \bb R$ by
\begin{equation*}
V_{N,\epsilon}^{H,h}(t,\eta) \;=\; \frac 1N 
\sum_{x\in \Lambda_N}
H(t,x/N) \big\{ \tau_xh(\eta) \,-\,
\widetilde h\big(u^{N,\epsilon}(x)\big)\big\}\; .
\end{equation*}
The sum is carried over all $x\in \Lambda_N$ for which the support of
$\tau_xh$ is contained in $\Lambda_N$.  For a function $G\in C([0,T])$
and cylinder functions $h$, $f$ whose supports are contained in
$\bb N$, $-\bb N$, respectively, let
$W_G^\pm: [0,T]\times \Lambda_N \to \bb R$ be defined by
\begin{equation*}
\begin{gathered}
W^{G,h,-}_{N,\epsilon} (t,\eta) \;=\; G(t) \,
\big\{\, h(\eta) - \widetilde h
( u^{N,\epsilon}(\epsilon)  ) \, \big\}\; , \\
W^{G,f,+}_{N,\epsilon} (t,\eta) \;=\; G(t) \,
\big\{\, (\tau_N f) (\eta) - \widetilde f (
u^{N,\epsilon}(1-\epsilon) ) \, \big\}\; .
\end{gathered}
\end{equation*}

\begin{theorem}
\label{hs04}
Fix $H$ in $C([0,T]\times [0,1])$, $G$ in $C([0,T])$, a cylinder
function $h$ whose support ia contained in $\bb N$, a sequence of
configurations $\{\eta^N \in \Omega_N : N\ge 1\}$ and $\delta>0$. Then
\begin{equation*}
\begin{gathered}
\lim_{\epsilon\to 0}\,\limsup_{N\to\infty} \frac 1N  \log
\bb P_{\eta^N} \Big [\, \Big |
\int_0^{T}V_{N,\epsilon}^{H,h}(t,\eta_t)\, dt \, \Big | \,>\, \delta
\Big] \; =\; -\infty \; , 
\\
\lim_{\epsilon\to 0}\,\limsup_{N\to\infty}  \frac 1N  \log
\bb P_{\eta^N} \Big [\, \Big | \int_0^{T}
W^{G,h,-}_{N,\epsilon} (t,\eta_t) \, dt \Big | \,>\, \delta \, \Big] \; =\;
-\infty  \; .
\end{gathered}
\end{equation*}
\end{theorem}

A similar result holds if the cylinder functions $h$ has support
contained in $-\bb N$ and the minus sign in
$W^{G,h,-}_{N,\epsilon} (t,\eta_t)$ is replaced by a plus sign.  The
proof of this result follows from Lemma~\ref{eq:dir_car} and the
computation presented in the proof of \cite[Lemma 3.2]{l92}.

\subsection*{An energy estimate}

The next result is Lemma 3.3 and Corollary 3.4 in \cite{BLM09}.
The proof is similar and the details are left to the reader.

\begin{proposition}
\label{p_energy}
Fix a sequence $\{G_j : j\ge 1\}$ of functions in
$C^{0,1}([0,T]\times [0,1])$ with compact support in
$[0,T]\times(0,1)$ and a sequence $\{\eta^N \in \Omega_N : N\ge 1\}$
of configurations.  There exists a finite constant $C_0$, depending
only on $\alpha$, $\beta$, $A$, $B$, such that 
\begin{equation*}
\limsup_{\epsilon \to 0} \limsup_{N\to\infty} \frac 1N  \log
\bb P^N_{\eta^N} \Big [\, \max_{1\le j\le k} \mc Q_{G_j}
(u^{N, \epsilon} ) \,\ge\,  \ell \,\Big] 
\; \le \; -\ell \; +\; C_0 \, (T+1)\;.
\end{equation*}
for all $k$, $\ell \ge 1$.
\end{proposition}

\subsection*{Upper bound}

The upper bound proof relies on the super-exponential estimate
presented in Theorem \ref{hs04} and on the energy estimate stated in
Proposition \ref{p_energy}. It is similar to the one presented in
\cite[Subsection 6.3]{FLM2011}. As a consequence of the argument, the
rate function can be set as $+\infty$ for trajectories that are not
absolutely continuous with respect to the Lebesgue measure or which do
not have finite energy. In other words, in the proof of the upper
bound one can set $I_{[0,T]}(\pi|\gamma)=+\infty$ for
$\pi \not\in D_{\mc E}([0,T], \ms M_{\rm ac})$.

\subsection*{Lower bound}

We follow the arguments presented in \cite[Subsection 3.4]{BLM09} and
\cite[Subsection 6.4]{FLM2011}. Fix an open set $\ms G$ of
$D([0,T], \ms M)$ and a density profile $\gamma: [0,1] \to [0,1]$.
Recall the definition of the set $\Pi_\gamma$ introduced in Definition
\ref{d04}. Fix a path
$\pi(t,dx) = u(t,x) \, dx \in \Pi_\gamma \cap \ms G$.

Let $(\eta^N : N\ge 1)$ be a sequence of configurations associated to
the density profile in the sense \eqref{eq:associated}.  Denote by
$\bb P^H_{\eta^N}$ the probability measure on $D([0,T], \Omega_N)$
induced by the weakly asymmetric exclusion process with Robin boundary
conditions defined in Section \ref{s4}.

Given two probability measures $P$ and $Q$ we denote by
${\rm Ent}\, (Q\, |\, P)$ the relative entropy of $Q$ with respect to
$P$. By Theorem \ref{p02}, Proposition \ref{l09b} and an elementary
computation,
\begin{equation*}
\lim_{N\to\infty} \frac 1N\,  {\rm Ent}\,
(\, \bb P^H_{\eta^N} \, |\,  \bb P_{\eta^N}\, ) \;=\;  I_{[0,T]}
(\pi |\gamma)\;. 
\end{equation*}
Therefore, since $N^{-1}\, \log (d\bb P^H_{\eta^N} /  d\bb
P_{\eta^N})$ is absolutely bounded, by the proof of the lower bound
presented at \cite[page 277]{KL},
\begin{equation*}
\liminf_{n\to\infty} \frac{1}{N}\,  \log \bb P_{\eta^N}
\big[\, \ms G \, \big] \;\geq\;
-  \,\inf_{\pi \in \ms G \cap \Pi_\gamma}  I_{[0,T]} (u|\gamma) \; . 
\end{equation*}
The lower bound follows from this result and the
$I_{[0,T]}(\,\cdot\,|\gamma)$-density stated in Theorem \ref{mt3}.

\section{Weakly asymmetric exclusion with Robin boundary
conditions}
\label{s4}

Recall the notation introduced in Section \ref{sec2}.  Fix
$H\in C^{1,2}([0,T]\times[0,1])$.  Consider the weakly asymmetric
exclusion process induced by the external field $H$ with Robin
boundary conditions.  The generator of this process, denoted by
$\ms L^{H}_N$, is given by
\begin{equation}
\label{LNH}
\ms L^{H}_N \;=\;  L^{H, \textrm{lb}}_{N}
\;+\; L^{H, \textrm{bulk}}_{N}
\;+\; L^{H, \textrm{rb}}_{N}\;,
\end{equation}
where, for a function $f:\Omega_N\rightarrow \bb{R}$,
\begin{equation*}
(L^{H,\textrm{bulk}}_{N}f) (\eta) \; =\; N^2
\sum_{x\in\Lambda_N^0}
e^{-\, (\eta_{x+\mf e}-\eta_{x})\, [\, H_t(x+\mathfrak e)-H_t(x)\,]}
\big\{\, f(\sigma^{x,x+\mathfrak e}\eta) \,-\, f(\eta)\, \big\}\,,
\end{equation*}
\begin{equation*}
( L_{N}^{H,\textrm{lb}}f)(\eta) \;=\;
\frac{N}{A} \, \Big\{\, e^{H_t(\mathfrak e)}\, \alpha\,
(1-\eta_{\mathfrak e}) \,+\, e^{-H_t(\mathfrak e)}\, (1-\alpha)\,
\eta_{\mathfrak e}\, \Big\} \, \big\{ \, f(\sigma^{\mathfrak e} \eta)
\,-\, f(\eta) \, \big\} \,,
\end{equation*}
\begin{equation*}
( L_{N}^{H,\textrm{rb}}f)(\eta) \;=\;
\frac{N}{B} \, \Big\{\, e^{H_t(\mathfrak r)}\, \beta\,
(1-\eta_{\mathfrak r}) \,+\, e^{-H_t(\mathfrak r)}\, (1-\beta)\,
\eta_{\mathfrak r}\, \Big\} \, \big\{ \, f(\sigma^{\mathfrak r} \eta)
\,-\, f(\eta) \, \big\} \,.
\end{equation*}

Denote by $\color{bblue} \mathbb P_{\mu}^H$, $\mu$ a probability
measure on $\Omega_N$, the measure on $D([0,T],\Omega_N)$ induced by
the Markov process with infinitesimal generator $\ms L_{N}^{H}$ and
initial state $\mu$.  Let $\color{bblue} \mathbb Q^H_{\mu} $ be the
probability on $ D([0,T],\ms M)$ induced by the empirical measure
$\bs \pi$ and the measure $\mathbb P_{\mu}^H$:
$ \mathbb Q^H_{\mu} = \mathbb P^H_{\mu} \,\circ\, \bs \pi^{-1}$.

\begin{theorem}
\label{p02}
Fix a measurable profile $\gamma: [0,1] \to [0,1]$. Let
$\{\mu_N: N\ge 1\}$ be a sequence of probability measures on
$\Omega_N$ associated to $\gamma$ in the sense \eqref{eq:associated}.
Then, the sequence of probability measures $\mathbb Q^H_{\mu_N}$
converges to the probability measure $\mathbb Q^H$ concentrated on the
trajectory $\pi(t,dx) =u (t,x)\, dx$, whose density $u$ is the unique
weak solution of
\begin{equation}
\label{5-01}
\left\{
\begin{aligned}
& \partial_t u \;=\; \Delta u \,-\,
2\, \nabla  \{ \sigma(u) \ \nabla  H \}\; , \\
& \nabla  u_t (1) \,-\, 2\, \sigma(u_t(1)) \, \nabla  H_t(1) \,=\,
\mf p_{\beta, B} \big(\, u_t(1)\,,\, H_t(1)\, \big) \;, \\
& \nabla  u_t (0) \,-\, 2\, \sigma(u_t(0)) \, \nabla  H_t(0) \,=\,
-\, \mf p_{\alpha, A} \big(\, u_t(0)\,,\, H_t(0)\, \big) \;, \\
& u(0, \cdot) = \gamma (\cdot)\; .
\end{aligned}
\right.
\end{equation}
\end{theorem}

The proof of this result is by now classical and divided in several
steps. One first proves tightness. Then, one shows that any limit
point of the sequence $\mathbb Q^H_{\mu_N}$ is concentrated on
trajectories $\pi (t,dx) = u(t,x)\, dx$ whose density belongs to
$\ms L^2 (0,T; \mc H^1)$, where the $\mc H^1$ is the Sobolev space
introduced in Section \ref{sec07}. Finally, one shows
that limit points of the sequence $\mathbb Q^H_{\mu_N}$ are
concentrated on trajectories which satisfy the identity
\eqref{8-01}. It remains to invoke the uniqueness of weak solutions,
stated in Theorem \ref{mt5}, to complete the proof.  The technical
details are standard and the arguments rely on the bound presented in
Lemma \ref{eq:dir_car}. We refer to \cite{KL, BMNS2017}.

\appendix
\section{The Robin Laplacian}
\label{sec04}

We present in this section some results on the Robin Laplacian needed
in the previous section. We refer to \cite{M83, S08} for
details. Denote by $\color{bblue} \Delta_R$ the Laplacian on $[0,1]$
with Robin boundary conditions, sometimes called the Robin Laplacian
\cite[Section 4.3]{S08}.

Fix $\lambda\in \bb R$ and consider the eigenvalue problem
\begin{equation}
\label{6-01}
\left\{
\begin{aligned}
& -\, \Delta f =  \lambda \, f \; , \\
& (\nabla  f ) (0) \,=\, A^{-1} \, f(0) \;,\\
& (\nabla  f)(1) \,=\, -\, B^{-1} f(1) \;.
\end{aligned}
\right.
\end{equation}
This problem has only the trivial solution $f=0$ for $\lambda\le
0$. For $\lambda>0$, the equation $-\, \Delta f = \lambda \, f$ can be
turned into a two-dimensional ODE which yields that the solutions of
\eqref{6-01} are given by
$f(x) = a \, [ \, \cos (\sqrt{\lambda} x) + b \, \sin (\sqrt{\lambda}
x)\,]$ for some $a$, $b\in \mathbb R$. The boundary conditions are
satisfied if and only if
\begin{equation}
\label{6-09}
\tan \sqrt{\lambda} \;=\; (A+B)\,
\frac{\sqrt{\lambda}}{\lambda AB -1}\;,
\end{equation}
in which case $b= (A\sqrt{\lambda})^{-1}$. This identity excludes
$\lambda=0$ from the set of eigenvalues of the Robin Laplacian.

An analysis of \eqref{6-09} shows that it has a countable set of
solutions $\color{bblue} \{\lambda_j : j\ge 1\}$, where $0<\lambda_1$,
$\lambda_j < \lambda_{j+1}$ and $\lambda_j \sim j^2$ in the sense that
there exists $0<c_0<c_1<\infty$ such that
\begin{equation}
\label{6-12}
c_0 \, j^2 \;\le\; \lambda_j \;\le\; c_1\, j^2 \;\;\text{for all $j\ge
1$} \;.
\end{equation}
Denote by $\color{bblue} \{f_j : j\ge 1\}$ the associated orthonormal
eigenvectors, which form a basis of $\ms L^2([0,1])$. By the previous
analysis,
\begin{equation}
\label{6-13}
f_j(x) \;=\; a_j \,\big\{\, \cos (\sqrt{\lambda_j} x)
\;+\; \frac{1}{A\sqrt{\lambda_j}}  \, \sin (\sqrt{\lambda_j}
x)\, \big\}\;,
\end{equation}
where $a_j$ is chosen for $f_j$ to have $\ms L^2$-norm equal to $1$. It
can be shown that $|a_j|\le C_0$ for all $j\ge 1$, where $C_0$ is a
finite constant depending only on $A$ and $B$. Therefore, by
\eqref{6-12},
\begin{equation}
\label{6-11}
\Vert \, f_j \, \Vert_\infty \;\le\; C_0 \;,
\quad \Vert \, \nabla ^n f_j \, \Vert_\infty
\;\le\; C_0 \, (\lambda_j)^{n/2} \;\le\; C_0 \, j^n
\end{equation}
for all $j\ge 1$, $n\ge 1$.
A straightforward computation provides a formula for the Green
function of the Robin Laplacian: Let
$K_R: [0,1] \times [0,1] \to \mathbb R_+$ be given by
\begin{equation}
\label{6-10}
K_R(x,y) \;=\; \frac{1}{1+A+B}\,
\left\{
\begin{aligned}
& (B+1-x)\, (A+y)\;, \quad 0\le y\le x \le 1\;,  \\
& (B+1-y)\, (A+x)\;, \quad 0\le x\le y \le 1 \;.
\end{aligned}
\right.
\end{equation}
Denote by $K_R$ the integral operator defined by
\begin{equation*}
(K_R f)(x) \;=\; \int_0^1 K_R(x,y)\, f(y)\; dy\;.
\end{equation*}
Then, $K_R = (-\Delta_R)^{-1}$.

Denote by $\color{bblue} \mc H_R$ the Hilbert space obtained by
completing the space
\begin{equation*}
C^2_{A,B}([0,1]) \;=\; \big\{\, f \in C^2([0,1]) : (\nabla  f )
(0) \,=\, A^{-1} \, f(0) \;,\, (\nabla  f)(1) \,=\, -\, B^{-1} f(1)
\,\big \}
\end{equation*}
endowed with the scalar product $\<\,\cdot\,,\,\cdot\,\>_{\mc H_R}$
defined by
\begin{equation}
\label{6-05}
\begin{aligned}
\<\,f\,,\,g\,\>_{\mc H_R} \;& =\; \<\,f\,,\, (-\, \Delta_R) g\,\> \\
& =\;
\frac{1}{A}\, f(0)\,g(0) \;+\; \int_0^1 (\nabla   f)(x)\,
(\nabla   g)(x) \; dx \;+\; \frac{1}{B}\, f(1)\, g(1) \;.
\end{aligned}
\end{equation}
Denote by $\color{bblue} \Vert f\Vert_{\mc H_R}$ the norm induced by the
scalar product $\<\,\cdot\,,\,\cdot\,\>_{\mc H_R}$. We have that
\begin{equation}
\label{6-16}
\Vert f\Vert^2_{\mc H_R} \;=\;
\sum_{k\ge 1} \lambda_k\, \< f \,,\, f_k\>^2 \;,
\end{equation}
for all $f\in \mc H_R$.

Recall from \eqref{n08} the definition of the Sobolev space $\mc H^1$.
The norms $\Vert \,\cdot\, \Vert_{\mc H_R}$ and
$\Vert \,\cdot\, \Vert_{\mc H^1}$ are equivalent. There exist finite
constants $0<C_1<C_2<\infty$ such that
\begin{equation}
\label{6-14}
C_1\, \Vert f\Vert_{\mc H^1} \;\le\; \Vert f\Vert_{\mc H_R}
\;\le\; C_2\, \Vert f\Vert_{\mc H^1}
\end{equation}
for all $f\in C^2([0,1])$. In particular, the spaces $\mc H_R$ and
$\mc H^1$ coincide.

In terms of the
eigenfunctions $f_k$,
\begin{equation}
\label{6-07}
\Vert \, f\, \Vert^2_{\mc H_R} \;=\;
\sum_{k\ge 1} \lambda_k \, |\, \<\,f \,,\, f_k\,\>\,|^2 \;.
\end{equation}
Moreover, a straightforward computation yields that for all
$f\in C^2_{A,B}([0,1])$,
\begin{equation}
\label{6-02}
\Vert \, f \, \Vert^2_\infty \;\le\; 2\, (A\vee 1)\,
\Vert f\Vert^2_{\mc H_R} \; .
\end{equation}

Fix a function $f$ in $\mc H^1$. It is well known that there exists a
continuous function $f^{(c)}: [0,1] \to \mathbb R$ (actually H\"older
continuous,
$|f^{(c)}(y) - f^{(c)}(x)| \le \Vert f \Vert_{2}
|y-x|^{1/2}$) such that $f = f^{(c)}$ almost surely. Moreover, for all
$h\in C^1([0,1])$,
\begin{equation}
\label{1-03}
\int_0^1 f \, \nabla  h\;dx \;=\; f^{(c)}(1)\, h(1)
\;-\; f^{(c)}(0)\, h(0)
\;-\; \int_0^1 \nabla  f \,  h\;dx\;.
\end{equation}
The next result provides an explicit formula for $f^{(c)}$ in terms of
the eigenvectors $f_k$.

\begin{lemma}
\label{l05}
There exists a finite constant $C_0$ such that
\begin{equation*}
\sum_{k\ge 1} \big|\, \< f \,,\, f_k\>\, \big| \;\le\; C_0
\, \Vert \, f\,\Vert_{\mc H_R}
\end{equation*}
for all $f \in \mc H^1$.  In particular,
$\sum_{k\ge 1} \< f \,,\, f_k\>\, f_k(\cdot)$ defines a continuous
function, and, for almost all $x\in [0,1]$,
\begin{equation}
\label{6-15}
f(x) \;=\; \sum_{k\ge 1} \< f \,,\, f_k\>\, f_k(x)\;.
\end{equation}
\end{lemma}

\begin{proof}
By \eqref{6-14}, $f$ belongs to $\mc H_R$. By Schwarz inequality,
\begin{equation*}
\Big( \sum_{k\ge 1} \big|\, \< f \,,\, f_k\>\, \big| \,\Big)^2
\;\le\;
\sum_{k\ge 1} \lambda_k\, \big|\, \< f \,,\, f_k\>\, \big|^2 \;
\sum_{k\ge 1} \frac{1}{\lambda_k}\;\cdot
\end{equation*}
The second sum is finite by \eqref{6-12} and the first one is finite
by \eqref{6-07}. This proves the first assertion.

Since each function $f_k$ is continuous, and a summable sum of
uniformly bounded continuous functions is continuous,
$\sum_{k\ge 1} \< f \,,\, f_k\>\, f_k(\cdot)$ defines a continuous
function. As $(f_k : k\ge 1)$ forms an orthonormal basis of
$\ms L^2([0,1])$, $f = \sum_{k\ge 1} \< f \,,\, f_k\>\, f_k$ as an
identity in $\ms L^2([0,1])$. In particular, these functions are equal
almost everywhere.
\end{proof}

Denote by $\color{bblue} (P^{(R)}_t: t\ge 0)$ the semigroup in
$\ms L^2([0,1])$ generated by the Robin Laplacian: For any function
$f\in \ms L^2([0,1])$, $t>0$,
\begin{equation}
\label{5-09}
P^{(R)}_t f \;=\; \sum_{k\ge 1} e^{-\lambda_k t}\, \< f \,,\, f_k\>\, f_k\;.
\end{equation}
In particular, for each $t\ge 0$, $P^{(R)}_t$ is a symmetric operator
in $\ms L^2([0,1])$ and $P^{(R)}_t f\in C^\infty ([0,1])$ for all
$f\in \ms L^2([0,1])$. Moreover, as $P^{(R)}_t$ is symmetric, by
\eqref{6-07}, $P^{(R)}_t$ is a contraction in $\mc H_R$ and
$\ms L^2([0,1])$:
\begin{equation}
\label{6-18}
\begin{gathered}
\Vert\, P^{(R)}_t f\,\Vert^2_{\mc H_R} \;=\;
\sum_{k\ge 1} e^{-2\lambda_k t}\, \lambda_k\,
|\, \< f \,,\, f_k\>\,|^2  \;\le\;
\Vert\,  f\,\Vert^2_{\mc H_R}\;, \\
\Vert\, P^{(R)}_t f\,\Vert^2_{2} \;=\;
\sum_{k\ge 1} e^{-2\lambda_k t}\,
|\, \< f \,,\, f_k\>\,|^2  \;\le\;
\Vert\,  f\,\Vert^2_{2}\;.
\end{gathered}
\end{equation}

Let $f\in \ms L^2([0,1])$ be given by
$f = \sum_{k\ge 1} \< f \,,\, f_k\> \, f_k$. For each $t>0$, there
exists a finite constant $C_0 (t)$ such that
\begin{equation}
\label{6-17}
\begin{aligned}
\Vert\, P^{(R)}_t f \,\Vert^2_{\infty} \;\le\;
C_0(t) \, \Vert\, f \,\Vert^2_{2} \;, \quad
\Vert\, P^{(R)}_t f \,\Vert^2_{\mc H_R} \;\le\;
C_0(t) \, \Vert\, f \,\Vert^2_{2}\;.
\end{aligned}
\end{equation}
Indeed, by \eqref{6-07} and since $P^{(R)}_t$ is symmetric and
$P^{(R)}_t f_k = e^{-\lambda_k t} f_k$,
\begin{equation*}
\Vert\, P^{(R)}_t f \,\Vert^2_{\mc H_R} \;=\;
\sum_{k\ge 1} \lambda_k \, e^{-2\lambda_k t}
\big|\, \< f \,,\, f_k\>\, \big|^2 
\;\le\; C_0(t) \, \sum_{k\ge 1} 
\big|\, \< f \,,\, f_k\>\, \big|^2
\;=\; C_0(t) \, \Vert\, f \,\Vert^2_{2} 
\end{equation*}
for some finite constant $C_0(t)$. On the other hand, by Schwarz
inequality and \eqref{6-11},
\begin{equation*}
\Vert\, P^{(R)}_t f \,\Vert^2_{\infty} \;=\;
\Big\Vert\, \sum_{k\ge 1} e^{-\lambda_k t}
\, \< f \,,\, f_k\>\, f_k   \, \Big\Vert^2_{\infty}
\;\le\; \sum_{k\ge 1} e^{-2\lambda_k t}
\sum_{k\ge 1} \, \< f \,,\, f_k\>^2
\;=\; C_0(t) \, \Vert\, f \,\Vert^2_{2} 
\end{equation*}
for some finite constant $C_0(t)$.

\begin{lemma}
\label{l07}
There exists a finite constant $C_0$ such that
\begin{equation*}
\Vert\, P^{(R)}_t f \,-\, f\,\Vert_2 \;\le \;
C_0\,  t^{1/3}\, \, \Vert\, f\, \Vert_{\mc H_R}
\end{equation*}
for all $t\ge 0$, $f \in \mc H_R$.
\end{lemma}

\begin{proof}
Since $(f_k : k\ge 1)$ is an orthonormal basis of $\ms L^2([0,1])$,
\begin{equation*}
\Vert\, P^{(R)}_t f \,-\, f\,\Vert^2_2 \;=\;
\sum_{k\ge 1} \big[\, e^{-\lambda_k \, t} \,-\, 1\,\big]^2 \,
\big|\, \< f \,,\, f_k\>\, \big|^2 \;.
\end{equation*}
Fix $k_0\ge 1$. Since the sequence $\lambda_k$ increases, the
right-hand side can be bounded by
\begin{equation*}
\big[\, e^{-\lambda_{k_0} \, t} \,-\, 1\,\big]^2
\sum_{k= 1}^{k_0-1}  \big|\, \< f \,,\, f_k\>\, \big|^2
\;+\; \frac{1}{\lambda_{k_0}} \,
\sum_{k \ge k_0}  \lambda_k\, \big|\, \< f \,,\, f_k\>\, \big|^2\;.
\end{equation*}
The first sum is bounded by $\Vert\, f\, \Vert^2_2$. In view of
\eqref{6-07}, the second one is bounded by
$\Vert\, f\, \Vert^2_{\mc H_R}$ so that
\begin{equation*}
\Vert\, P^{(R)}_t f \,-\, f\,\Vert^2_2 \;\le \;
\big[\, 1  \,-\, e^{-\lambda_{k_0} \, t}  \,\big]^2
\,  \Vert\, f\, \Vert^2_2 \;+\;
\frac{1}{\lambda_{k_0}} \Vert\, f\, \Vert^2_{\mc H_R} \;.
\end{equation*}
As $1-e^{-x} \le x$, $x>0$, and since, by \eqref{6-14},
$\Vert\, f\, \Vert_2 \le C_0 \Vert\, f\, \Vert_{\mc H_R}$ for some
finite constant $C_0$,
\begin{equation*}
\Vert\, P^{(R)}_t f \,-\, f\,\Vert^2_2 \;\le \;
\Big\{\, C_0\, (\, \lambda_{k_0} \, t  \,)^2
\;+\; \frac{1}{\lambda_{k_0}} \,\Big\}\, \Vert\, f\, \Vert^2_{\mc H_R} \;.
\end{equation*}
To complete the proof, it remains to choose $k_0$ such that
$\lambda^{-3}_{k_0} \sim t^2$.
\end{proof}

\begin{lemma}
\label{l08}
There exists a finite constant $C_0$ such that
\begin{equation*}
\Vert \, P^{(R)}_t f \,-\, f  \,\Vert_\infty \;\le \;
C_0\,  t^{1/5}\, \, \Vert\, f\, \Vert_{\mc H_R}
\end{equation*}
for all $t\ge 0$, $f \in C([0,1]) \cap \mc H_R$.
\end{lemma}

\begin{proof}
Fix $x\in [0,1]$. Since $f$ is continuous, by \eqref{6-15} and
\eqref{6-11},
\begin{equation*}
\big\{ \, P^{(R)}_t f (x) \,-\, f (x) \,\big\}^2 \;\le\;
C_0^2\, \Big(\, \sum_{k\ge 1} \big[\, 1 \,-\, e^{-\lambda_k \, t} \,\big] \,
\big|\, \< f \,,\, f_k\>\, \big|\,\Big)^2
\end{equation*}
for some finite constant $C_0$. By Schwarz inequality and
\eqref{6-07}, the right-hand side is bounded by
\begin{equation*}
C_0^2\, \sum_{k\ge 1}  \frac{1}{\lambda_k} \,
\big[\, 1 \,-\, e^{-\lambda_k \, t} \,\big]^2 \,
\sum_{k\ge 1}  \lambda_k \,
\,
\big|\, \< f \,,\, f_k\>\, \big|^2 \;=\;
C_0^2\, \sum_{k\ge 1}  \frac{1}{\lambda_k} \,
\big[\, 1 \,-\, e^{-\lambda_k \, t} \,\big]^2\Vert\, f\, \Vert^2_{\mc H_R}  \;.
\end{equation*}
It remains to estimate the sum.  Fix $k_0\ge 1$.  Since the sequence
$\lambda_k$ increases, as $1-e^{-x} \le x$, $x>0$, by \eqref{6-12},
the sum is less than or equal to
\begin{equation*}
C\, \big[\, 1 \,-\, e^{-\lambda_{k_0} \, t} \,\big]^2
\;+\; \sum_{k \ge k_0}  \frac{1}{\lambda_{k}} \;\le\;
C\, \Big\{ \, (k^2_0 \, t)^2 \;+\; \frac{1}{k_0} \,\Big\}
\end{equation*}
for some finite constant $C$. It remains to choose $k_0$ such that
$k_0^5 \sim t^{-2}$.
\end{proof}

\section{Initial-value problems with Robin boundary conditions}
\label{sec06}

We present in this section some result on the initial-boundary value
problems \eqref{1-06}, \eqref{5-01}.  Recall the definition of the
Sobolev space $\mc H^1$ introduced in \eqref{n08}.  Fix a function
$\phi\in \ms L^2([0,1])$, and consider the initial-boundary problem

\begin{equation}
\label{1-06h}
\begin{cases}
\partial_tu \,=\, \Delta u\\
(\nabla  u) (t,0) \,=\, A^{-1} \,u(t,0) \\
(\nabla  u)(t,1) \,=\, -\, B^{-1}\, u(t,1) \\
u(0,\cdot) \,=\, \phi (\cdot)\;.
\end{cases}
\end{equation}

\begin{definition}
\label{d01}
A function $u$ in $\ms L^2(0,T; \mc H^{1})$ is said to be a
generalized solution in the cylinder $[0,T]\times [0,1]$ of the
equation \eqref{1-06h} if
\begin{equation*}
\begin{aligned}
& \int_0^1  u_t \, H_t \; dx \;-\; \int_0^1 \phi \, H_0 \; dx
\;-\; \int_0^t ds \, \int_0^1  u_s \, \partial_s H_s \; dx \\
& \quad =\;- \int_0^t ds \, \int_0^1 \nabla  u_s\, \nabla  H_s \; dx
\;-\; \int_0^t \big\{\,  \frac{1}{B}\, u_s(1)\, H_s(1)
\,+\, \frac{1}{A} \, u_s(0)\, H_s(0) \,\big\} \; ds
\end{aligned}
\end{equation*}
for every $0<t\le T$, function $H$ in $C^{1,2}([0,T]\times [0,1])$.
\end{definition}

Next result is proved in \cite{BMNS2017}. We present it here in sake
of completeness.

\begin{theorem}
\label{t01}
For each $\phi\in \ms L^2([0,1])$, there exists one and only one
generalized solution to \eqref{1-06h}. The solution is smooth in
$(0,\infty) \times [0,1]$ and can be represented as
$u(t,x) = (P^{(R)}_t \phi)(x)$, where $P^{(R)}_t$ is the semigroup
associated to the Robin Laplacian. Moreover,
\begin{equation}
\label{10-01}
\min \{\, 0 \,,\, {\rm ess\, inf }\, \phi \,\}
\,\le\, u(t,x) \,\le\, \max \{\, 0 \,,\,
{\rm ess\, sup }\, \phi \,\}
\end{equation}
for all $(t,x) \in \mathbb R_+ \times [0,1]$. Finally, if $\phi (x) \le b$
for some $b>0$, then, for each $t_0>0$ there exists $\epsilon >0$ such
that $u(t,x) \le b -\epsilon$ for all
$(t,x) \in [t_0, \infty) \times [0,1]$. Analogously,  if $\phi (x) \ge a$
for some $a<0$, then, for each $t_0>0$ there exists $\epsilon >0$ such
that $u(t,x) \ge a + \epsilon$ for all
$(t,x) \in [t_0, \infty) \times [0,1]$.
\end{theorem}

\begin{proof}
Existence and uniqueness of generalized solutions, as well as their
representation in terms of the semigroup $P^{(R)}_t$ is the content of
Theorems 1 and 3 in \cite[Section VI.2]{M83}.

We turn to \eqref{10-01}. Assume first that $\phi$ belongs to
$\mc H^1$.  By \eqref{6-14}, $\phi \in \mc H_R$, and, by
Lemma~\ref{l08}, $u(t)$ converges to $\phi$ in $\ms L^\infty ([0, 1])$
as $t\to 0$.  Since the solution is smooth in
$(0,\infty) \times [0,1]$, by the maximum principle stated in Theorems
2 and 3 of \cite[Chapter 3]{PW84},
\begin{equation*}
\min \{\, 0 \,,\, \inf_{0\le y\le 1} u(t_0,y)  \,\}
\,\le\, u(t,x) \,\le\, \max \{\, 0 \,,\,
\sup_{0\le y\le 1} u(t_0,y) \,\}
\end{equation*}
for all $(t,x) \in [t_0, \infty) \times [0,1]$. Letting $t_0\to 0$, as
$u(t_0)$ converges to $\phi$ in $\ms L^\infty([0,1])$, yields
\eqref{10-01}.

To extend this result to $\phi\in \ms L^2([0,1])$, we consider a
sequence $\phi_n \in \mc H^1$ which converges to $\phi$ in
$\ms L^2([0,1])$ and such that
${\rm ess\, inf }\, \phi \,\le\, \phi_n (x) \,\le\, {\rm ess\,
sup }\, \phi$ for all $0\le x\le 1$. Denote by $u^n$ the solution of
\eqref{1-06h} with initial condition $\phi_n$. Fix $t>0$. By the
result for initial conditions in $\mc H^1$,
\begin{equation*}
\begin{aligned}
& \min \{\, 0 \,,\, {\rm ess\, inf }\, \phi \,\}
\;\le\;
\min \{\, 0 \,,\, \inf_{0\le y\le 1} \phi_n(y) \,\} \\
& \quad \;\le\;
u^n(t,x)
\;\le\;
\max \{\, 0 \,,\, \sup_{0\le y\le 1} \phi_n(y)  \,\}
\,\le\, \max \{\, 0 \,,\,
{\rm ess\, sup }\, \phi \,\}\;.
\end{aligned}
\end{equation*}
for all $0\le x\le 1$. By \eqref{6-17}, $u^n(t)$ converges to $u(t)$
in $\ms L^\infty ([0,1])$. This completes the proof of \eqref{10-01}.

Assume that $\phi (x) \le b$ for some $b>0$.  By \eqref{10-01},
$u(t,x) \le b$ for all $t \ge 0$, $0\le x\le 1$.  Fix $t_0>0$, and
assume that $\max_{0\le x\le 1} u(t_0,x) = b$. As $b>0$, the boundary
conditions imply that the maximum cannot be attained at the
boundary. On the other hand, if it is attained at the interior, by
Theorem 2 of \cite[Chapter 3]{PW84} and by the smoothness of the
solution, $u(t,x) = b$ for all $(t,x) \in (0,t_0] \times [0,1]$. This
is not possible at the boundary. Therefore,
$\max_{0\le x\le 1} u(t_0,x) < b$. By the maximum principle, this
bound can be extended to all $(t,x) \in [t_0, \infty) \times [0,1]$.
The same argument applies to the lower bound.
\end{proof}

Let $\bar\rho\in \ms M_{\rm ac}$ be the unique stationary solution of
the equation \eqref{1-06}. That, is $\bar\rho$ is the solution of the
elliptic equation
\begin{equation}
\label{1-05}
\begin{cases}
\Delta \rho  \,=\,  0\\
(\nabla   \rho) (0) \,=\, A^{-1} [\,\rho (0)-\alpha\,] \\
(\nabla  \rho)(1) \,=\, B^{-1} [\, \beta - \rho (1)] \;.
\end{cases}
\end{equation}
An elementary computation yields that
$\bar\rho$ is given by
\begin{equation*}
\bar\rho(x) \;=\; \frac{\alpha (1+B) + \beta A}{1+B+A}
\;+\; \frac{(\beta-\alpha)\, x}{1+B+A}\;\cdot
\end{equation*}
Note that $\bar\rho$ is the linear interpolation between
$\bar\rho(-A) = \alpha$ and $\bar\rho(1+B) = \beta$.

\begin{definition}
\label{d02}
Fix $\gamma: [0,1] \to [0,1]$.  A function $u$ in
$\ms L^2(0,T; \mc H^1)$ is said to be a generalized solution in the
cylinder $[0,T]\times [0,1]$ of the equation \eqref{1-06} if
$u(t,x) - \bar\rho$ is a generalized solution of the initial-boundary
problem \eqref{1-06h} with initial condition $\gamma - \bar\rho$.
\end{definition}

Therefore, a function $u$ in $\ms L^2(0,T; \mc H^1)$ is a
generalized solution in the cylinder $[0,T]\times [0,1]$ of the
equation \eqref{1-06} if
\begin{equation*}
\begin{aligned}
& \int_0^1  u_t \, H_t \; dx \;-\; \int_0^1 \gamma \, H_0 \; dx
\;-\; \int_0^t ds \, \int_0^1  u_s \, \partial_s H_s \; dx
=\;- \int_0^t ds \, \int_0^1 \nabla  u_s\, \nabla  H_s \; dx \\
& \quad
\;-\; \int_0^t \Big\{\,  \frac{1}{B}\, [\, u_s(1) - \beta\, ] \, H_s(1)
\,+\, \frac{1}{A} \, [\, u_s(0) - \alpha\,] \, H_s(0) \,\Big\} \; ds
\end{aligned}
\end{equation*}
for every $0<t\le T$, function $H$ in $C^{1,2}([0,T]\times [0,1])$.

\begin{theorem}
\label{mt4}
Fix $\gamma: [0,1] \to [0,1]$.  There exists a unique generalized
solution of \eqref{1-06}.  The solution is smooth in
$(0,T]\times [0,1]$ and satisfies the bounds
\begin{equation}
\label{1-08}
\min \{\, \alpha \,,\, {\rm ess\, inf }\,  \gamma \,\}
\,\le\, u(t,x) \,\le\, \max \{\, \beta \,,\,
{\rm ess\, sup }\, \gamma  \,\}
\end{equation}
for all $(t,x) \in [0,T] \times [0,1]$. Moreover, for all
$0<t_0 \le T$ there exists $\epsilon>0$ such that
$\epsilon \le u(t,x) \le 1-\epsilon$ for all
$(t,x) \in [t_0,T]\times [0,1]$.
\end{theorem}

\begin{proof}
The proof of this result is similar to the one of Theorem \ref{t01}.
\end{proof}

Fix $\gamma: [0,1]\to [0,1]$, and denote by $u^{(\gamma)}$ the unique
weak solution of \eqref{1-06} with initial condition $\gamma$.

\begin{lemma}
\label{l14}
There exists a
finite constant $C_0$, which depends only on $\alpha$, $\beta$, $A$,
$B$ such that
\begin{equation*}
\begin{aligned}
& \int_0^t ds \int_0^1 \frac{(\nabla  u_s)^2}{\sigma(u_s)} \, dx
+\; \frac{1}{A}  \int_0^t  \,\Big|\,  [\, u_s(0) -\alpha\, ]\,
\log \frac{u_s(0)}{1-u_s(0)} \,\Big| \, ds \\
&\qquad \,+\, \frac{1}{B} \int_0^t
\,\Big| \, [\, u_s(1) -\beta\, ]\,
\log \frac{u_s(1)}{1-u_s(1)}  \,\Big| \, ds \;\le \;
C_0\, t \;+\;  \int_0^1 F_0(\gamma)\, dx \;-\; \int_0^1 F_0(u_t)\, dx
\end{aligned}
\end{equation*}
for all $t>0$ and all $\gamma: [0,1]\to [0,1]$.
\end{lemma}

\begin{proof}
Fix $F\in C^2([0,1])$, an initial profile $\gamma:[0,1]\to [0,1]$, and
denote by $u$ the solution of \eqref{1-06}. Since $u$ is smooth on
$(0,\infty) \times [0,1]$, integrating by parts and in view of the
boundary conditions, for all $0<\delta<t<\infty$,
\begin{equation*}
\begin{aligned}
& \int_0^1 F(u_t)\, dx \;-\; \int_0^1 F(u_\delta)\, dx
\;=\; -\, \int_\delta^t ds \int_0^1 F'' (u_s) \, (\nabla  u_s)^2 \, dx \\
& -\, \int_\delta^t  \frac{1}{A} \, [\, u_s(0) -\alpha\, ]\,  F' (u_s(0) ) \, ds
\,-\, \int_\delta^t  \frac{1}{B} \, [\, u_s(1) -\beta\, ]\,
F' (u_s(1) ) \, ds\;.
\end{aligned}
\end{equation*}
As $u_\delta$ converges to $\gamma$ in $\ms L^2([0,1])$, letting
$\delta\to 0$ yields that for all $t>0$,
\begin{equation*}
\begin{aligned}
&  \int_0^t ds \int_0^1 F'' (u_s) \, (\nabla  u_s)^2 \, dx
+\; \int_0^t  \frac{1}{A} \, [\, u_s(0) -\alpha\, ]\,  F' (u_s(0) ) \,
ds \\
&\quad \,+\, \int_0^t  \frac{1}{B} \, [\, u_s(1) -\beta\, ]\,
F' (u_s(1) ) \, ds \;=\; \int_0^1 F(\gamma)\, dx
\;-\; \int_0^1 F(u_t)\, dx \;.
\end{aligned}
\end{equation*}

Since for each $t>0$, there exists $\epsilon>0$ such that
$\epsilon \le u(s,x) \le 1-\epsilon$ for all
$(s,x) \in [t,\infty) \times [0,1]$, the previous argument can be
applied to the function $F_0 (r) = r \log r + (1-r) \log (1-r)$. It
yields that
\begin{equation}
\label{n05}
\begin{aligned}
& \int_0^t ds \int_0^1 \frac{(\nabla  u_s)^2}{\sigma(u_s)} \, dx
+\; \int_0^t  \frac{1}{A} \, [\, u_s(0) -\alpha\, ]\,
\log \frac{u_s(0)}{1-u_s(0)}  \, ds \\
&\qquad \,+\, \int_0^t  \frac{1}{B} \, [\, u_s(1) -\beta\, ]\,
\log \frac{u_s(1)}{1-u_s(1)}  \, ds \;=\;
\int_0^1 F_0(\gamma)\, dx \;-\; \int_0^1 F_0(u_t)\, dx
\end{aligned}
\end{equation}
for all $t>0$. Clearly, for each $\varrho>0$, the function
$f_\varrho : (0,1) \to \mathbb R$ defined by
$f_\varrho (r) = [\, r \,-\, \varrho \, ]\, \log [\, r/(1-r)\,]$ is
bounded below by a finite constant, say $-\, c_1(\varrho) <0$. Hence,
$| \, f_\varrho (r) \,| \le f_\varrho (r) + 2 c_1$. Therefore, there
exists a finite constant $C_0 = C_0(A,B, \alpha, \beta)$ such that
\begin{equation*}
\begin{aligned}
& \int_0^t ds \int_0^1 \frac{(\nabla  u_s)^2}{\sigma(u_s)} \, dx
+\; \frac{1}{A}  \int_0^t  \,\Big|\,  [\, u_s(0) -\alpha\, ]\,
\log \frac{u_s(0)}{1-u_s(0)} \,\Big| \, ds \\
&\qquad \,+\, \frac{1}{B} \int_0^t
\,\Big| \, [\, u_s(1) -\beta\, ]\,
\log \frac{u_s(1)}{1-u_s(1)}  \,\Big| \, ds \;\le \;
C_0\, t \;+\;  \int_0^1 F_0(\gamma)\, dx \;-\; \int_0^1 F_0(u_t)\, dx
\end{aligned}
\end{equation*}
for all $t>0$, as claimed.
\end{proof}

As $u_t$ converges to $\gamma$ in $\ms L^2([0,1])$, letting
$t\to 0$ in the previous lemma yields that
\begin{equation}
\label{n06}
\begin{aligned}
\lim_{t\to 0} \Big\{\,
\int_0^t ds \int_0^1 \frac{(\nabla  u_s)^2}{\sigma(u_s)} \, dx
& +\; \frac{1}{A}  \int_0^t  \,\Big|\,  [\, u_s(0) -\alpha\, ]\,
\log \frac{u_s(0)}{1-u_s(0)} \,\Big| \, ds \\
&\qquad \,+\, \frac{1}{B} \int_0^t
\,\Big| \, [\, u_s(1) -\beta\, ]\,
\log \frac{u_s(1)}{1-u_s(1)}  \,\Big| \, ds \,\Big\} \;= \; 0\;.
\end{aligned}
\end{equation}

\begin{definition}
\label{d03}
Fix $H\in C^{0,1}([0,T]\times [0,1])$.  A function $u$ in
$\ms L^2(0,T; \mc H^1)$ such that $0\le u\le 1$ a.e.\! is said to be a
generalized solution in the cylinder $[0,T]\times [0,1]$ of the
equation \eqref{5-01} if
\begin{equation}
\label{8-01}
\begin{aligned}
& \int_0^1  u_t \, G_t \; dx \;-\; \int_0^1 \gamma \, G_0 \; dx
\;-\; \int_0^t ds \, \int_0^1  u_s \, \partial_s G_s \; dx \\
&\quad \;=\;
\int_0^t ds \, \int_0^1 \big\{\,
-\, \nabla  u_s \, \nabla  G_s \, +\, 2\, \sigma (u_s)\,
\nabla  H_s \,\nabla  G_s \,\big\} \; dx \\
&\quad  \;+\; \int_0^t \big\{\,  \mf p_{\beta,B} ( u_s(1), H_s(1))
\, G_s(1) \,+\, \mf p_{\alpha,A} ( u_s(0) , H_s(0)) \, G_s(0)  \,\big\} \; ds
\end{aligned}
\end{equation}
for every $0<t\le T$ and function $G$ in $C^{1,2}([0,T]\times [0,1])$.
\end{definition}

\begin{theorem}
\label{mt5}
Fix $\gamma: [0,1]\to [0,1]$ and $H$ in $C^{0,1}([0,T]\times
[0,1])$. There exists a unique weak solution of \eqref{5-01}.
\end{theorem}

\begin{proof}
Existence follows from the hydrodynamic limit of the WASEP. Uniqueness
is based on the energy estimate. Fix two initial
conditions $\gamma^{(1)}$,  $\gamma^{(2)}$, and denote by  $u^{(1)}$,
$u^{(2)}$ two weak solutions of \eqref{5-01} with initial conditions
$\gamma^{(1)}$,  $\gamma^{(2)}$, respectively.

Before presenting a rigorous argument we provide an heuristic one.
Approximate $w = u^{(2)} - u^{(1)}$ by a sequence of functions $G$ in
$C^{1,2}([0,T]\times [0,1])$. By \eqref{8-01}, for all $0<t\le T$,
\begin{equation}
\label{8-03}
\begin{aligned}
& \frac{1}{2}\, \int_0^1  w^2_t \; dx
\;-\; \frac{1}{2}\, \int_0^1 [\, \gamma^{(2)} - \gamma^{(1)}\,]^2  \; dx
\;+\;
\int_0^t ds \, \int_0^1  (\nabla  w_s)^2
\; dx \\
&\quad \;=\; 2\,
\int_0^t ds \, \int_0^1 [\, \sigma (u^{(2)}_s) - \sigma (u^{(1)}_s) \,]
\, \nabla  H_s \, \nabla  w_s  \; dx  \\
&\quad  \;-\; \int_0^t \Big\{\, w_s(1)^2\,
\overline{\mf p}_{\beta, B}(H_s(1)) \,+\,
w_s(0)^2\, \overline{\mf p}_{\alpha, A}(H_s(0)) \, \Big\} \; ds\;,
\end{aligned}
\end{equation}
where
$\color{bblue} \overline{\mf p}_{\varrho, D} (M) \,=\, D^{-1}\, \{\,
\varrho \, e^{M} + (1-\varrho) e^{-M} \,\}$.  As
$\overline{\mf p}_{\varrho, D}(M) \ge 0$, the last integral is
negative. Therefore, by Young's inequality
$2xy \le a x^2 + a^{-1} x^2$, and since $\nabla H$ is uniformly
bounded and $\sigma$ Lipschitz continuous,
\begin{equation*}
\begin{aligned}
& \frac{1}{2}\, \int_0^1  w^2_t \; dx
\;-\; \frac{1}{2}\, \int_0^1 [\, \gamma^{(2)} - \gamma^{(1)}\,]^2  \; dx
\;+\; \frac{1}{2}
\int_0^t ds \, \int_0^1  (\nabla  w_s)^2
\; dx \\
&\quad \;\le \; C_0(H)\,
\int_0^t ds \, \int_0^1 w^2_s  \; dx  \;.
\end{aligned}
\end{equation*}
for some finite constant $C_0(H)$ which depends on $H$.
It remains to apply Gronwal inequality to conclude that
\begin{equation*}
\int_0^1  w^2_t \; dx \;\le\;
e^{C_0t}\,
\int_0^1 [\, \gamma^{(2)} - \gamma^{(1)}\,]^2  \; dx \;,
\end{equation*}
which yields uniqueness.

We turn to a rigorous proof.  Recall the notation introduced in the
proof of Lemma \ref{l03}.  Fix a smooth function
$F:\mathbb R \to \mathbb R$ and recall that $w = u^{(2)} -
u^{(1)}$. As $w^{\epsilon,\delta}$ is a smooth function, for
$0<t\le T$,
\begin{equation*}
\<\, F(w^{\epsilon,\delta}_t)\,\> \,-\,
\<\, F(w^{\epsilon,\delta}_0)\,\> \,=\,
\int_0^t ds \int_0^1 F'(w^{\epsilon,\delta}_s)\,
\partial_s w^{\epsilon,\delta}_s \, dx\;.
\end{equation*}
Integrating by parts, the right-hand side becomes
\begin{equation*}
\int_0^1  w^{\epsilon,\delta}_t\, F' (w^{\epsilon,\delta}_t)\; dx \,-\,
\int_0^1  w^{\epsilon,\delta}_0\, F' (w^{\epsilon,\delta}_0)\; dx \,-\,
\int_0^t ds \int_0^1 w^{\epsilon,\delta}_s\,
\partial_s F'(w^{\epsilon,\delta}_s)\, dx\;.
\end{equation*}
By Lemma \ref{l03}, actually its proof since we changed the definition
of $w^\epsilon$, this expression is equal to
\begin{equation}
\label{n07}
\int_0^1  w_t\, F' (w^{\epsilon,\delta}_t)^{\epsilon,\delta}\; dx \,-\,
\int_0^1  w_0\, F' (w^{\epsilon,\delta}_0)^{\epsilon,\delta}\; dx \,-\,
\int_0^t ds \int_0^1 w_s\,
\partial_s [F'(w^{\epsilon,\delta}_s)^{\epsilon,\delta}]\, dx \;+\;
R_{\epsilon,\delta}\;,
\end{equation}
where for all $\epsilon>0$, $\lim_{\delta\to 0} R_{\epsilon,\delta} =0$.

Take $F(a) = (1/2) \, a^2$. Let $\phi^{(2)}$ be the convolution of $\phi$
with itself:
\begin{equation*}
\phi^{(2)} (t) \;=\; \int_{\mathbb R} \phi(t-s) \, \phi(s)\; ds\;,
\end{equation*}
and set
$\color{bblue} \phi_\delta^{(2)} (t) = \delta^{-1}\, \phi^{(2)}
(t/\delta)$. Since $P^{(R)}_t$ is a semigroup and since $P^{(R)}_t$
comutes with the time convolution, for any function
$f\in \ms L^2([0,T]\times [0,1])$,
\begin{equation*}
(f^{\epsilon,\delta})^{\epsilon,\delta} (t,x) \;=\;
\int_{\mathbb R} [\, P^{(R)}_{2\epsilon } f(t+s)\,](x)
\, \phi_\delta^{(2)} (s) \; ds  \;.
\end{equation*}
Therefore, the first three terms of \eqref{n07} are equal to
\begin{equation*}
\int_0^1  w_t\, w^{2\epsilon,\delta}_t\; dx \,-\,
\int_0^1  w_0\, w^{2\epsilon,\delta}_0\; dx \,-\,
\int_0^t ds \int_0^1 w_s\,
\partial_s w^{2\epsilon,\delta}_s\, dx \;,
\end{equation*}
with the convention, starting from this equation and up to the end of
the proof, that the superscript $\delta$ represent now convolution
with $\phi_\delta^{(2)}$ instead of $\phi_\delta$.

By \eqref{8-01}, this sum is equal to
\begin{equation*}
\begin{aligned}
& \int_0^t ds \, \int_0^1 \big\{\,
-\, \nabla  w_s \, \nabla  w^{2\epsilon,\delta}_s \,
+\, 2\, \{\, \sigma (u^{(2)}_s) - \sigma (u^{(1)}_s) \,\}
\, \nabla  H_s
\,\nabla  w^{2\epsilon,\delta}_s \,\big\} \; dx \\
&\quad  \;-\; \int_0^t \big\{\,
\overline{\mf p}_{\beta,B} (H_s(1))
\, w_s(1)\, w^{2\epsilon,\delta}_s(1) \,+\,
\overline{\ \mf p} _{\alpha,A} (H_s(0)) \,
w_s(0) \, w^{2\epsilon,\delta}_s(0) \  \,\big\} \; ds\;,
\end{aligned}
\end{equation*}
where $\overline{\ \mf p} _{\varrho, D} (M)$ has been introduced in
\eqref{8-03}.  By \eqref{6-17}, \eqref{6-14}, for each $\epsilon >0$,
$\nabla w^{2\epsilon}$ belongs to $\ms L^2([0,T]\times
[0,1])$. Therefore, as $\delta\to 0$,
$\nabla w^{2\epsilon,\delta} = (\nabla w^{2\epsilon})^{\delta} \to
\nabla w^{2\epsilon}$ in $\ms L^2([0,T]\times [0,1])$.

On the other hand, by \eqref{6-05} and \eqref{6-14},
\begin{equation*}
|\, w^{2\epsilon, \delta}_t(1) \,-\, w^{2\epsilon}_t(1)\,|^2 \;\le\; C_0 \,
\Vert\, w^{2\epsilon, \delta}_t \,-\, w^{2\epsilon}_t \,\Vert^2_{\mc
H^1}
\end{equation*}
for some finite constant $C_0$ independent of $\epsilon$ and $t$. A
similar inequality holds at $x=0$.  Therefore, as
$\nabla  w^{2\epsilon,\delta} \to \nabla  w^{2\epsilon}$ in
$\ms L^2([0,T]\times [0,1])$ as $\delta\to 0$,
$ w^{2\epsilon, \delta}_t(1) \to w^{2\epsilon}_t(1)$ in $\ms L^2([0,T])$
as $\delta\to 0$. In conclusion, letting $\delta\to 0$, the sum
appearing in the penultimate displayed equation converges to
\begin{equation*}
\begin{aligned}
& \int_0^t ds \, \int_0^1 \big\{\,
-\, \nabla  w_s \, \nabla  w^{2\epsilon}_s \,
+\, 2\, \{\, \sigma (u^{(2)}_s) - \sigma (u^{(1)}_s) \,\}
\, \nabla  H_s
\,\nabla  w^{2\epsilon}_s \,\big\} \; dx \\
&\quad  \;- \; \int_0^t \big\{\,
\overline{\mf p}_{\beta,B} (H_s(1))
\, w_s(1)  \, w^{2\epsilon}_s(1) \,+\,
\overline{\ \mf p} _{\alpha,A} (H_s(0)) \,
w_s(0) \,  w^{2\epsilon}_s(0) \  \,\big\} \; ds\;.
\end{aligned}
\end{equation*}

By the first assertion of Lemma \ref{l02}, as $\epsilon\to 0$,
$\nabla  w^{2\epsilon}$ converges to $\nabla  w$ in
$\ms L^2([0,T] \times [0,1])$. Therefore, as $\epsilon \to 0$, the first
line converges to
\begin{equation*}
\int_0^t ds \, \int_0^1 \big\{\,
-\, \nabla  w_s \, \nabla  w_s \,
+\, 2\, [\, \sigma (u^{(2)}_s) - \sigma (u^{(1)}_s) \,]
\, \nabla  H_s \,\nabla  w_s \,\big\} \; dx \;.
\end{equation*}

On the other hand, as $w\in \ms L^2(0,T; \mc H^1)$, by Lemma \ref{l08},
$w^{\epsilon}_t(1) \to w_t(1)$ in $\ms L^2([0,T])$. Hence, by the
dominated convergence theorem, as $\epsilon \to
0$, the second line converges to
\begin{equation*}
-\, \int_0^t \big\{\,
\overline{\mf p}_{\beta,B} (H_s(1))
\, w_s(1)^2 \,+\,
\overline{\ \mf p} _{\alpha,A} (H_s(0)) \,
w_s(0)^2 \  \,\big\} \; ds\;.
\end{equation*}
This proves that equation \eqref{8-03} is in force and completes the
proof of the theorem.
\end{proof}

We conclude this section with a heat equation with mixed boundary
equations. Fix a function $\phi\in \ms L^2([0,1])$, and consider the
initial-boundary problem

\begin{equation}
\label{n09}
\begin{cases}
\partial_tu \,=\, \Delta u\\
(\Delta  u) (t,0) \,=\, A^{-1} \,\nabla u(t,0) \\
(\Delta  u)(t,1) \,=\, -\, B^{-1}\, \nabla u(t,1) \\
u(0,\cdot) \,=\, \phi (\cdot)\;.
\end{cases}
\end{equation}

One can define generalized solutions of this problem as in Definition
\ref{d01} and prove existence and uniqueness as stated in Theorem
\ref{t01}. The solution can be represented as $u_t = P^{(M)}_t \phi$,
where $(P^{(M)}_t : t\ge 0)$ represents the semigroup associated to
the Laplacian with boundary conditions
\begin{equation*}
(\Delta  f ) (0) \,=\, A^{-1} \, (\nabla f) (0) \;,\qquad
(\Delta  f)(1) \,=\, -\, B^{-1} (\nabla f) (1) \;.
\end{equation*}
Denote this operator by $\Delta_M$.  An elementary computation shows
that the eigenvalues of $\Delta_M$ coincide with those of $\Delta_R$.

We claim that for all $s\ge 0$ and function $f$ in $C^1([0,1])$,
\begin{equation}
\label{6-23b}
\nabla P^{(R)}_{s} f \;=\; P^{(M)}_{s} \nabla f\;.
\end{equation}
To check this identity, fix $f$ in $C^1([0,1])$, and let
$u_s := P^{(R)}_{s} f$. Clearly $u_s$ is the solution of \eqref{1-06h}
with initial condition $u_0=f$. Let $v_s := \nabla u_s$, Then, $v_s$
solves \eqref{n09} initial condition $v_0=\nabla f$. Hence, $v_s$ can
be represented as $v_s = P^{(M)}_{s} \nabla f$, that is,
$P^{(M)}_{s} \nabla f = v_s = \nabla u_s = \nabla P^{(R)}_{s} f$, as
claimed.

\noindent \textbf{Acknowledgment.} T.\ F.\ has been partially
supported by the CNPq Bolsa de Produtividade em Pesquisa PQ
301269/2018-1. P.G. thanks FCT/Portugal for financial support through
the project UID/MAT/04459/2013. This project has received funding from
the European Research Council (ERC) under the European Union's Horizon
2020 research and innovative programme (grant agreement n. 715734).
C.\ L.\ has been partially supported by FAPERJ CNE E-26/201.207/2014,
by CNPq Bolsa de Produtividade em Pesquisa PQ 303538/2014-7, by
ANR-15-CE40-0020-01 LSD of the French National Research Agency.
A. N. has been partially supported by the CNPq Bolsa de Produtividade
em Pesquisa PQ 310589/2018-5.


\begin{thebibliography}{}

\bibitem{BMNS2017} R. Baldasso, O. Menezes, A.  Neumann, R. R. Souza:
Exclusion process with slow boundary. J. Stat. Phys. {\bf 167},
1112--1142 (2017).


\bibitem{BDGJL} L. Bertini, A. De Sole, D. Gabrielli, G.
Jona--Lasinio, C. Landim: Macroscopic fluctuation theory for
stationary non equilibrium state. J. Statist. Phys. {\bf 107},
635--675 (2002).

\bibitem{BDGJL2003} L. Bertini, A. De Sole, D. Gabrielli,
G. Jona-Lasinio and C. Landim: \textit{Large Deviations for the
Boundary Driven Symmetric Simple Exclusion Process}. Mathematical
Physics, Analysis and Geometry 6: 231-267 (2003).

\bibitem{BDGJL2011} L. Bertini, A. De Sole, D. Gabrielli,
G. Jona-Lasinio, C.  Landim; Action functional and quasi-potential for
the Burgers equation in a bounded interval, Commun. Pure Appl.
Math. {\bf 64} 649--696 (2011).

\bibitem{BDGJL2015} L. Bertini, A. De Sole, D. Gabrielli,
G. Jona-Lasinio, C.  Landim; Macroscopic Fluctuation
Theory. Rev. Modern Phys. {\bf 87}, 593--636 (2015)

\bibitem{BGL2009} L. Bertini, D. Gabrielli, C. Landim; Strong
asymmetric limit of the quasi-potential of the boundary driven weakly
asymmetric exclusion process. Commun. Math. Phys. {\bf 289}, 311--334
(2009).

\bibitem{BLM09} L Bertini, C Landim, M Mourragui: Dynamical large
deviations for the boundary driven weakly asymmetric exclusion
process.  Ann. Probab. {\bf 37}, 2357--2403 (2009)

\bibitem{BG04} Th. Bodineau, G. Giacomin: From dynamic to static large
deviations in boundary driven exclusion particle systems. Stoch.
Processes Appl. {\bf 110}, 67--81 (2004)

\bibitem{BEL21} A. Bouley, C. Erignoux, C. Landim: Steady state large
deviations for one-dimensional, symmetric exclusion processes in weak
contact with reservoirs. arXiv:2107.06606 (2021)

\bibitem{D} Private discussions with B. Derrida at the Oberwolfach
Workshop 1648, November 2016.

\bibitem{D2007} B. Derrida: Non-equilibrium steady states:
fluctuations and large deviations of the density and of the current
J. Stat. Mech. P07023 (2007).

\bibitem{DEHP93} B. Derrida, M. R. Evans, V. Hakim, and V. Pasquier,
Exact solution of a 1D asymmetric exclusion model using a matrix
formulation, J. Phys. A {\bf 26} 1493--1517 (1993).

\bibitem{DHS2021} B. Derrida, O. Hirschberg, T Sadhu: Large
deviations in the symmetric simple exclusion process with slow
boundaries. J. Stat. Phys. (2021) 

\bibitem{DLS2003} B. Derrida, J. L.  Lebowitz, E. R. Speer: Large
deviation of the density profile in the steady state of the open
symmetric simple exclusion process, J. Statist. Phys. {\bf 107},
599--634 (2002).

\bibitem{ED2004} C. Enaud, B. Derrida: Large deviation functional of
the weakly asymmetric exclusion process. J. Statist. Phys. {\bf 114},
537–562 (2004).

\bibitem{F2009} J. Farfan: Static large deviations of boundary driven
exclusion processes. arXiv:0908.1798, 2009

\bibitem{FLM2011} J. Farfan, C. Landim, M. Mourragui: Hydrostatics and
dynamical large deviations of boundary driven gradient symmetric
exclusion processes.  Stochastic Process. Appl. {\bf 121}, 725--758
(2011).

\bibitem{FLT2018} J. Farf\'an, C. Landim, K. Tsunoda; Static large
deviations for a reaction-diffusion model.
Probab. Th. Rel. Fields. {\bf 174}, 49–101 (2019)

\bibitem{FW98} M. I. Freidlin, A. D. Wentzell: Random perturbations of
dynamical systems. Second edition. Grundlehren der Mathematischen
Wissenschaften [Fundamental Principles of Mathematical Sciences], {\bf
260}. Springer-Verlag, New York, 1998.


\bibitem{FGN_LPP} T. Franco, P. Gon\c calves, A. Neumann: Large
Deviations for the SSEP with slow boundary: the non-critical case,
submitted and online at arxiv.org.


\bibitem{j20} R. L. Jack: Ergodicity and large deviations in physical
systems with stochastic dynamics. Eur. Phys. J. B  {\bf 93}: 74 (2020)

\bibitem{KL} C. Kipnis and C. Landim: \textit{Scaling limits of
interacting particle systems}. Grundlehren der Mathematischen
Wissenschaften [Fundamental Principles of Mathematical Sciences],
Volume 320.  Springer-Verlag, Berlin (1999).

\bibitem{kov} Kipnis C., Olla S., Varadhan S.R.S., \emph{Hydrodynamics
and large deviations for simple exclusion processes.}  Commun.  Pure
Appl. Math. {\bf 42}, 115--137 (1989).

\bibitem{l92} C. Landim: Occupation time large deviations for the
symmetric simple exclusion process.  Ann. Probab. {\bf 20},
206--231 (1992). 

\bibitem{LT2018} C. Landim and K. Tsunoda: Hydrostatics and dynamical
large deviations for a reaction-diffusion model.
Ann. Inst. H. Poincar\' e Probab. Statist. {\bf 54}, 51--74
(2018).

\bibitem{M83} V. P. Mikha\"ilov: {\it Partial Differential Equations},
2nd edition, Nauka, Moscow, 1983.

\bibitem{PW84} M. H. Protter and H. F. Weinberger: {\it
Maximum Principles in Differential Equations}
Springer--Verlag, New York, 1984

\bibitem{qrv} Quastel, J., Rezakhanlou, F., Varadhan, S. R. S.,
\emph{Large deviations for the symmetric simple exclusion process in
dimensions {$d\geq 3$}}, Probab. Th. Rel. Fields {\bf 113}, 1--84,
(1999).

\bibitem{S08} W. A. Strauss: {\it Partial Differential Equations: An
Introduction} John Wiley \& Sons, Inc., ISBN 13 978-0470-05456-7, 2008

\end{thebibliography}
\end{document}